\documentclass[11pt]{amsart}
\usepackage{amsmath,amsxtra,amssymb,latexsym,epsfig,amscd,amsthm,fancybox,epsfig}
\usepackage[mathscr]{eucal}
\usepackage{graphicx}
\usepackage{multicol,xcolor}
\usepackage{epsfig} % for postscript graphics files
\usepackage{epstopdf}
\usepackage{cases}
\usepackage{subfig}
\usepackage{color}
\usepackage{hyperref}
\usepackage{bigints}
\usepackage{natbib}
\newtheorem{thm}{Theorem}[section]
\newtheorem {asp}{Assumption}[section]

\newtheorem{rmk}{Remark}[section]
\newtheorem{pron}{Proposition}[section]
\newtheorem{deff}{Definition}[section]

\newtheorem{cor}{Corollary}[section]
\newtheorem{lem}{Lemma}[section]
\newtheorem{prop}{Proposition}[section]
\theoremstyle{definition}

\theoremstyle{remark}

\newtheorem{example}{Example}[section]
\numberwithin{equation}{section}

\newcommand{\eps}{\varepsilon}

\newcommand{\C}{\mathcal{C}}

\newcommand{\F}{\mathcal{F}}

\newcommand{\E}{\mathbb{E}}
\newcommand{\BE}{\mathbf{E}}
\newcommand{\BB}{\mathbf{B}}

\newcommand{\BD}{\mathbf{D}}

\newcommand{\BX}{\mathbf{X}}
\newcommand{\bx}{\mathbf{x}}
\newcommand{\BY}{\mathbf{Y}}
\newcommand{\by}{\mathbf{y}}

\newcommand{\ba}{\mathbf{a}}
\newcommand{\bb}{\mathbf{b}}

\newcommand{\N}{\mathbb{N}}
\newcommand{\PP}{\mathbb{P}}

\newcommand{\R}{\mathbb{R}}

\numberwithin{equation}{section}

\newcommand{\bed}{\begin{displaymath}}
\newcommand{\eed}{\end{displaymath}}
\newcommand{\bea}{\bed\begin{array}{rl}}
\newcommand{\eea}{\end{array}\eed}

\newcommand{\barray}{\begin{array}{ll}}
\newcommand{\earray}{\end{array}}
\newcommand{\diag}{{\rm diag}}
\def\disp{\displaystyle}

\newcommand{\1}{\boldsymbol{1}}
\newcommand{\Cov}{\mathrm{Cov}}

\def\bar{\overline}
\def\hat{\widehat}
\def\a.s{\text{\;a.s.\;}}

\title[Stochastic population growth]{Stochastic population growth in spatially heterogeneous environments: The density-dependent case}

\author[A. Hening]{Alexandru Hening }
\thanks{A. Hening was in part supported by EPSRC grant EP/K034316/1}
\address{Department of Mathematics\\
Tufts University\\
Bromfield-Pearson Hall\\
503 Boston Avenue\\
Medford, MA 02155\\
United States
}
\email{Alexandru.Hening@tufts.edu}
\address{Department of Mathematics \\
 Imperial College London\\
 South Kensington Campus\\
  London, SW7 2AZ\\
 United Kingdom}
 \email{a.hening@imperial.ac.uk}

\author[D. Nguyen]{Dang H. Nguyen }
\thanks{The research of D. Nguyen and G. Yin was supported in part by
 the National Science Foundation.}
\address{Department of Mathematics \\
 Wayne State University\\
 Detroit, MI 48202 \\
 United States}
 \email{dangnh.maths@gmail.com}

 \author[G. Yin]{George Yin}
\address{Department of Mathematics \\
 Wayne State University\\
 Detroit, MI 48202 \\
 United States}
 \email{gyin@math.wayne.edu}

\begin{document}

\begin{abstract}
 This work is devoted to
studying the dynamics of a structured population that is subject to
the combined effects of environmental stochasticity,
competition for resources, spatio-temporal heterogeneity and dispersal.
The population is
spread throughout $n$ patches whose population abundances
are modeled as the solutions of a system of nonlinear stochastic differential equations living on $[0,\infty)^n$.

We prove that $ r$, the stochastic growth rate of the total population in the absence of competition,  determines the long-term behaviour of the population. The parameter $ r$ can be expressed as the Lyapunov exponent of an associated linearized system of stochastic differential equations.

Detailed analysis shows that
 if $ r>0$,
 the population abundances
 converge polynomially fast
 to a unique invariant probability measure on $(0,\infty)^n$,
 while when $ r<0$,
  the population abundances of the patches
  converge almost surely to $0$ exponentially fast.
  This generalizes and extends the results of Evans et al  (2014 J. Math. Biol.) and proves one of their conjectures.

Compared to recent developments, our model incorporates very general density-dependent growth rates and competition terms. Furthermore, we prove that persistence is robust to small, possibly density dependent, perturbations of the growth rates, dispersal matrix and covariance matrix of the environmental noise. We also show that the stochastic growth rate depends continuously on the coefficients.
Our work allows the environmental noise driving our system to be degenerate. This is relevant from a biological point of view since, for example, the environments of the different patches can be
 perfectly correlated. We show how one can adapt the nondegenerate results to the degenerate setting. As an example we fully analyze the two-patch case, $n=2$, and show that the stochastic growth rate is a decreasing function of the dispersion rate. In particular, coupling two sink patches can never yield persistence, in contrast to the results from the non-degenerate setting treated by Evans et al. which show that sometimes coupling by dispersal can make the system persistent.

\end{abstract}
\maketitle
\tableofcontents

\section{Introduction}\label{sec:int}

 The survival of an organism is influenced by both biotic (competition for resources, predator-prey interactions) and abiotic (light, precipitation, availability of resources) factors. Since these factors are space-time dependent, all types of organisms have to choose their dispersal strategies: If they disperse they can arrive in locations with different environmental conditions while if they do not disperse they face the temporal fluctuations of the local environmental conditions. The dispersion strategy impacts key attributes of a population including its spatial distribution and temporal fluctuations in its abundance. Individuals selecting more favorable habitats are more likely to survive or reproduce. When population densities increase in these habitats, organisms may prosper by selecting habitats that were previously unused.
There have been numerous studies of the interplay between dispersal and environmental heterogeneity and how this influences population growth; see \cite{H83, GH02, S04, RRB05, S10, CCL12, DR12, ERSS13} and references therein.
The mathematical analysis for stochastic models with density-dependent feedbacks is less explored. In the setting of discrete-space discrete-time models there have been thorough studies by \cite{BS09, S10, SBA11}. Continuous-space discrete-time population models that disperse and experience uncorrelated, environmental stochasticity have been studied by \cite{HTW88, HTW88b, HTW90}. They show that the leading Lyapunov exponent $ r$ of the linearization of the system around the extinction state almost determines the persistence and extinction of the population. For continuous-space continuous-time population models \cite{MS04} study the dynamics of random Kolmogorov type PDE models in bounded domains. Once again, it is shown that the leading Lyapunov exponent $ r$ of the linarization around the trivial equilibrium $0$ almost determines when the population goes extinct and when it persists. In the current paper we explore the question of persistence and extinction when the population dynamics is given by a system of stochastic differential equations. In our setting, even though our methods and techniques are very different from those used by \cite{HTW88,MS04}, we still make use of the system linearized around the extinction state. The Lyapunov exponent of this linearized system plays a key role throughout our arguments.

\cite{ERSS13} studied a linear stochastic model that describes the dynamics of populations that continuously experience uncertainty in time and space. Their work has
shed some light on key issues from population biology. Their results
provide fundamental insights into ``ideal free'' movement in the face of uncertainty, the evolution of dispersal rates, the single
large or several small (SLOSS) debate in conservation biology, and the persistence of coupled sink populations.
In this paper,
we propose a
density-dependent model of stochastic population growth
that captures the interactions between dispersal and environmental heterogeneity and complements the work of
\cite{ERSS13}. We then present a rigorous and comprehensive study of the proposed model based on stochastic analysis.

 The dynamics of a population in nature is stochastic. This is due to \textit{environmental stochasticity} - the fluctuations of the environment make the growth rates random. One of the simplest models for a population living in a single patch is
 \begin{equation}\label{e:U2}
 d U(t) = U(t)(a-b  U(t))dt + \sigma  U(t)dW(t), t\geq 0,
 \end{equation}
where $U(t)$ is the population abundance at time $t$, $a$ is the mean per-capita growth rate, $b>0$ is the strength of intraspecific competition, $\sigma^2$ is the infinitesimal variance of fluctuations in the per-capita growth rate and $(W(t))_{t\geq 0}$ is a standard Brownian motion. The long-term behavior of \eqref{e:U2} is determined by the \textit{stochastic growth rate} $a-\frac{\sigma^2}{2}$ in the following way
 (see \cite{EHS15, DP84}):
\begin{itemize}
\item If $a-\frac{\sigma^2}{2}>0$ and $ U(0)=u>0$, then $(U(t))_{t\geq 0}$ converges weakly to its unique invariant probability measure $\rho$ on $(0,\infty)$.
\item If $a-\frac{\sigma^2}{2}<0$ and $ U(0)=u>0$, then $\lim_{t\to \infty} U(t)=0$ almost surely.
\item If $a-\frac{\sigma^2}{2}=0$ and $ U(0)=u>0$, then $\liminf_{t\to \infty} U(t)=0$ almost surely, $\limsup_{t\to \infty} U(t)=\infty$ almost surely, and $\lim_{t\to \infty}\frac{1}{t}\int_0^tU(s)\,ds=0$ almost surely.
\end{itemize}

Organisms are always affected by temporal heterogeneities, but they are subject to spatial heterogeneities only when they disperse. Population growth is influenced by spatial heterogeneity through the way organisms respond to environmental signals (see \cite{H83, CC91, C00, SLS09}). There have been several analytic studies that contributed to a better understanding of the separate effects of spatial and temporal heterogeneities on population dynamics. However, few theoretical studies  have considered the combined effects of spatio-temporal heterogeneities, dispersal, and density-dependence for discretely structured populations with continuous-time dynamics.

As seen in both the continuous (\cite{ERSS13}) and the discrete (\cite{PL98}) settings, the extinction risk of a population is greatly affected by the spatio-temporal correlation between the environment in the different patches.
For example, if spatial correlations are weak, one can show that populations coupled via dispersal can survive even though every patch, on its own, would go extinct (see \cite{ERSS13, JY98, HQ89}).
Various species usually exhibit spatial synchrony. Ecologists are interested in this pattern as it can lead to the extinction of rare species. Possible causes for synchrony are dispersal and spatial correlations in the environment (see \cite{L93, K00, LKB04}). Consequently, it makes sense to look at stochastic patch models coupled by dispersion for which the environmental noise of the different patches can be strongly correlated.
We do this by extending the setting of \cite{ERSS13}
   by allowing the environmental noise driving the system
  to be degenerate.

The rest of the  paper is organized as follows. In Section \ref{sec:mod}, we introduce our model for a population living in a patchy environment. It takes into account the dispersal between different patches and density-dependent feedback. The temporal fluctuations of the environmental conditions of the various patches are modeled by Brownian motions that are correlated.
We start by considering the relative abundances of the different patches in a low density approximation. We show that these relative abundances converge in distribution to their unique invariant probability measure
asymptotically as  time goes to infinity. Using this invariant probability measure we derive an expression for $ r$, the stochastic growth rate (Lyapunov exponent) in the absence of competition. We show that this $ r$ is key in analyzing the long-term behavior of the populations. In Appendix \ref{sec:+} we show that if $ r>0$ then the abundances converge weakly, polynomially fast, to their unique invariant probability measure on $(0,\infty)^n$. In Appendix \ref{sec:general-}, we show that if $ r<0$ then all the population abundances go extinct asymptotically, at an exponential rate
(with exponential constant $ r$). Appendix \ref{sec:degenerate} is dedicated to the case when the noise driving our system is degenerate (that is, the dimension of the noise is lower than the number of patches). In Appendix \ref{s:robust}, we show that $ r$ depends continuously on the coefficients of our model and that persistence is robust - that is, small perturbations of the model do not make a persistent system become extinct. We provide some numerical examples and possible generalizations in Section \ref{s:discussion}.

\section{Model and Results}\label{sec:mod}

 We study a population with overlapping generations, which live in a spatio-temporally heterogeneous environment
 consisting of $n$ distinct patches. The growth rate of each patch is determined by both deterministic and stochastic environmental inputs. We denote by $X_i(t)$ the population abundance at time $t\geq 0$ of the $i$th patch and write $\BX(t)=(X_1(t),\dots,X_n(t))$ for the vector of population abundances. Following \cite{ERSS13}, it is appropriate to model $\BX(t)$ as a Markov process with the following properties when $0\leq \Delta t\ll 1$:
\begin{itemize}
\item the conditional mean is
\[\E \left[X_i(t+\Delta t) -X_i(t)~|~X_i(t)=x_i\right]\approx \left[a_ix_i - x_ib_i(x_i) + \sum_{j\neq i} \left(x_jD_{ji}-x_iD_{ij}\right)\right]\Delta t,
\]
where $a_i\in\R$ is the per-capita growth rate in the $i$th patch, $b_i(x_i)$ is the per-capita strength of intraspecific competition in patch $i$ when the abundance of the patch is $x_i$, and $D_{ij}\geq 0$ is the dispersal rate from patch $i$ to patch $j$;
\item the conditional covariance is
\[
\Cov\left[X_i(t+\Delta t) - X_i(t), X_j(t+\Delta t) - X_j(t)~|~\BX= \mathbf{x} \right]\approx\sigma_{ij}x_ix_j\Delta t
\]
for some covariance matrix $\Sigma=(\sigma_{ij})$.
\end{itemize}
The difference between our model and the one from \cite{ERSS13} is that we added density-dependent feedback through the $x_ib_i(x_i)$ terms.

We work on a complete probability space $(\Omega,\F,\{\F_t\}_{t\geq0},\PP)$ with filtration $\{\F_t\}_{t\geq 0}$ satisfying the usual
conditions. We consider the system
\begin{equation}\label{e4.0}
d X_i(t)=\left(X_i(t)\left(a_i-b_i(X_i(t))\right)+\sum_{j=1}^n D_{ji}X_j(t)\right)dt+X_i(t)dE_i(t), \, i=1,\dots,n,
\end{equation}
where $D_{ij}\geq0$ for $j\ne i$ is the per-capita rate at which the population in patch $i$ disperses to patch $j$, $D_{ii}=-\sum_{j\ne i} D_{ij}$ is the total per-capita immigration rate out of patch $i$, $\BE(t)=(E_1(t),\dots, E_n(t))^T=\Gamma^\top\BB(t)$,  $\Gamma$ is a $n\times n$ matrix such that
$\Gamma^\top\Gamma=\Sigma=(\sigma_{ij})_{n\times n}$
and $\BB(t)=(B_1(t),\dots, B_n(t))$ is a vector of independent standard Brownian motions adapted to the filtration $\{\F_t\}_{t\geq 0}$.
Throughout the paper, we
work with the following assumption regarding the growth of the instraspecific 
competition rates.

\begin{asp}\label{a:competition} {\rm
For each $i=1,\dots,n$ the function $b_i:\R_+\mapsto\R$ is locally Lipschitz and vanishing at $0$. Furthermore, there are $M_b>0$, $\gamma_b>0$ such that
\begin{equation}\label{e:b}
\dfrac{\sum_{i=1}^n x_i(b_i(x_i)-a_i)}{\sum_{i=1}^n x_i}>\gamma_b\text{ for any } x_i\geq0, i=1,\dots,n \text{ satisfying } \sum_{i=1}^n x_i\geq M_b
\end{equation} }
\end{asp}

\begin{rmk} {\rm
Note that if we set $x_j=x\geq M_b$ and $x_i=0, i\neq j$, we get from \eqref{e:b} that
\[
b_j(x) - a_j > \gamma_b, \,x\geq M_b, j=1,\dots,n.
\]
}\end{rmk}

\begin{rmk}
{\rm Note that condition \eqref{e:b} is biologically reasonable because it holds if the $b_i$'s are sufficiently
large for large $x_i$'s. We
provide some simple scenarios
when Assumption \ref{a:competition} is satisfied.
\begin{itemize}
\item[a)] Suppose $b_i:[0,\infty)\to [0,\infty), i=1,\dots, n$ are locally Lipschitz and vanishing at $0$. Assume that there exist $\gamma_b>0, \tilde M_b>0$ such that
\[
\inf_{x\in  [\tilde M_b,\infty)} b_i(x) - a_i-\gamma_b>0,~ i =1,\dots,n
\]
It is easy to show that Assumption \ref{a:competition} holds.

\item[b)] Particular cases of (a) are for example, any $b_i:\R_+\mapsto \R$ that are locally Lipschitz, vanishing at $0$
such that $\lim_{x\to \infty} b_i(x)=\infty$.

\item[c)] One natural choice for the competition functions, which is widely used throughout the literature,
is $b_i(x)=\kappa_i x, x\in (0,\infty)$ for some $\kappa_i> 0$. In this case the competition terms become $-x_ib(x_i) = - \kappa_i x_i^2$.

\end{itemize}
}\end{rmk}

\begin{rmk}
{\rm
Note that if we have the SDE
\begin{equation}\label{e4.0g}
d  X_i(t)=\left(X_i(t)f_i(X_i(t))+\sum_{j=1}^n D_{ji}X_j(t)\right)dt+X_i(t)dE_i(t), \, i=1,\dots,n
,\end{equation}
where $f_i$ are locally Lipschitz this can always be rewritten in the form \eqref{e4.0} with
\[
a_i := f_i(0) \,\text{ and }\, b_i(x) :=  f_i(0) - f_i(x),\, \,i=1,\dots,n.
\]
Therefore, our setting is in fact
very general and incorporates both nonlinear growth rates and nonlinear competition terms.

The drift $\tilde f(\bx)=(\tilde f_1(\bx),\dots,\tilde f_n(\bx))$ where $\tilde f_i(\bx)=x_i(a_i-b_i(x_i))+\sum_{j=1}^n D_{ji}X_j(t)$
is sometimes said to be \textit{cooperative}. This is
because $f_i(\bx)\leq f_i(\by)$ if $(\bx,\by)\in \R^n_+$ such that $x_i=y_i, x_j\leq y_j$ for $j\ne i$.
A distinctive property of \textit{cooperative systems}
is that comparison arguments are generally satisfied.
We refer to \cite{IG} for more details.}
\end{rmk}

\begin{rmk}
If the dispersal matrix $(D_{ij})$ has a normalized dominant
left eigenvector  $\alpha=(\alpha_1,\dots,\alpha_n)$ then one can show that the system
\[
d X_i(t)=\left(X_i(t)\left(a_i-b_iX_i(t)\right)+\delta \sum_{j=1}^n D_{ji}X_j(t)\right)dt+X_i(t)dE_i(t), \, i=1,\dots,n,
\]
converges as $\delta\to\infty$ to a system $(\tilde X_1(t),\dots,\tilde X_n(t))$ for which
\[
\tilde X_i(t)=\alpha_i \tilde X(t),\, t\geq 0, \, i=1,\dots,n,
\]
where $\tilde X(t)=\tilde X_1(t)+\dots + \tilde X_n(t)$ and $\tilde X$ is an autonomous Markov process that satisfies the SDE
\[
d\tilde X(t) = \tilde X(t) \sum_{i=1}^n \alpha_i(a_i-b_i\alpha_i\tilde X(t))\,dt + \tilde X(t)\sum_{i=1}^n\alpha_i\,dE_i(t).
\]
As such, our system is a general version of the system treated in \cite{EHS15}. One can recover the system from \cite{EHS15} as an infinite dispersion limit of ours.
\end{rmk}

We
denote
by $\BX^{\mathbf{x}}(t)$ the solution of \eqref{e4.0} started at $\BX(0)= \mathbf{x}\in \R^n_+$. Following \cite{ERSS13}, we call
matrices $D$ with zero row sums and non-negative off-diagonal entries \textit{dispersal matrices}. If $D$ is a dispersal matrix,
then it is a generator of
a continuous-time Markov chain.
Define $P_t:=\exp(tD), t\geq 0$. Then $P_t, t\geq 0$ is a matrix with non-negative entries that gives the transition probabilities of a Markov chain: The $(i,j)$th entry of $P_t$ gives the proportion of the population that was initially in patch $i$ at time $0$ but has dispersed to patch $j$ at time $t$ and $D$ is the generator of this Markov chain.
If one wants to include mortality induced
because of dispersal, one can add cemetery patches in which dispersing individuals enter and experience a killing rate before moving to their final destination.
Our model is a density-dependent generalization of the one by \cite{ERSS13}. We are able to prove that the linearization of the density-dependent model fully determines the non-linear density-dependent behavior,
a fact which was conjectured by \cite{ERSS13}. Furthermore, we prove stronger convergence results and thus extend the work of \cite{ERSS13}. Analogous results for discrete-time versions of
the model have been studied by \cite{BS09} for discrete-space and by \cite{HTW88, HTW88b} for continuous-space.

We will work under the following assumptions.

\begin{asp}\label{a:dispersion} {\rm
The dispersal matrix $D$ is \textit{irreducible}.
}\end{asp}
\begin{asp}\label{a:nonsingular} {\rm
The covariance matrix $\Sigma$ is non-singular.
}\end{asp}

Assumption \ref{a:dispersion} is equivalent to forcing the entries of the matrix $P_t=\exp(tD)$ to be strictly positive for all $t>0$. This means that it is possible for the population to disperse between any two patches. We can always reduce our problem to this setting by working with the maximal irreducible subsets of patches.
Assumption \ref{a:nonsingular} says that our randomness is non-degenerate, and thus truly $n$-dimensional. We show in Appendix \ref{sec:degenerate} how to get the desired results when Assumption \ref{a:nonsingular} does not hold.

Throughout the paper we set $\R^n_+:=[0,\infty)^n$ and $\R_+^{n,\circ}:=(0,\infty)^n$. We define the total abundance of our population at time $t\geq 0$ via
$S(t):=\sum_{i=1}^n X_i(t)$ and
let $Y_i(t):=\frac{X_i(t)}{S(t)}$ be the proportion of the total population that is in patch $i$ at time $t\geq0$. Set $\BY(t)=(Y_1(t),\dots, Y_n(t))$.
   An application of It\^o's lemma to \eqref{e4.0} yields
\begin{equation}\label{e4.1}
\begin{split}
dY_i(t)=&Y_i(t)\left(a_i-\sum_{j=1}^na_jY_j(t)-b_i(S(t)Y_i(t))+\sum_{j=1}^nY_j(t)b_j(S(t)Y_j(t))\right)dt+\sum_{j=1}^{n}D_{ji}Y_j(t)dt\\
&+Y_i(t)\left(\sum_{j,k=1}^n\sigma_{kj}Y_k(t)Y_j(t))-\sum_{j=1}^n\sigma_{ij}Y_j(t)\right)dt  +Y_i(t)\left[dE_i(t)-\sum_{j=1}^n Y_j(t)dE_j(t)\right]\\
dS(t)=&S(t)\left(\sum_{i=1}^n(a_iY_i(t)-Y_i(t)b_i(S(t)Y_i(t)))\right)dt+S(t)\sum_{i=1}^nY_i(t)dE_i(t)
\end{split}
\end{equation}
We can rewrite \eqref{e4.1} in the following compact equation for $(\BY(t), S(t))$ where $\bb(\bx)=(b_1(x_1),\dots, b_n(x_n))$.
\begin{equation}\label{eq.bys}
\begin{split}
d\BY(t)=&\left(\diag(\BY(t))-\BY(t)\BY^\top(t)\right)\Gamma^\top d\BB(t)\\
&+\BD^\top\BY(t)dt+\left(\diag(\BY(t))-\BY(t)\BY^\top(t)\right)(\ba-\Sigma \BY(t)-\bb(S(t)\BY(t)))dt\\
dS(t)=&S(t)\left[\ba-{\bb(S(t)\BY(t))}\right]^\top\BY(t)dt+S(t){\BY(t)}^\top \Gamma^\top d\BB(t),
\end{split}
\end{equation}
where $\BY(t)$ lies in the simplex $\Delta:=\{(y_1,\dots,y_n)\in\R^{n}_+: y_1+\dots+y_n=1\}$.
Let $\Delta^{\circ}=\{(y_1,\dots,y_n)\in\R^{n,\circ}_+: y_1+\dots+y_n=1\}$ be the interior of $\Delta$.

Consider equation \eqref{eq.bys} on the boundary $((\by,s): \by\in\Delta, s=0)$ (that is, we set $S(t)\equiv 0$ in the equation for $\BY(t)$). We have
the following system
\begin{equation}\label{eq.by}
\begin{split}
d\tilde\BY(t)=&\left(\diag(\tilde\BY(t))-\tilde\BY(t)\tilde\BY^\top(t)\right)\Gamma^\top d\BB(t)\\
&+\BD^\top\tilde\BY(t)dt+\left(\diag(\tilde\BY(t))-\tilde\BY(t)\tilde\BY^\top(t)\right)(\ba-\Sigma \tilde\BY(t))dt
\end{split}
\end{equation}
on the simplex $\Delta$.
We also introduce the linearized version of \eqref{e4.0}, where the competition terms $b_i(x_i)$ are all set to $0$,
\begin{equation}\label{e1.1c}
d\mathcal X_i(t)=\left(\mathcal X_i(t)a_i+\sum_{j=1}^n D_{ji}\mathcal X_j(t)\right)dt+\mathcal X_i(t)dE_i(t), \,\, i=1,\dots,n.
\end{equation}
and let $\mathcal{S}(t)=\sum_{i=1}^n\mathcal{X}_i(t)$ be the total population abundance, in the absence of competition. The processes $(\mathcal{X}_1(t),\dots,\mathcal{X}_n(t))$, $\tilde\BY(t)$ and $\mathcal{S}(t)$ have been studied by \cite{ERSS13}.

\cite[Proposition 3.1]{ERSS13} proved that the process $(\tilde\BY(t))_{t\geq 0}$ is an irreducible Markov process, which has the strong Feller property and admits a unique invariant probability measure $\nu^*$ on $\Delta$. Let $\tilde \BY(\infty)$ be a random variable on $\Delta$ with distribution $\nu$. We define
\begin{equation}\label{lambda}
 r:=\int_{\Delta}\left(\ba^\top{\bf y}-\frac12{\bf y}^\top\Sigma{\bf y}\right)\nu^*(d{\bf y})= \sum_{i}a_i\E\left[\tilde Y_i(\infty)\right] - \frac{1}{2} \E\left[\sum_{ij}\sigma_{ij}\tilde Y_i(\infty)\tilde Y_j(\infty)\right]
\end{equation}

\begin{rmk}{\rm
We note that $ r$ is the stochastic growth rate (or Lyapunov exponent) of the total population $\mathcal{S}(t)$ in the absence of competition. That is,
\[
\PP\left\{\lim_{t\to \infty}\frac{\ln\mathcal{S}^{\bx}(t)}{t} =  r\right\} =1.
\]
The expression \eqref{lambda} for $ r$ coincides with the one derived by \cite{ERSS13}.
}\end{rmk}

We use
superscripts
to denote the starting points of our processes. For example $(\BY^{\by, s}(t), S^{\by, s}(t))$ denotes the solution of \eqref{e4.1} with $(\BY(0), S(0))=(\by,s)\in \Delta\times (0,\infty)$. Fix  $\bx\in\R^n_+$ and define the \textit{normalized occupation measures},
\begin{equation}\label{e:occupation}
\Pi^{(\bx)}_t(\cdot)=\dfrac1t\int_0^t\1_{\{\BX^{\bx}(u)\in\cdot\}}du.
\end{equation}
These random measures describe the distribution of the observed population dynamics up to time $t$. If we define the sets
\[
S_\eta:=\{\bx=(x_1,\dots,x_n)\in\R^{n,\circ}_+: |x_i|\leq\eta~\text{for some}~i=1,\dots,n\},
\]
then $\Pi^{(\bx)}_t(S_\eta)$ is the fraction of the time in the interval $[0,t]$ that the total abundance of some patch is less than $\eta$ given that our population starts at $\BX(0)=\bx$.

\begin{deff}
One can define a distance on the space of probability measures living on the Borel measurable subsets of $\R_+^n$, that is on the space $(\R_+^n,\mathcal{B}(\R_+^n))$. This is done by defining $\|\cdot,\cdot\|_{\text{TV}}$, the \textit {total variation norm}, via
\[
\|\mu,\nu\|_{\text{TV}} := \sup_{A\in \mathcal{B}(\R_+^n)} |\mu(A)-\nu(A)|.
\]
\end{deff}

\begin{thm}\label{t:survival}
Suppose that Assumptions \ref{a:dispersion} and \ref{a:nonsingular} hold and that $ r>0$.
The process $\BX(t) = (X_1(t),\dots,X_n(t))_{t\geq 0}$ has a unique invariant probability measure $\pi$ on $\R^{n,\circ}_+$ that is absolutely continuous with respect to the Lebesgue measure and for any $q^*>0$,
\begin{equation}
\lim\limits_{t\to\infty} t^{q^*}\|P_\BX(t, \mathbf{x}, \cdot)-\pi(\cdot)\|_{\text{TV}}=0, \;\mathbf{x}\in\R^{n,\circ}_+,
\end{equation}
and $P_\BX(t,\mathbf{x},\cdot)$ is the transition probability of $(\BX(t))_{t\geq 0}$. Moreover, for any initial value $\mathbf{x}\in\R^{n}_+\setminus\{\mathbf{0}\}$ and any $\pi$-integrable function $f$ we have
\begin{equation}\label{slln}
\PP\left\{\lim\limits_{T\to\infty}\dfrac1T\int_0^Tf\left(\BX^{\mathbf{x}}(t)\right)dt=\int_{\R_+^{n,\circ}}f(\mathbf{u})\pi(d\mathbf{u})\right\}=1.
\end{equation}
\end{thm}

\begin{rmk} {\rm
Theorem \ref{t:survival} is a direct consequence of Theorem \ref{thm2.2}, which
will be proved in Appendix \ref{sec:+}. As a corollary we get the following result.
}\end{rmk}

\begin{deff} {\rm
Following \cite{RS14}, we say
that
the model \eqref{e4.0} is stochastically persistent if for all $\eps>0$, there exists $\eta>0$ such that with probability one,
\[
\Pi^{(\bx)}_t(S_\eta)\leq \eps
\]
for $t$ sufficiently large and $\bx\in S_\eta\setminus \{\mathbf{0}\}$.
}\end{deff}

\begin{cor}
If Assumptions \ref{a:dispersion} and \ref{a:nonsingular} hold, and $ r>0$, then the process $\BX(t)$ is stochastically persistent.
\end{cor}

\begin{proof}
By Theorem \ref{t:survival}, we have that for all $\mathbf{x}\in\R^{n,\circ}_+$,
\[
\PP\left\{\Pi^{(\bx)}_t \Rightarrow \pi ~\text{as}~t\to\infty\right\}=1.
\]
Since $\pi$ is supported on $\R^{n,\circ}_+$, we get the desired result.
\end{proof}

\textbf{Biological interpretation of Theorem \ref{t:survival}} \textit{The quantity $ r$ is the Lyapunov exponent or stochastic growth rate of the total population process $(\mathcal{S}(t))_{t\geq 0}$ in the absence of competition. This number describes the long-term growth rate of the population in the presence of a stochastic environment. According to \eqref{lambda} $ r$ can be written as the difference $\bar \mu - \frac{1}{2}\bar\sigma^2$ where}
\begin{itemize}
\item \textit{$\bar \mu$ is the average of per-capita growth rates with respect to the asymptotic distribution $\tilde \BY(\infty)$ of the population in the absence of competition.}
\item \textit{$\bar\sigma^2$ is the infinitesimal variance of the environmental stochasticity averaged according to the asymptotic distribution of the population in the absence of competition.}
\end{itemize}
\textit{We note by \eqref{lambda} that $ r$ depends on the dispersal matrix, the growth rates at $0$ and the covariance matrix of the environmental noise. As such, the stochastic growth rate can change due to the dispersal strategy or environmental fluctuations. }

\textit{When the stochastic growth rate of the population in absence of competition is strictly positive (i.e. $ r>0$) our population is persistent in a strong sense: for any starting point $(X_1(0),\dots,X_n(0)) = (x_1,\dots,x_n)\in \R_+^{n,\circ}$ the distribution of the population densities at time $t$ in the $n$ patches $(X_1(t)$, $\dots$, $X_n(t))$ converges as $t\to \infty$ to the unique probability measure $\pi$ that is supported on $\R_+^{n,\circ}$.}

\begin{deff}{\rm
We say the population of patch $i$ goes extinct if for all $\bx\in\R^{n}_+\setminus\{\mathbf{0}\}$
\[
\PP\left\{\lim_{t\to\infty}X^\bx_i(t)=0\right\}=1.
\]
We say the population goes extinct if the populations from all the patches go extinct, that is if for all $\bx\in\R^{n}_+\setminus\{\mathbf{0}\}$
\[
\PP\left\{\lim_{t\to\infty}\BX^\bx(t)=\mathbf{0}\right\}=1.
\] }
\end{deff}

\begin{thm}\label{t:extinction}
Suppose that Assumptions \ref{a:dispersion} and \ref{a:nonsingular} hold and that $ r<0$.
Then for any $i=1,\dots,n$ and any $\mathbf{x} = (x_1,\dots,x_n)\in \R_+^{n}$,
\begin{equation}
\PP\left\{\lim_{t\to\infty}\frac{\ln {X}_i^{\mathbf{x}}(t)}{t}=  r\right\}=1.
\end{equation}
\end{thm}

\textbf{Biological interpretation of Theorem \ref{t:extinction}} \textit{If the stochastic growth rate of the population in the absence of competition is negative (i.e. $ r<0$) the population densities of the $n$ patches $(X_1(t),\dots,X_n(t))$ go extinct exponentially fast with rates $ r<0$ with probability $1$ for any starting point $(X_1(0),\dots,X_n(0))= (x_1,\dots,x_n)\in \R_+^{n}$.}

In Appendix \ref{sec:+}, we prove Theorem \ref{t:survival} while Theorem \ref{t:extinction} is proven in Appendix \ref{sec:general-}.

\subsection{Degenerate noise}
We
 consider the evolution of the process $(\BX(t))_{t\geq 0}$ given by \eqref{e4.0} when Assumption \ref{a:nonsingular} does not hold. If the covariance matrix $\Sigma=\Gamma^T\Gamma$ coming for the Brownian motions $\BE(t)=(E_1(t),\dots, E_n(t))^T=\Gamma^\top\BB(t)$ is singular, the environmental noise driving our SDEs has a lower dimension than the dimension $n$ of the underlying state space.
 It becomes much more complex to
 prove that our process is Feller and irreducible. In order to verify the Feller property, we have to verify the so-called H{\"o}rmander condition,
  and
  to verify the irreducibility, we have to investigate the controllability of a related
  control system.

We are able to prove the following extinction and persistence results.
  \begin{thm}\label{t:ext_deg}
  Assume that $\tilde\BY(t)$ has a \textbf{unique} invariant probability measure $\nu^*$. Define $ r$ by \eqref{lambda}. Suppose that $ r<0$.
Then for any $i=1,\dots,n$ and any $\mathbf{x} = (x_1,\dots,x_n)\in \R_+^{n}$
\begin{equation}
\PP\left\{\lim_{t\to\infty}\frac{\ln {X}_i^{\mathbf{x}}(t)}{t}=  r\right\}=1.
\end{equation}
In particular, for any $i=1,\dots,n$ and any $\mathbf{x} = (x_1,\dots,x_n)\in \R_+^{n}$
\[
\PP\left\{\lim_{t\to\infty}X^{\mathbf{x}}_i(t)=0\right\}=1.
\]
\end{thm}
\begin{rmk}{\rm The extra assumption in this setting is that the Markov process describing the proportions of the populations of the patches evolving without competition, $\tilde\BY(t)$, has a unique invariant probability measure.
In fact, we conjecture that  $\tilde\BY(t)$ always has a unique invariant probability measure.
We were able to prove this conjecture when $n=2$ -- see Remark \ref{r:lambda_2d} for details. }
\end{rmk}
\begin{thm}\label{t:pers_deg}
Assume that $\tilde\BY(t)$ has a \textbf{unique} invariant probability measure $\nu^*$. Define $ r$ by \eqref{lambda}. Suppose that Assumption \ref{a:dispersion} holds and that $ r>0$.
Assume further that there is a sufficiently large $T>0$ such that the Markov chain $(\BY(kT),S(kT))_{k\in\N}$ it is irreducible and aperiodic, and
that every compact set in $\Delta^\circ\times(0,\infty)$ is petite for this Markov chain.

The process $\BX(t) = (X_1(t),\dots,X_n(t))_{t\geq 0}$ has a unique invariant probability measure $\pi$ on $\R^{n,\circ}_+$ that is absolutely continuous with respect to the Lebesgue measure and for any $q^*>0$,
\begin{equation}
\lim\limits_{t\to\infty} t^{q^*}\|P_\BX(t, \mathbf{x}, \cdot)-\pi(\cdot)\|_{\text{TV}}=0, \;\mathbf{x}\in\R^{n,\circ}_+,
\end{equation}
where $\|\cdot,\cdot\|_{\text{TV}}$ is the total variation norm and $P_\BX(t,\mathbf{x},\cdot)$ is the transition probability of $(\BX(t))_{t\geq 0}$. Moreover, for any initial value $\mathbf{x}\in\R^{n}_+\setminus\{\mathbf{0}\}$ and any $\pi$-integrable function $f$, we have
\begin{equation}
\PP\left\{\lim\limits_{T\to\infty}\dfrac1T\int_0^Tf\left(\BX^{\mathbf{x}}(t)\right)dt=\int_{\R_+^{n,\circ}}f(\mathbf{u})\pi(d\mathbf{u})\right\}=1.
\end{equation}
\end{thm}

\begin{rmk}{\rm We require as before that $\tilde\BY(t)$ has a unique invariant probability measure. Furthermore, we require that there exists some time $T>0$ such that if we observe the process $(\BY(t),S(t))$ at the fixed times $T, 2T,3T,\dots, kT,\dots$ it is irreducible (loosely speaking this means that the process can visit any state) and aperiodic (returns to a given state occur at irregular times).  }
\end{rmk}
\subsection{Case study: $n=2$}

Note that the two Theorems above have some extra assumptions. We exhibit how one can get these conditions explicitly as functions of the various parameters of the model. For the sake of a clean exposition we chose to fully treat the case when $n=2$ and $b_i(x)=b_ix,x\geq0,  i=1,2$ for some $b_1,b_2>0$ (each specific case would have to be studied separately as the computations change in each setting).
As a result, \eqref{e4.0} becomes
\begin{equation*}
\begin{cases}
dX_1(t)=\big(X_1(t)(a_1-b_1 X_1(t))-\alpha X_1(t)+\beta X_2(t)\big)dt+\sigma_1X_1(t)dB(t)  \\
dX_2(t)=\big(X_2(t)(a_2-b_2 X_2(t))+\alpha X_1(t)-\beta X_2(t)\big)dt+\sigma_2X_2(t)dB(t),
\end{cases}
\end{equation*}
where $\sigma_1$, $\sigma_2$ are non-zero constants and $(B(t))_{t\geq 0}$ is a one dimensional Brownian motion.
The Lyapunov exponent can now be expressed as (see Remark \ref{r:lambda_2d})
\begin{equation}\label{e:lambda_ours}
 r = a_2-\frac{\sigma_2^2}{2} + (a_1-a_2+\sigma_2^2)\int_0^1 y\rho_1^*(y)\,dy - \frac{(\sigma_1-\sigma_2)^2}{2}\int_0^1 y^2\rho_1^*(y)\,dy
\end{equation}
where $\rho_1^*$ is given in \eqref{e-rho*} later.

If $\sigma_1=\sigma_2=:\sigma$,  one has (see Remark \ref{r:lambda_2d})
\begin{equation}\label{e:lambda_2}
 r =
a_2-\dfrac{\sigma^2}2+(a_1-a_2+\sigma^2)y^\star.
\end{equation}
\begin{thm}\label{t:extinction_degenerate_2}
Define $ r$ by \eqref{e:lambda_ours} if $\sigma_1\ne\sigma_2$ and by \eqref{e:lambda_2} if $\sigma_1=\sigma_2=\sigma$. If $ r<0$ then for any $i=1,2$ and any $\mathbf{x} = (x_1,x_2)\in \R_+^{2}$
\begin{equation}
\PP\left\{\lim_{t\to\infty}\frac{\ln {X}_i^{\mathbf{x}}(t)}{t}=  r\right\}=1.
\end{equation}
\end{thm}

\begin{thm}\label{t:survival_degenerate_2}
Suppose that $\sigma_1\neq \sigma_2$ or $\beta+(b_2/b_1)( a_1-a_2-\alpha+\beta)-\alpha (b_2/b_1)^2\ne 0$.
 Define $ r$ as in Theorem \ref{t:extinction_degenerate_2}. If $ r>0$ then the conclusion of Theorem \ref{t:pers_deg} holds.
\end{thm}

\begin{rmk}
{\rm
Once again the parameter $ r$ tells us when the population goes extinct and when it persists.
To obtain the conclusion of Theorem \ref{t:pers_deg}  when $ r>0$, we need
$\sigma_1\neq \sigma_2$ or $\beta+(b_2/b_1)( a_1-a_2-\alpha+\beta)-\alpha (b_2/b_1)^2\ne 0.$
The condition $\sigma_1\neq \sigma_2$ tells us that the noise must at least differ through its variance. If $\sigma_1= \sigma_2$ then we require
\[
a_1+\beta \frac{b_1+b_2}{b_2} \neq a_2+\alpha \frac{b_1+b_2}{b_1}.
\]
The term $\beta \frac{b_1+b_2}{b_2}$ measures the dispersion rate of individuals from patch $2$ to patch $1$ averaged by the inverse relative competition strength of patch $2$. In particular, if $b_1=b_2$ we have that
\[
2(\beta-\alpha) \neq a_2 - a_1,
\]
that is twice the difference of the dispersal rates cannot equal the difference of the growth rates.
The dynamics of the system is very different if these conditions do not hold (see Section \ref{s:n=2_deg_noHorm} and Theorem \ref{t:survival_degenerate_1d}).}
\end{rmk}

\begin{thm}\label{t:survival_degenerate_1d}
Suppose that $\sigma_1= \sigma_2=\sigma, b_1=b_2$ and $2(\beta-\alpha) = a_2 - a_1$. In this setting one can show that the stochastic growth rate is given by $ r=a_1-\alpha+\beta-\frac{\sigma^2}2$. Assume that $(X_1(0),X_2(0))=\bx=(x_1,x_2)\in\R_+^{2,\circ}$ and let $U(t)$ be the solution to
\[
dU(t)=U(t)(a_1-\alpha+\beta-b U(t))\,dt+\sigma U(t)dB(t), U(0)=x_2.
\]
Then we get the following results
\begin{itemize}
\item [1)] If $x_1=x_2$ then $\PP(X_1^\bx(t)= X_2^\bx(t)=U(t), t\geq 0)=1.$
\item [2)]If $x_1\neq x_2$ then $\PP(X^\bx_1(t)\neq  X^\bx_2(t),  t\geq 0)=1.$
\item [3)]
If $ r<0$ then $X_1(t)$ and $X_2(t)$ converges to $0$ exponentially fast.
If $ r>0$ then $$\PP\left\{\lim_{t\to\infty}\dfrac{X^\bx_1(t)}{U^\bx(t)}=\lim_{t\to\infty}\dfrac{X^\bx_2(t)}{U^\bx(t)}=1\right\}=1.$$
Thus, both $X_1(t)$ and $X_2(t)$ converge to a unique invariant probability measure $\rho$ on $(0,\infty)$,
which is the invariant probability measure of $U(t)$. The invariant probability measure of $ (X_1(t),X_2(t))_{t\geq 0}$ is concentrated on the one-dimensional manifold
 $\{\bx=(x_1,x_2)\in\R^{2,\circ}_+: x_1=x_2\}$.
\end{itemize}

\end{thm}
The proof of Theorem \ref{t:survival_degenerate_1d} is presented in Section \ref{s:n=2_deg_noHorm}.

\subsection{Robust persistence and extinction}
The model we work with is
an approximation of the real biological models. As a result, it is relevant to see if `close models' behave similarly to ours. This reduces to studying the robustness of our system. Consider the process
\begin{equation}\label{e:hatX}
d\hat X_i=\hat X_i\left(\hat a_i-\hat b_i(X_i)\right)dt +\hat D_{ij}(\hat\BX)\hat X_idt+\hat X_i\hat \Gamma(\hat\BX) d\BB(t)
\end{equation}
where $\hat\bb(\cdot), \hat D(\cdot),\hat\Gamma(\cdot)$ are locally Lipschitz functions
and
$\hat D_{ij}(\bx)\geq0$ for all $\bx\in\R^n_+, i\ne j$ and $\hat D_{ii}(\bx)=-\sum_{j\ne i} D_{ij}(\bx).$
If there exists $\theta>0$ such that
\begin{equation}\label{robust}
\sup\limits_{\bx\in\R^{n,\circ}_+}\left\{\|\ba-\hat \ba\|, \|\bb(\bx)-\hat\bb(\bx)\|, \|D-\hat D(\bx)\|, \|\Gamma-\hat\Gamma(\bx)\|\right\}<\theta,
\end{equation}
then we call $\hat \BX$ a \textit{$\theta$-perturbation} of $\BX$.

\begin{thm}\label{probust2}
Suppose that the dynamics of $(\BX(t))_{t\geq0}$ satisfy the assumptions of Theorem \ref{t:survival}. Then there exists $\theta>0$ such that any $\theta$-perturbation $(\hat\BX(t))_{t\geq0}$  of $(\BX(t))_{t\geq0}$ is persistent.
Moreover, the process $(\hat\BX(t))_{t\geq0}$ has a unique invariant probability measure $\hat\pi$ on $\R^{n,\circ}_+$ that is absolutely continuous with respect to the Lebesgue measure and for any $q^*>0$
\begin{equation*}
\lim\limits_{t\to\infty} t^{q^*}\|P_{\hat\BX}(t, \mathbf{x}, \cdot)-\hat\pi(\cdot)\|_{\text{TV}}=0, \;\mathbf{x}\in\R^{n,\circ}_+,
\end{equation*}
where $P_{\hat\BX}(t, \mathbf{x}, \cdot)$ is the transition probability of $(\hat\BX(t))_{t\geq0}$.
\end{thm}

\textbf{Biological interpretation of Theorem \ref{probust2}} \textit{As long as the perturbation of our model is small, persistence does not change to extinction. Our model, even though it is only an approximation of reality, can provide relevant information regarding biological systems. Small enough changes in the growth rates, the competition rates, the dispersion matrix and the covariance matrix leave a persistent system unchanged.
}

\section{Theoretical and Numerical Examples}
This subsection is devoted to some theoretical and numerical examples. We choose the dimension to be $n=2$, so that we can compute the stochastic growth rate explicitly.

\begin{rmk}\label{r:lambda_2d}
{\rm If
an explicit expression for $ r$ is desirable, one needs to determine the first and second moments for the invariant probability measure $\nu^*$. One can show that $\rho^*$, the density of $\nu^*$ with respect to Lebesgue measure, satisfies
\begin{equation}\label{e:FP}
-\sum_{i}\frac{\partial}{\partial y_i}[ \mu_i(\by)\rho^*(\by)]+\frac{1}{2}\sum_{i,j}\frac{\partial^2}{\partial y_i \partial y_j}[ v_{ij}(\by)\rho^*(\by)]=0, ~\by\in\Delta,
\end{equation}
where $\mu_i(\by)$ and $v_{i,j}(\by)$ are the entries of
\begin{equation*}
\begin{split}
\mu(\by) &= D^\top\by + \left(\diag(\by)-\by\by^\top\right)\left(\ba-\Sigma \by\right),\\
v(\by) &= \left(\diag(\by)-\by\by^\top(t)\right)\Gamma^\top\Gamma \left(\diag(\by)-\by\by^\top(t)\right),
\end{split}
\end{equation*}
and $\rho^*$ is constrained by $\int_\Delta \rho^*(\by)d\by=1$ with appropriate boundary conditions. The boundary conditions are usually found by characterizing the domain of the infinitesimal generator of the Feller diffusion process $\tilde \BY(t)$, which
is usually a very difficult problem.

However, following \cite{ERSS13},
 in the case of two patches ($n=2$) and non-degenerate noise the problem is significantly easier.
Let  $\Sigma=\diag(\sigma_1^2,\sigma_2^2)$.  The system becomes
\begin{equation}\label{n=2-nonde}
\begin{cases}
dX_1(t)=\big(X_1(t)(a_1-b X_1(t))-\alpha X_1(t)+\beta X_2(t)\big)dt+\sigma_1 X_1(t)dB_1(t)  \\
dX_2(t)=\big(X_2(t)(a_2-b X_2(t))+\alpha X_1(t)-\beta tX_2(t)\big)dt+\sigma_2 X_2(t)dB_2(t).
\end{cases}
\end{equation}

It is easy to find the density $\rho_1^*$ of $\tilde Y_1(\infty)$ explicitly (by solving \eqref{e:FP}) and noting that $0,1$ are both entrance boundaries for the diffusion $\tilde Y_1(t)$).
Then
\[
\rho_1^*(x) = Cx^{\beta-\alpha_1}(1-x)^{-\beta-\alpha_2}\exp\left(-\frac{2}{\sigma_1^2+\sigma_2^2}\left(\frac{\beta}{x}+\frac{\alpha}{1-x}\right)\right), \,x\in (0,1)
\]
where $C>0$ is a normalization constant and
\begin{equation*}
\begin{split}
\alpha_i&:= \frac{2\sigma_i^2}{\sigma_1^2+\sigma_2^2}, \,i=1,2\\
\beta&:= \frac{2}{\sigma_1^2+\sigma_2^2}(a_1-a_2+\beta-\alpha).
\end{split}
\end{equation*}
One can then get the following explicit expression for the Lyapunov exponent
\begin{equation}\label{e:lambda_evans}
 r = a_2-\frac{\sigma_2^2}{2} + (a_1-a_2+\sigma_2^2)\int_0^1 y\rho_1^*(y)\,dy - \frac{\sigma_1^2+\sigma_2^2}{2}\int_0^1 y^2\rho_1^*(y)\,dy.
\end{equation}

Next, consider the degenerate case
\begin{equation}\label{e:SDE_twopatch}
\begin{cases}
dX_1(t)=\big(X_1(t)(a_1-b_1 X_1(t))-\alpha X_1(t)+\beta X_2(t)\big)dt+\sigma_1X_1(t)dB(t)  \\
dX_2(t)=\big(X_2(t)(a_2-b_2 X_2(t))+\alpha X_1(t)-\beta X_2(t)\big)dt+\sigma_2X_2(t)dB(t),
\end{cases}
\end{equation}
where $\sigma_1$, $\sigma_2$ are non-zero constants and $(B(t))_{t\geq 0}$ is a one dimensional Brownian motion.
Since $\tilde Y_1(t)+\tilde Y_2(t)=1$,
to find the invariant probability measure of $\tilde\BY(t)$,
we only need to find the invariant probability measure of $\tilde Y_1(t)$.

If $\sigma_2\ne\sigma_2$ we can find the invariant density $\rho_1^*$ of $\tilde Y_1(\infty)$ explicitly (by solving  \eqref{e:FP}.
Then
\begin{equation}\label{e-rho*}
\rho_1^*(x) = Cx^{\hat \beta-\hat \alpha_1}(1-x)^{-\hat\beta-\hat \alpha_2}\exp\left(-\frac{2}{(\sigma_1-\sigma_2)^2}\left(\frac{\beta}{x}+\frac{\alpha}{1-x}\right)\right), \,x\in (0,1)
\end{equation}
where $C>0$ is a normalization constant and
\begin{equation*}
\begin{split}
\hat \alpha_1&:= \frac{-2\sigma_1}{(\sigma_1-\sigma_2)},\quad\hat \alpha_2:= \frac{2\sigma_2}{(\sigma_1-\sigma_2)}, \\
\hat \beta&:= \frac{2}{(\sigma_1-\sigma_2)^2}(a_1-a_2+\beta-\alpha).
\end{split}
\end{equation*}
The Lyapunov exponent can now be expressed as
\begin{equation*}
 r = a_2-\frac{\sigma_2^2}{2} + (a_1-a_2+\sigma_2^2)\int_0^1 y\rho_1^*(y)\,dy - \frac{(\sigma_1-\sigma_2)^2}{2}\int_0^1 y^2\rho_1^*(y)\,dy.
\end{equation*}
We note that the structure of the stochastic growth rate $ r$ for non-degenerate noise \eqref{e:lambda_evans} and for degenerate noise \eqref{e:lambda_ours} with $\sigma_1\neq \sigma_2$  is the same. The only difference is that one needs to make the substitution $\sigma_1^2+\sigma_2^2\mapsto (\sigma_1-\sigma_2)^2$ and the changes in $\hat\alpha_i$.

If $\sigma_1=\sigma_2=:\sigma$ the system \eqref{eq.by} for $\tilde \BY(t)=(\tilde Y_1(t), \tilde Y_2(t))$
can be written as
\begin{equation}
\begin{cases}
d\tilde Y_1(t)=&\Big(\tilde Y_1(t)(a_1-a_1\tilde Y_1(t)-a_2\tilde Y_2(t))-\alpha \tilde Y_1(t)+\beta \tilde Y_2(t)\Big)dt\\
&+\sigma^2\tilde Y_1(t)\Big[(\tilde Y_1(t)+\tilde Y_2(t))^2-(\tilde Y_1(t)+\tilde Y_2(t))^2\Big]dt
 \\
d\tilde Y_1(t)=&\Big(\tilde Y_2(t)(a_2-a_1\tilde Y_1(t)-a_2\tilde Y_2(t))-\beta \tilde Y_2(t)+\alpha \tilde Y_1(t)\Big)dt\\
&+\sigma^2\tilde Y_2(t)\Big[(\tilde Y_1(t)+\tilde Y_2(t))^2-(\tilde Y_1(t)+\tilde Y_2(t))^2\Big]dt.
\end{cases}
\end{equation}
 Using the fact that $\tilde Y_1(t)+\tilde Y_2(t)=1$ this reduces to
\begin{equation}\label{e3-ex4}
d\tilde Y_1(t)=\Big((a_1-a_2)[1-\tilde Y_1(t)]\tilde Y_1(t)+\beta-(\alpha+\beta)\tilde Y_1(t)]\Big)dt.
\end{equation}
 The unique equilibrium of \ref{e3-ex4} in [0,1]
 is the root $y^\star$ in [0,1] of
 $(a_1-a_2)(1-y)y+\beta-(\alpha+\beta)y=0.$
Hence, the unique invariant probability measure of $\tilde\BY(t)$
in this case is the Dirac measure concentrated in $(y^\star,1-y^\star)$.
Thus
\begin{equation*}
 r =
a_2-\dfrac{\sigma^2}2+(a_1-a_2+\sigma^2)y^\star.
\end{equation*}
}
\end{rmk}

\subsection{The degenerate case $\sigma_1=\sigma_2, \alpha=\beta$}\label{s:n=2_deg}

Consider the following system,
where $\alpha, \sigma,a_i, b_i, i=1,2$ are positive constants.
\begin{equation}
\begin{cases}
dX_1(t)=\big(X_1(t)(a_1-b_1 X_1(t))-\alpha X_1(t)+\alpha X_2(t)\big)dt+\sigma X_1(t)dB(t)  \\
dX_2(t)=\big(X_2(t)(a_2-b_2 X_2(t))+\alpha X_1(t)-\alpha X_2(t)\big)dt+\sigma X_2(t)dB(t).
\end{cases}
\end{equation}
Suppose that $a_1\ne a_2$ or that $b_1\ne b_2$.
 This system is degenerate since
 both equations are driven by a single Brownian motion.
 In this case, the unique equilibrium of \eqref{e3-ex4} in [0,1]
 is the root $y^\star$ in [0,1] of
 $(a_1-a_2)(1-y)y+\alpha(1-2y)=0.$
 Solving this quadratic equation, we have
 $y^\star=\dfrac{a_1-a_2-2\alpha+\sqrt{(a_1-a_2)^2+4\alpha^2}}{2(a_1-a_2)}$ if $a_1\neq a_2$ and $y^\star=\frac{1}{2}$ if $a_1= a_2$.

 It can be proved easily that
 this equilibrium is asymptotically stable
 and that
 $\lim_{t\to\infty}\tilde Y_1(t)=y^\star$.
 Thus, if $a_1\neq a_2$
 $$
 \begin{aligned}
  r=&a_1y^\star+a_2(1-y^\star)-\dfrac{\sigma^2}2\\
 =&a_2+\dfrac{a_1-a_2-2\alpha+\sqrt{(a_1-a_2)^2+4\alpha^2}}2-\dfrac{\sigma^2}2\\
 =&\dfrac{a_1+a_2-2\alpha+\sqrt{(a_1-a_2)^2+4\alpha^2}}2-\dfrac{\sigma^2}2.
 \end{aligned}
 $$
 As a result
 \begin{equation}\label{e:lambda_22}
 r =
\begin{cases}
\dfrac{a_1+a_2-2\alpha+\sqrt{(a_1-a_2)^2+4\alpha^2}}2-\dfrac{\sigma^2}2 &\text{if } a_1\neq a_2, b_1=b_2\\
a_1-\dfrac{\sigma^2}2&\text{if } a_1= a_2, b_1\neq b_2.
\end{cases}
\end{equation}
 Note that if  $a_1\ne a_2$ and $b_1=b_2$
 \[
 \alpha+(b_2/b_1)( a_1-a_2)-\alpha (b_2/b_1)^2 = a_1-a_2\ne 0
 \]
 and that if $a_1= a_2$ and $b_1\neq b_2$
 \[
 \alpha+(b_2/b_1)( a_1-a_2)-\alpha (b_2/b_1)^2 =\alpha\left(1-b_2/b_1\right)\ne 0.
 \]
 Therefore, the assumptions of Theorem \ref{t:survival_degenerate_2} hold.
 If $ r<0$, by Theorem \ref{t:extinction_degenerate_2} the population goes extinct,
 while if $ r>0$, the population persists by Theorem \ref{t:survival_degenerate_2}.

\subsection{The degenerate case when the conditions of Theorem  \ref{t:survival_degenerate_2} are violated} \label{s:n=2_deg_noHorm}

We analyse the system
\begin{equation}
\begin{cases}
dX_1(t)=\big(X_1(t)(a_1-b X_1(t))-\alpha X_1(t)+\beta X_2(t)\big)dt+\sigma X_1(t)dB(t)  \\
dX_2(t)=\big(X_2(t)(a_2-b X_2(t))+\alpha X_1(t)-\beta X_2(t)\big)dt+\sigma X_2(t)dB(t),
\end{cases}
\end{equation}
when $2(\beta-\alpha) = a_2 - a_1$. In this case $\sigma_1= \sigma_2=\sigma$,
\[
\beta+(b_2/b_1)( a_1-a_2-\alpha+\beta)-\alpha (b_2/b_1) =0
\]
and
$$ r=a_1-\alpha+\beta-\frac{\sigma^2}2.$$
If $ r<0$ then $\lim_{t\to\infty}X_1(t)=\lim_{t\to\infty}X_2(t)=0$ almost surely
as the result of Theorem \ref{t:extinction_degenerate_2}.

We focus on the case $ r>0$ and show that some of the results violate the conclusions of Theorem \ref{t:survival_degenerate_2}.

If we set $Z(t)=X_1(t)/X_2(t)$ then (see \eqref{e5.2})
\[
dZ(t)=\Big((1-Z(t))Z(t)X_2(t)+\beta+\hat a_1 Z(t)-\alpha Z^2(t)\Big)dt.
\]
Noting that $\hat a_1=a_1-a_2-\alpha+\beta=\alpha-\beta$ yields
\[
d(Z(t)-1)=\Big(-(Z(t)-1)Z(t)X_2(t)-(Z(t)-1)(\alpha Z(t)+\beta)\Big)dt.
\]
Assume $Z(0)\neq 1$ and without loss of generality suppose $Z(0)> 1$.
This implies
\begin{equation}\label{e-z-1}
Z(t)-1 = (Z(0)-1)\exp\left(-\int_0^t \left[Z(s)X_2(s)+(\alpha Z(s)+\beta)\right]\,ds\right).
\end{equation}
Since $Z(t)$ and $X_2(t)$ do not explode to $\pm \infty$ in finite time we can conclude that
if $Z(0)\ne 0$ then $Z(t)\ne 0$ for any $t\geq0$ with probability 1.
In other words,
if $\bx = (x_1,x_2)\in \R_+^{2,\circ}$ with $x_1\neq x_2$ then
\[
\PP(X^\bx_1(t)= X^\bx_2(t), t\geq 0)=0.
\]
One can further see from \eqref{e-z-1} that
$Z(t)-1$ tends to $0$ exponentially fast.
If $Z(0)=1$ let $X_1(0)=X_2(0)=x>0$. Similar arguments to the above show that
\[
\PP(X^\bx_1(t)\neq X^\bx_2(t), t\geq 0)=0	.
\]
To gain more insight into the asymptotic properties of $(X_1(t),X_2(t))$,
we study
$$
\begin{aligned}
dX_2(t)=&X_2(t)\Big((\hat a_2-b X_2(t))+\alpha Z(t)\Big)dt+\sigma X_2(t)dB(t)\\
=&X_2(t)\Big(a_1-\alpha+\beta-b X_2(t))+\alpha(Z(t)-1)\Big)dt+\sigma X_2(t)dB(t)
\end{aligned}
$$
We have from It\^o's formula that,
$$
\begin{aligned}
d \dfrac1{X_2(t)}=\left(b+(-a_1+\alpha-\beta+\sigma^2-\alpha(Z(t)-1))\dfrac1{X_2(t)}\right)dt-\sigma\dfrac1{X_2(t)} dB(t).
\end{aligned}
$$
By the variation-of constants formula (see \cite[Section 3.4]{MAO}),
we have
$$\dfrac1{X_2(t)}=\phi^{-1}(t)\left[\dfrac1{x_2}+b\int_0^t\phi(s)ds\right]$$
where
$$\phi(t):=\exp\left[ r t+\alpha\int_0^t(Z(s)-1)ds+\sigma B(t)\right].$$
Thus,
$$
X_2(t)=\dfrac{\phi(t)}{x_2^{-1}+b\int_0^t\phi(s)ds}.
$$
It is well-known that
$$
U(t):=\dfrac{e^{ r t+\sigma B(t)}}{x_2^{-1}+b\int_0^te^{ r s+\sigma B(s)}ds},
$$
is the solution to the stochastic logistic equation
\[
dU(t)=U(t)(a_1-\alpha+\beta-b U(t))\,dt+\sigma U(t)dB(t), U(0)=x_2.
\]
By the law of the iterated logarithm, almost surely
\begin{equation}\label{e-xx1}
\lim_{t\to\infty}\phi(t)=\lim_{t\to\infty}e^{ r t+\sigma B(t)}=\infty.
\end{equation}
We have
$$\dfrac{X_2(t)}{U(t)}=\dfrac{\exp\left(\alpha\int_0^t(Z(s)-1)ds\right)\left[x^{-1}_2+b\int_0^te^{ r t+\sigma B(t)}ds\right]}{x^{-1}_2+b\int_0^t\phi(s)ds}.
$$
In view of \eqref{e-xx1}, we can use L'hospital's rule
to obtain
\begin{equation}\label{e-xx2}
\begin{aligned}
\lim_{t\to\infty}\dfrac{X_2(t)}{U(t)}=&
\lim_{t\to\infty}\dfrac{\exp\left(\alpha\int_0^t(Z(s)-1)ds\right)e^{ r t+\sigma B(t)}}{\phi(t)}\\
&+\lim_{t\to\infty}\dfrac{\alpha(Z(t)-1)\exp\left(\alpha\int_0^t(Z(s)-1)ds\right)\left[x^{-1}_2+b\int_0^te^{ r t+\sigma B(t)}ds\right]}{b\phi(t)}\\
=&
1+\lim_{t\to\infty}\dfrac{\alpha(Z(t)-1)\left[x^{-1}_2+b\int_0^te^{ r t+\sigma B(t)}ds\right]}{be^{ r t+\sigma B(t)}}
\end{aligned}
\end{equation}
almost surely.
By the law of the iterated logarithm,
$\lim_{t\to\infty} \dfrac{e^{ r t+\sigma B(t)}}{e^{( r-\eps)t}}=\infty$
and
$\lim_{t\to\infty} \dfrac{e^{ r t+\sigma B(t)}}{e^{( r+\eps)t}}=0$
for any $\eps>0$.
Applying this and \eqref{e-z-1} to \eqref{e-xx2},
it is easy to show that with probability $1$
$$\lim_{t\to\infty}\dfrac{X_2(t)}{U(t)}=1.$$
Since  $\lim_{t\to\infty} Z(t)=1$ almost surely,
we also have $\lim_{t\to\infty}\dfrac{X_1(t)}{U(t)}=1$ almost surely.
Thus,  the long term behavior of $X_1(t)$ and $X_2(t)$ is governed by the one-dimensional diffusion $U(t)$.
In particular,
 both $X_1(t)$ and $X_2(t)$ converge to a unique invariant probability measure $\rho$ on $(0,\infty)$,
which is the invariant probability measure of $U(t)$.
 In this case, the invariant probability measure of $\BX(t) = (X_1(t),X_2(t))_{t\geq 0}$
 is not absolutely continuous with respect to the Lebesgue measure on $\R^{2,\circ}_+$.
 Instead, the invariant probability measure
 is concentrated on the one-dimensional manifold
 $\{\bx=(x_1,x_2)\in\R^{2,\circ}_+: x_1=x_2\}$.

\textbf{Biological interpretation} \textit{The stochastic growth rate in this degenerate setting is given by $ r=a_1-\alpha+\beta-\frac{\sigma^2}2$. We note that this term is equal to the stochastic growth rate of patch $1$, $a_1- \frac{\sigma^2}{2}$, to which we add $\beta$, the rate of dispersal from patch $1$ to patch $2$, and subtract $\alpha$, the rate of dispersal from patch $2$ to patch $1$. When
\[
a_1- \frac{\sigma^2}{2} > \alpha - \beta
\]
one has persistence, while when
\[
a_1- \frac{\sigma^2}{2} < \alpha - \beta
\]
one has extinction. In particular, if the patches on their own are sink patches so that $a_1- \frac{\sigma^2}{2}<0$ and $a_2- \frac{\sigma^2}{2}<0$ dispersion cannot lead to persistence since
\[
a_1- \frac{\sigma^2}{2} > \alpha - \beta ~\text{and}~a_2- \frac{\sigma^2}{2} > \beta - \alpha
\]
cannot hold simultaneously. The behavior of the system when $ r>0$ is different from the behavior in the non-degenerate setting of Theorem \ref{t:survival} or the degenerate setting of Theorem \ref{t:survival_degenerate_2}. Namely, if the patches start with equal populations then the patch abundances remain equal for all times and evolve according to the one-dimensional logistic diffusion $U(t)$. If the patches start with different population abundances then $X_1(t)$ and $X_2(t)$ are never equal but tend to each other asymptotically as $t\to\infty$. Furthermore, the long term behavior of $X_1(t)$ and $X_2(t)$ is once again determined by the logistic diffusion $U(t)$ as almost surely $\frac{X_i(t)}{U(t)}\to 1$ as $t\to\infty$. As such, if $ r>0$ we have persistence but the invariant measure the system converges to does not have $\R_+^{2,\circ}$ as its support anymore. Instead the invariant measure has the line  $\{\bx=(x_1,x_2)\in\R^{2,\circ}_+: x_1=x_2\}$ as its support.}

\begin{example}\label{ex:deg}
We discuss the case when $a_1\neq a_2$ and $ \sigma_1=\sigma_2$. The stochastic growth rate can be written by the analysis in the sections above as

\begin{equation}\label{e:rr}
 r =
\begin{cases}
\dfrac{a_1+a_2-2\alpha+\sqrt{(a_1-a_2)^2+4\alpha^2}}2-\dfrac{\sigma^2}2 &\text{if } \alpha=\beta, b_1=b_2\\
a_1-\alpha+\beta-\frac{\sigma^2}{2}&\text{if } a_2-a_1=2(\beta-\alpha), b_1=b_2.
\end{cases}
\end{equation}

\begin{figure}[htp]
\centering
\includegraphics[width=10cm, height=7cm]{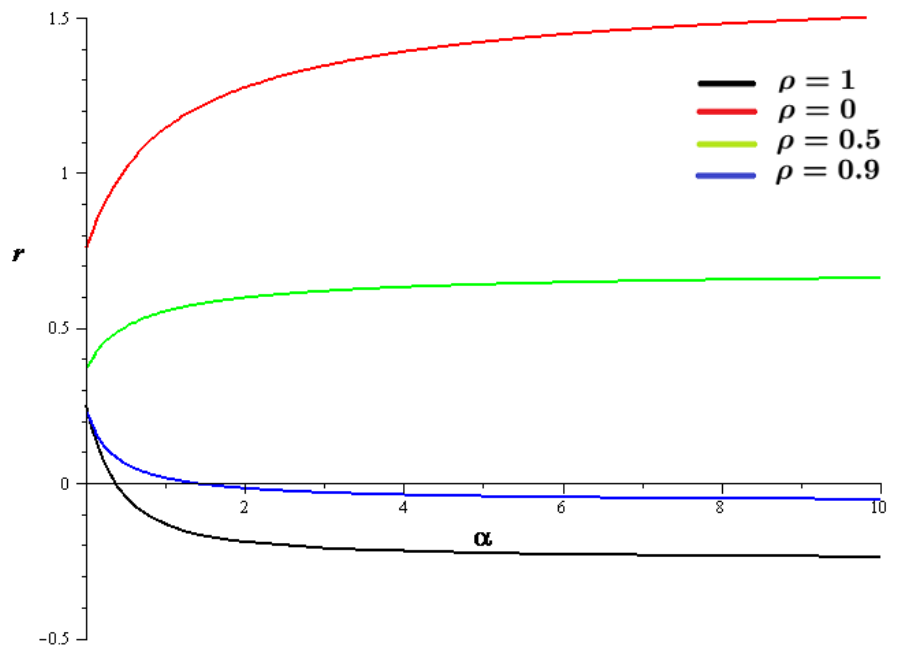}
\caption{
Consider \eqref{n=2-nonde} when $\alpha=\beta$ and the Brownian motions $B_1$ and $B_2$
are assumed to have correlation $\rho$.
The graphs show the stochastic growth rate $ r$ as a function of the dispersal rate $\alpha$ for different values of the correlation. Note that if $\rho=0$ we get the setting when the Brownian motions of the two patches are independent while when $\rho=1$ we get that one Brownian motion drives the dynamics of both patches. The parameters are $\alpha=\beta, a_1=3,a_2=4, \sigma_1^2=\sigma_2^2=7$.}\label{f_lambda}
\end{figure}

\textbf{Biological interpretation} \textit{In the case when $a_1=a_2$, $\sigma_1=\sigma_2$ and $b_1\neq b_2$ (so that the two patches only differ in their competition rates) the stochastic growth rate $ r$ does not depend on the dispersal rate $\alpha$. The system behaves just as a single-patch system with stochastic growth rate $a_1-\frac{\sigma^2}{2}$. In contrast to \cite[Example 1]{ERSS13} coupling two sink patches by dispersion cannot yield persistence.}

\textit{However, if the growth rates of the patches are different $a_1\neq a_2$ then the expression for $ r$ given in \eqref{e:rr} yields for $\alpha\gg |a_1-a_2|$ that  $$ r \approx\dfrac{a_1+a_2}2-\dfrac{\sigma^2}2+\frac{(a_1-a_2)^2}{8\alpha}.$$ In particular
\[
\lim_{\alpha\to\infty} r(\alpha) =  \dfrac{a_1+a_2}2-\dfrac{\sigma^2}2.
\]
We note that $ r$ is a decreasing function of the dispersal rate $\alpha$ for large values of $\alpha$ (also see Figure \ref{f_lambda}). This is different from the result of \cite[Example 1]{ERSS13} where $ r$ was shown to be an increasing function of $\alpha$. In contrast to the non-degenerate case, coupling patches by dispersal decreases the stochastic growth rate and as such makes persistence less likely. This highlights the negative effect of spatial correlations on population persistence and why one may no longer get the rescue effect. This is one of your main biological conclusions. Furthermore, we also recover that dispersal has a negative impact on the stochastic growth rate when there is spatial heterogeneity (i.e. $a_1\neq a_2$). This fact has a long history, going back to the work by \cite{K82}.}

\end{example}

\section{Discussion and Generalizations}\label{s:discussion}

For numerous models of population dynamics it is natural to assume that time is continuous. One reason for this is that often environmental conditions change continuously with time and therefore can naturally be described by continuous time models. There have been a few papers dedicated to the study of stochastic differential equation models of interacting, unstructured populations in stochastic environments (see \cite{BHS08, SBA11, EHS15}). These models however do not account for population structure or correlated environmental fluctuations.

Examples of structured populations can be found by looking at a population in which individuals can live in one of $n$ patches (e.g. fish swimming between basins of a lake or butterflies dispersing between meadows).  Dispersion is viewed by many population biologists as an important mechanism for survival. Not only does dispersion allow individuals to escape unfavorable landscapes (due to environmental changes or lack of resources), it also facilitates populations to smooth out local spatio-temporal environmental changes. Patch models of dispersion have been studied extensively in the deterministic setting (see for example \cite{H83,CCL12}). In the stochastic setting, there have been results for discrete time and space by \cite{BS09}, for continuous time and discrete space by \cite{ERSS13} and for structured populations that evolve continuously both in time and space.

We analyze the dynamics of a population that is spread throughout $n$ patches, evolves in a stochastic environment (that can be spatially correlated), disperses among the patches and whose members compete with each other for resources. We characterize the long-term behavior of our system as a function of $ r$ - the growth rate in the absence of competition. The quantity $ r$ is also the Lyapunov exponent of a suitable linearization of the system around $0$. Our analysis shows that $ r<0$ implies extinction and $ r>0$ persistence. The limit case $ r=0$ cannot be analyzed in our framework. We expect that new methods have to be developed in order to tackle the $ r=0$ scenario.

Since mathematical models are always approximations of nature it is necessary to study how the persistence and extinction results change under small perturbations of the parameters of the models. The concept of robust persistence (or permanence) has been introduced by \cite{HS92}. They showed that for certain systems persistence holds even when one has small perturbations of the growth functions. There have been results on robust persistence in the deterministic setting for Kolmogorov systems by \cite{S00, GH03}. Recently, robust permanence for deterministic Kolmogorov equations with respect to perturbations in both the growth functions and the feedback dynamics has been analyzed by \cite{PS16}. In the stochastic differential equations setting results on robust persistence and extinction have been proven by \cite{SBA11, BHS08}. We prove analogous results in our  framework where the populations are coupled by dispersal. For robust persistence we show in Appendix \ref{s:robust} that even with \textit{density-dependent}{ perturbations of the growth rates, dispersion matrix and environmental covariance matrix, if these perturbations are sufficiently small and if the unperturbed system is persistent then the perturbed system is also persistent. In the case of extinction we can prove robustness when there are small \textit{constant} perturbations of the growth rates, dispersal matrices and covariance matrices.}

In ecology there has been an increased interest in the spatial synchrony present in population dynamics. This refers to the changes in the time-dependent characteristics (i.e. abundances etc) of structured populations. One of the mechanisms which creates synchrony is the dependence of the population dynamics on a synchronous random environmental factor such as temperature or rainfall. The synchronizing effect of environmental stochasticity, or the so-called \textit{Moran effect}, has been observed in multiple population models. Usually this effect is the result of random but correlated weather effects acting on spatially structured populations. Following \cite{L93} one could argue that our world is a spatially correlated one. For many biotic and abiotic factors, like population density, temperature or growth rate, values at close locations are usually similar. For an in-depth analysis of spatial synchrony see \cite{K00, LKB04}.
Most stochastic differential models appearing in population dynamics treat only the case when the noise is non-degenerate (although see \cite{R03, DNDY16}). This simplifies the technical proofs significantly. However, from a biological point of view it is not clear that the noise should never be degenerate. For example if one models a system with multiple populations then all populations can be influenced by the same factors (a disease, changes in temperature and sunlight etc). Environmental factors can intrinsically create spatial correlations and as such it makes sense to study how these degenerate systems compare to the non-degenerate ones. In our setting the $n$ different patches could be strongly spatially correlated. Actually, in some cases it could be more realistic to have the same one-dimensional Brownian motion $(B_t)_{t\geq 0}$ driving the dynamics of all patches. We were able to find conditions under which the proofs from the non-degenerate case can be generalized to the degenerate setting. This is a first step towards a model that tries to explain the complex relationship between dispersal, stochastic environments and spatial correlations.

We fully analyze what happens if there are only two patches, $n=2$, and the noise is degenerate. Our results show unexpectedly, and in contrast to the non-degenerate results by \cite{ERSS13}, that coupling two sink patches cannot yield persistence. More generally, we show that the stochastic growth rate is a decreasing function of the dispersal rate. In specific instances of the degenerate setting, even when there is persistence, the invariant probability measure the system converges to does not have $\R_+^{2,\circ}$ as its support. Instead, the abundances of the two patches converge to an invariant probability measure supported on the line $\{\bx=(x_1,x_2)\in\R^{2,\circ}_+: x_1=x_2\}$.  These examples shows that degenerate noise is not just an added technicality - the results can be completely different from those in the non-degenerate setting. The negative effect of spatial correlations (including the fully degenerate case) has been studied in several papers for discrete-time models (see \cite{S10,HQ89, PL98, BPR02,RRB05}). The negative impact of dispersal on the stochastic growth rate $r$ when there is spatial heterogeneity (i.e. $a_1\neq a_2$) has a long history going back to the work of \cite{K82} on the \textit{Reduction Principle}. Following \cite{A12} the Reduction Principle can be stated as the widely exhibited phenomenon that mixing
reduces growth, and differential growth selects for reduced
mixing. The first use of this principle in the study of the evolution of dispersal can be found in \cite{H83}. The work of \cite{KLS06} provides an independent proof of the Reduction Principle and applications to nonlinear competing species in discrete-time, discrete-space models. In the case of continuous-time, discrete-space models (given by branching processes) a version of the Reduction Principle is analysed by \cite{SLS09}.

\subsection{$k$ species competing and dispersing in $n$ patches}
Real populations do not evolve in isolation and as a result much of ecology is concerned with
understanding the characteristics that allow two species to coexist, or one species to take over the
habitat of another. It is of fundamental importance to understand what will happen to an invading
species. Will it invade successfully or die out in the attempt? If it does invade, will it coexist
with the native population? Mathematical models for invasibility
%[??invasibility??]
have contributed significantly to
the understanding of the epidemiology of infectious disease outbreaks (\cite{cross05}) and ecological
processes (\cite{law96}; \cite{cas01}). There is widespread empirical evidence
that heterogeneity, arising from abiotic (precipitation, temperature, sunlight) or biotic (competition,
predation) factors, is important in determining invasibility (\cite{davies05}; \cite{pyvsek05}). However, few theoretical studies have investigated this;
see, e.g., \cite{SLS09}, \cite{schreiber11}, \cite{schreiber12}.

In this paper we have considered the dynamics of one population that disperses through $n$ patches. One possible generalization would be to look at $k$ populations $(\BX^1,\dots,\BX^k)$ that compete with each other for resources, have different dispersion strategies and possibly experience the environmental noise differently. Looking at such a model could shed light upon fundamental problems regarding invasions in spatio-temporally heterogeneous environments.

The extension of our results to competition models could lead to the development of a stochastic version of the treatment of the evolution of dispersal developed for patch models in the deterministic setting by \cite{H83} and \cite{CCL12}. In the current paper we have focused on how spatio-temporal variation influences the persistence and extinction of structured populations. In a follow-up paper we intend to look at the dispersal strategies in terms of \textit{evolutionarily stable strategies} 
(ESS) which can be characterized by showing that a population having a dispersal strategy $(D_{ij})$ cannot be invaded by any other population having a different dispersal strategy $(\tilde D_{ij})$. The first thing to check would be whether this model has ESS and, if they exist, whether they are unique. One might even get that there are no ESS in our setting. For example, \cite{SL11} show that there exist no ESS for periodic non-linear models and instead one gets a coalition of strategies that act as an ESS.
We expect to be able to generalize the results of \cite{CCL12} to a stochastic setting using the methods from this paper.

{\bf Acknowledgments.}  We thank Sebastian J. Schreiber and three anonymous referees for their detailed comments which helped improve this manuscript.

\bibliographystyle{plainnat}
\bibliography{multipatch}

\begin{thebibliography}{59}
\providecommand{\natexlab}[1]{#1}
\providecommand{\url}[1]{\texttt{#1}}
\expandafter\ifx\csname urlstyle\endcsname\relax
  \providecommand{\doi}[1]{doi: #1}\else
  \providecommand{\doi}{doi: \begingroup \urlstyle{rm}\Url}\fi

\bibitem[Altenberg(2012)]{A12}
L.~Altenberg.
\newblock The evolution of dispersal in random environments and the principle
  of partial control.
\newblock \emph{Ecological Monographs}, 82\penalty0 (3):\penalty0 297--333,
  2012.

\bibitem[Assing and Manthey(1995)]{AM95}
S.~Assing and R.~Manthey.
\newblock The behavior of solutions of stochastic differential inequalities.
\newblock \emph{Probab. Theory Related Fields}, 103\penalty0 (4):\penalty0
  493--514, 1995.

\bibitem[Bascompte et~al.(2002)Bascompte, Possingham, and Roughgarden]{BPR02}
J.~Bascompte, H.~Possingham, and J.~Roughgarden.
\newblock Patchy populations in stochastic environments: critical number of
  patches for persistence.
\newblock \emph{The American Naturalist}, 159\penalty0 (2):\penalty0 128--137,
  2002.

\bibitem[Bena{\"\i}m and Schreiber(2009)]{BS09}
M.~Bena{\"\i}m and S.~J. Schreiber.
\newblock Persistence of structured populations in random environments.
\newblock \emph{Theoretical Population Biology}, 76\penalty0 (1):\penalty0
  19--34, 2009.

\bibitem[Bena{\"\i}m et~al.(2008)Bena{\"\i}m, Hofbauer, and Sandholm]{BHS08}
M.~Bena{\"\i}m, J.~Hofbauer, and W.~H. Sandholm.
\newblock Robust permanence and impermanence for stochastic replicator
  dynamics.
\newblock \emph{Journal of Biological Dynamics}, 2\penalty0 (2):\penalty0
  180--195, 2008.

\bibitem[Blath et~al.(2007)Blath, Etheridge, and Meredith]{BEM07}
J.~Blath, A.~Etheridge, and M.~Meredith.
\newblock Coexistence in locally regulated competing populations and survival
  of branching annihilating random walk.
\newblock \emph{Ann. Appl. Probab.}, 17\penalty0 (5-6):\penalty0 1474--1507,
  2007.
\newblock ISSN 1050-5164.

\bibitem[Cantrell and Cosner(1991)]{CC91}
R.~S. Cantrell and C.~Cosner.
\newblock The effects of spatial heterogeneity in population dynamics.
\newblock \emph{Journal of Mathematical Biology}, 29\penalty0 (4):\penalty0
  315--338, 1991.

\bibitem[Cantrell et~al.(2012)Cantrell, Cosner, and Lou]{CCL12}
R.~S. Cantrell, C.~Cosner, and Y.~Lou.
\newblock Evolutionary stability of ideal free dispersal strategies in patchy
  environments.
\newblock \emph{J. Math. Biol.}, 65\penalty0 (5):\penalty0 943--965, 2012.

\bibitem[Caswell(2001)]{cas01}
H.~Caswell.
\newblock \emph{Matrix population models}.
\newblock Wiley Online Library, 2001.

\bibitem[Chesson(2000)]{C00}
P.~Chesson.
\newblock General theory of competitive coexistence in spatially-varying
  environments.
\newblock \emph{Theoretical population biology}, 58\penalty0 (3):\penalty0
  211--237, 2000.

\bibitem[Chueshov(2002)]{IG}
Igor Chueshov.
\newblock \emph{Monotone random systems theory and applications}, volume 1779.
\newblock Springer Science \& Business Media, 2002.

\bibitem[Cross et~al.(2005)Cross, Lloyd-Smith, Johnson, and Getz]{cross05}
P.~C. Cross, J.~O. Lloyd-Smith, P.~L.~F. Johnson, and W.~M. Getz.
\newblock Duelling timescales of host movement and disease recovery determine
  invasion of disease in structured populations.
\newblock \emph{Ecology Letters}, 8\penalty0 (6):\penalty0 587--595, 2005.

\bibitem[Davies et~al.(2005)Davies, Chesson, Harrison, Inouye, Melbourne, and
  Rice]{davies05}
K.~F. Davies, P.~Chesson, S.~Harrison, B.~D. Inouye, B.~Melbourne, and K.~J.
  Rice.
\newblock Spatial heterogeneity explains the scale dependence of the
  native-exotic diversity relationship.
\newblock \emph{Ecology}, 86\penalty0 (6):\penalty0 1602--1610, 2005.

\bibitem[Dennis and Patil(1984)]{DP84}
B.~Dennis and G.~P. Patil.
\newblock The gamma distribution and weighted multimodal gamma distributions as
  models of population abundance.
\newblock \emph{Mathematical Biosciences}, 68\penalty0 (2):\penalty0 187--212,
  1984.

\bibitem[Dieu et~al.(2016)Dieu, Nguyen, Du, and Yin]{DNDY16}
N.~T. Dieu, D.~H. Nguyen, N.~H. Du, and G.~Yin.
\newblock Classification of asymptotic behavior in a stochastic sir model.
\newblock \emph{SIAM Journal on Applied Dynamical Systems}, 15\penalty0
  (2):\penalty0 1062--1084, 2016.

\bibitem[Du et~al.(2016)Du, Nguyen, and Yin]{DDY}
N.~H. Du, D.~H. Nguyen, and G.~Yin.
\newblock Conditions for permanence and ergodicity of certain stochastic
  predator-prey models.
\newblock \emph{Journal of Applied Probability}, 53, 2016.

\bibitem[Durrett and Remenik(2012)]{DR12}
R.~Durrett and D.~Remenik.
\newblock Evolution of dispersal distance.
\newblock \emph{Journal of Mathematical Biology}, 64\penalty0 (4):\penalty0
  657--666, 2012.

\bibitem[Evans et~al.(2013)Evans, Ralph, Schreiber, and Sen]{ERSS13}
S.~N. Evans, P.~L. Ralph, S.~J. Schreiber, and A.~Sen.
\newblock Stochastic population growth in spatially heterogeneous environments.
\newblock \emph{J. Math. Biol.}, 66\penalty0 (3):\penalty0 423--476, 2013.

\bibitem[Evans et~al.(2015)Evans, Hening, and Schreiber]{EHS15}
S.~N. Evans, A.~Hening, and S.~J. Schreiber.
\newblock Protected polymorphisms and evolutionary stability of patch-selection
  strategies in stochastic environments.
\newblock \emph{J. Math. Biol.}, 71\penalty0 (2):\penalty0 325--359, 2015.
\newblock ISSN 0303-6812.

\bibitem[Garay and Hofbauer(2003)]{GH03}
B.~M Garay and J.~Hofbauer.
\newblock Robust permanence for ecological differential equations, minimax, and
  discretizations.
\newblock \emph{SIAM Journal on Mathematical Analysis}, 34\penalty0
  (5):\penalty0 1007--1039, 2003.

\bibitem[Gei{\ss} and Manthey(1994)]{GM94}
C.~Gei{\ss} and R.~Manthey.
\newblock Comparison theorems for stochastic differential equations in finite
  and infinite dimensions.
\newblock \emph{Stochastic Process. Appl.}, 53\penalty0 (1):\penalty0 23--35,
  1994.

\bibitem[Gonzalez and Holt(2002)]{GH02}
A.~Gonzalez and R.~D. Holt.
\newblock The inflationary effects of environmental fluctuations in
  source--sink systems.
\newblock \emph{Proceedings of the National Academy of Sciences}, 99\penalty0
  (23):\penalty0 14872--14877, 2002.

\bibitem[Hardin et~al.(1988{\natexlab{a}})Hardin, Tak{\'a}{\v{c}}, and
  Webb]{HTW88}
D.~P. Hardin, P.~Tak{\'a}{\v{c}}, and G.~F. Webb.
\newblock Asymptotic properties of a continuous-space discrete-time population
  model in a random environment.
\newblock \emph{Journal of Mathematical Biology}, 26\penalty0 (4):\penalty0
  361--374, 1988{\natexlab{a}}.

\bibitem[Hardin et~al.(1988{\natexlab{b}})Hardin, Tak{\'a}{\v{c}}, and
  Webb]{HTW88b}
D.~P. Hardin, P.~Tak{\'a}{\v{c}}, and G.~F. Webb.
\newblock A comparison of dispersal strategies for survival of spatially
  heterogeneous populations.
\newblock \emph{SIAM Journal on Applied Mathematics}, 48\penalty0 (6):\penalty0
  1396--1423, 1988{\natexlab{b}}.

\bibitem[Hardin et~al.(1990)Hardin, Tak{\'a}{\v{c}}, and Webb]{HTW90}
D.~P. Hardin, P.~Tak{\'a}{\v{c}}, and G.~F. Webb.
\newblock Dispersion population models discrete in time and continuous in
  space.
\newblock \emph{Journal of Mathematical Biology}, 28\penalty0 (1):\penalty0
  1--20, 1990.

\bibitem[Harrison and Quinn(1989)]{HQ89}
S.~Harrison and J.~F. Quinn.
\newblock Correlated environments and the persistence of metapopulations.
\newblock \emph{Oikos}, pages 293--298, 1989.

\bibitem[Hastings(1983)]{H83}
A.~Hastings.
\newblock Can spatial variation alone lead to selection for dispersal?
\newblock \emph{Theoretical Population Biology}, 24\penalty0 (3):\penalty0
  244--251, 1983.

\bibitem[Hutson and Schmitt(1992)]{HS92}
V.~Hutson and K.~Schmitt.
\newblock Permanence and the dynamics of biological systems.
\newblock \emph{Mathematical Biosciences}, 111\penalty0 (1):\penalty0 1--71,
  1992.

\bibitem[Ikeda and Watanabe(1989)]{IW}
N.~Ikeda and S.~Watanabe.
\newblock \emph{Stochastic Differential Equations and Diffusion Processes}.
\newblock North-Holland Publishing Co., Amsterdam, 1989.

\bibitem[Jansen and Yoshimura(1998)]{JY98}
V.~A.~A. Jansen and J.~Yoshimura.
\newblock Populations can persist in an environment consisting of sink habitats
  only.
\newblock \emph{Proceedings of the National Academy of Sciences}, 95\penalty0
  (7):\penalty0 3696--3698, 1998.

\bibitem[Jarner and Roberts(2002)]{JR}
S.~F. Jarner and G.~O. Roberts.
\newblock Polynomial convergence rates of {M}arkov chains.
\newblock \emph{Ann. Appl. Probab.}, 12\penalty0 (1):\penalty0 224--247, 2002.
\newblock ISSN 1050-5164.

\bibitem[Kallenberg(2002)]{K02}
O.~Kallenberg.
\newblock \emph{Foundations of Modern Probability}.
\newblock Springer, 2002.

\bibitem[Karlin(1982)]{K82}
S.~Karlin.
\newblock Classifications of selection-migration structures and conditions for
  a protected polymorphism.
\newblock \emph{Evol. Biol}, 14\penalty0 (61):\penalty0 204, 1982.

\bibitem[Kendall et~al.(2000)Kendall, Bj{\o}rnstad, Bascompte, Keitt, and
  Fagan]{K00}
B.~E. Kendall, O.~N. Bj{\o}rnstad, J.~Bascompte, T.~H. Keitt, and W.~F. Fagan.
\newblock Dispersal, environmental correlation, and spatial synchrony in
  population dynamics.
\newblock \emph{The American Naturalist}, 155\penalty0 (5):\penalty0 628--636,
  2000.

\bibitem[Khasminskii(2012)]{RK}
R.~Khasminskii.
\newblock \emph{Stochastic stability of differential equations}, volume~66 of
  \emph{Stochastic Modelling and Applied Probability}.
\newblock Springer, Heidelberg, second edition, 2012.
\newblock With contributions by G. N. Milstein and M. B. Nevelson.

\bibitem[Kirkland et~al.(2006)Kirkland, Li, and Schreiber]{KLS06}
S.~Kirkland, C.-K. Li, and S.~J. Schreiber.
\newblock On the evolution of dispersal in patchy landscapes.
\newblock \emph{SIAM Journal on Applied Mathematics}, 66\penalty0 (4):\penalty0
  1366--1382, 2006.

\bibitem[Kliemann(1987)]{WK}
W.~Kliemann.
\newblock Recurrence and invariant measures for degenerate diffusions.
\newblock \emph{Ann. Probab.}, 15:\penalty0 690--707, 1987.

\bibitem[Law and Morton(1996)]{law96}
R.~Law and R.~D. Morton.
\newblock Permanence and the assembly of ecological communities.
\newblock \emph{Ecology}, pages 762--775, 1996.

\bibitem[Legendre(1993)]{L93}
P.~Legendre.
\newblock Spatial autocorrelation: trouble or new paradigm?
\newblock \emph{Ecology}, 74\penalty0 (6):\penalty0 1659--1673, 1993.

\bibitem[Liebhold et~al.(2004)Liebhold, Koenig, and Bj{\o}rnstad]{LKB04}
A.~Liebhold, W.~D. Koenig, and O.~N. Bj{\o}rnstad.
\newblock Spatial synchrony in population dynamics.
\newblock \emph{Annual Review of Ecology, Evolution, and Systematics}, pages
  467--490, 2004.

\bibitem[Mao(1997)]{MAO}
X.~Mao.
\newblock \emph{Stochastic differential equations and their applications}.
\newblock Horwood Publishing Series in Mathematics \& Applications. Horwood
  Publishing Limited, Chichester, 1997.

\bibitem[Meyn and Tweedie(1993)]{MT}
S.~P. Meyn and R.~L. Tweedie.
\newblock Stability of {M}arkovian processes. {III}. foster-lyapunov criteria
  for continuous-time processes.
\newblock \emph{Adv. in Appl. Probab.}, 24\penalty0 (3):\penalty0 518--548,
  1993.

\bibitem[Mierczy{\'n}ski and Shen(2004)]{MS04}
J.~Mierczy{\'n}ski and W.~Shen.
\newblock Lyapunov exponents and asymptotic dynamics in random {K}olmogorov
  models.
\newblock \emph{Journal of Evolution Equations}, 4\penalty0 (3):\penalty0
  371--390, 2004.

\bibitem[Nummelin(1984)]{EN}
E.~Nummelin.
\newblock \emph{General irreducible {M}arkov chains and nonnegative operators},
  volume~83 of \emph{Cambridge Tracts in Mathematics}.
\newblock Cambridge University Press, Cambridge, 1984.

\bibitem[Palmqvist and Lundberg(1998)]{PL98}
E.~Palmqvist and P.~Lundberg.
\newblock Population extinctions in correlated environments.
\newblock \emph{Oikos}, pages 359--367, 1998.

\bibitem[{Patel} and {Schreiber}(2016)]{PS16}
S.~{Patel} and S.~J. {Schreiber}.
\newblock {Robust permanence for ecological equations with internal and
  external feedbacks}.
\newblock \emph{ArXiv e-prints}, 2016.
\newblock URL \url{https://arxiv.org/abs/1612.06554}.

\bibitem[Py{\v{s}}ek and Hulme(2005)]{pyvsek05}
P.~Py{\v{s}}ek and P.~E. Hulme.
\newblock Spatio-temporal dynamics of plant invasions: linking pattern to
  process.
\newblock \emph{Ecoscience}, 12\penalty0 (3):\penalty0 302--315, 2005.

\bibitem[Rey-Bellet(2006)]{LB}
L.~Rey-Bellet.
\newblock Ergodic properties of {M}arkov processes.
\newblock In \emph{Open quantum systems. {II}}, volume 1881 of \emph{Lecture
  Notes in Math.}, pages 1--39. Springer, Berlin, 2006.

\bibitem[Roth and Schreiber(2014)]{RS14}
G.~Roth and S.~J. Schreiber.
\newblock Persistence in fluctuating environments for interacting structured
  populations.
\newblock \emph{Journal of Mathematical Biology}, 69\penalty0 (5):\penalty0
  1267--1317, 2014.

\bibitem[Roy et~al.(2005)Roy, Holt, and Barfield]{RRB05}
M.~Roy, R.~D. Holt, and M.~Barfield.
\newblock Temporal autocorrelation can enhance the persistence and abundance of
  metapopulations comprised of coupled sinks.
\newblock \emph{The American Naturalist}, 166\penalty0 (2):\penalty0 246--261,
  2005.

\bibitem[Rudnicki(2003)]{R03}
Ryszard Rudnicki.
\newblock Long-time behaviour of a stochastic prey–predator model.
\newblock \emph{Stochastic Processes and their Applications}, 108\penalty0
  (1):\penalty0 93 -- 107, 2003.

\bibitem[Schmidt(2004)]{S04}
K.~A Schmidt.
\newblock Site fidelity in temporally correlated environments enhances
  population persistence.
\newblock \emph{Ecology Letters}, 7\penalty0 (3):\penalty0 176--184, 2004.

\bibitem[Schreiber(2000)]{S00}
S.~J. Schreiber.
\newblock Criteria for cr robust permanence.
\newblock \emph{Journal of Differential Equations}, 162\penalty0 (2):\penalty0
  400--426, 2000.

\bibitem[Schreiber(2010)]{S10}
S.~J. Schreiber.
\newblock Interactive effects of temporal correlations, spatial heterogeneity
  and dispersal on population persistence.
\newblock \emph{Proceedings of the Royal Society of London B: Biological
  Sciences}, 2010.

\bibitem[Schreiber(2012)]{schreiber12}
S.~J. Schreiber.
\newblock The evolution of patch selection in stochastic environments.
\newblock \emph{The American Naturalist}, 180\penalty0 (1):\penalty0 17--34,
  2012.

\bibitem[Schreiber and Li(2011)]{SL11}
S.~J. Schreiber and C-K Li.
\newblock Evolution of unconditional dispersal in periodic environments.
\newblock \emph{Journal of Biological Dynamics}, 5\penalty0 (2):\penalty0
  120--134, 2011.

\bibitem[Schreiber and Lloyd-Smith(2009)]{SLS09}
S.~J. Schreiber and J.~O. Lloyd-Smith.
\newblock Invasion dynamics in spatially heterogeneous environments.
\newblock \emph{The American Naturalist}, 174\penalty0 (4):\penalty0 490--505,
  2009.

\bibitem[Schreiber and Ryan(2011)]{schreiber11}
S.~J. Schreiber and M.~E. Ryan.
\newblock Invasion speeds for structured populations in fluctuating
  environments.
\newblock \emph{Theoretical Ecology}, 4\penalty0 (4):\penalty0 423--434, 2011.

\bibitem[Schreiber et~al.(2011)Schreiber, Bena{\"\i}m, and Atchad{\'e}]{SBA11}
S.~J. Schreiber, M.~Bena{\"\i}m, and K.~A.~S. Atchad{\'e}.
\newblock Persistence in fluctuating environments.
\newblock \emph{Journal of Mathematical Biology}, 62\penalty0 (5):\penalty0
  655--683, 2011.

\end{thebibliography}

\appendix
\section{The case $ r>0$}\label{sec:+}
The next sequence of lemmas and propositions is used
to prove Theorem \ref{t:survival}. We start by showing that our processes are well-defined Markov processes.

\begin{prop}\label{e:solutions}
The SDE (stochastic differential equation) defined by \eqref{e4.0} has unique strong solutions $\BX(t)=(X_1(t), \dots, X_n(t))$, $t\geq0$ for any $\bx=(x_1,\dots, x_n)\in\R^{n}_+$. Furthermore,   $\BX(t)$ is a strong Markov process with the Feller property, is irreducible on $\R_+^n\setminus \{\mathbf{0}\}$ and  $\PP\{X_i(t)>0,\, t>0, i=1,\dots,n\}=1$ if $\BX(0)\in\R_+^n\setminus \{\mathbf{0}\}$.
\end{prop}
\begin{proof}

Since the coefficients of \eqref{e4.0} are locally Lipschitz,
there exists a unique local solution to \eqref{e4.0} with a given initial value.
In other words, for any initial value, there is a stopping time $\tau_e>0$ and a process $(\BX(t))_{t\geq 0}$ satisfying \eqref{e4.0} up to $\tau_e$ and $\lim\limits_{t\to\tau_e} \|\BX(t)\|=\infty$
(see e.g. \cite[Section 3.4]{RK}).
Clearly, if $\BX(0)=0$ then $\BX(t)=0,\, t\in[0,\tau_e)$ which implies that $\tau_e=\infty$.
By a comparison theorem for SDEs  (see \cite[Theorem 1.2]{GM94} and Remark \ref{r:comparison} below),
\begin{equation}\label{e1.1b}
\PP\left\{X_i(t)<\mathcal X_i(t),\,  t\in(0,\tau_e), i=1,\dots, n\right\}=1 \text{ if }X_i(0)=\mathcal X_i(0) \geq M_b
\end{equation}
where $(\mathcal X_i(t))_{t\geq 0}$ is given by \eqref{e1.1c}.
Since \eqref{e1.1c} has a global solution due to the Lipschitz property of its coefficients,
we have from \eqref{e1.1b} that
$\tau_e=\infty$ almost surely.
Define the process
\[
d\bar {\mathcal{X}}_i(t)=\left(- \left|\frac{3a_i}{2}\right|\bar {\mathcal{X}}_i(t)+\sum_{j=1}^n D_{ji}\bar {\mathcal{X}}_j(t)\right)dt+\bar {\mathcal{X}}_i(t)dE_i(t),\, \, i=1,\dots,n.
\]
Since the $b_i$s are continuous and vanish at $0$, there exists $r>0$ such that for $|\mathbf{x}|\leq r$ we have
\begin{equation}\label{e:drift}
- \left|\frac{3a_i}{2}\right| \leq a_i-b_i(x_i), \,i=1,\dots,n.
\end{equation}
Let $\tau$ be the stopping time
\begin{equation}\label{e:tau}
\tau:= \inf\left\{t:\left|\bar {\mathcal{X}}(t)\right|>r\right\}
\end{equation}
Now, consider the case $\BX(0)\in\R_+^n\setminus \{\mathbf{0}\}$. By \cite[Proposition 3.1]{ERSS13}, \eqref{e:drift}, \eqref{e:tau} and a comparison argument (see Remark \ref{r:comparison} and the proof of \cite[Theorem 4.1]{EHS15}), we can show that
\[
\PP\left\{X_i\geq \bar {\mathcal{X}}_i(t) >0,  t\in(0, \tau)\right\}=1,
\]
which implies
\begin{equation}\label{e1.1a}
\PP\left\{X_i(t)>0,  t\in(0,\infty)\right\}=1\,\text{ for all }\BX(0)\in\R_+^n\setminus \{\mathbf{0}\}.
\end{equation}

Moreover, since $\PP\left\{0\leq X_i(t)<\mathcal X_i(t)\,~\text{for all}~ t\geq0, i=1,\dots,n\right\}=1$,
we can use standard arguments (e.g., \cite[Theorem 2.9.3]{MAO}) to obtain the Feller property of
the solution to \eqref{e4.0}.
\end{proof}

\begin{rmk}
There are different possible definitions of ``Feller'' in the literature. What we mean by Feller is that the semigroup $(T_t)_{t\geq 0}$ of the process maps the set of bounded continuous functions $C_b(\R_+^n)$ into itself i.e.
\[
T_t(C_b(\R_+^n)) \subset C_b(\R_+^n), ~t\geq 0.
\]
\end{rmk}

\begin{deff}\label{def:mn}
{\rm We call a mapping $f:\R^d\to\R^d$ \textbf{quasi-monotonously increasing}, if for $j=1,\dots,d$
\[
f_j(x)\leq f_j(y),
\]
whenever $x_j=y_j$ and $x_l\leq y_l, l\neq j$.
}
\end{deff}

\begin{rmk}\label{r:comparison}{\rm
 One often wants to
  apply the well-known comparison theorem for one-dimensional SDEs (see \cite{IW}) to a multidimensional setting.  Below we explain why we can make use of comparison theorems for stochastic differential equations in our setting.
Consider the following two systems
\begin{equation}\label{e:R}
dR_j(t) = a_j(t, R(t))\,dt + \sum_{k=1}^r\sigma_{jk}(t,R(t))dW_k(t)
\end{equation}
and
\begin{equation}\label{e:S}
dS_j(t) = b_j(t, S(t))\,dt + \sum_{k=1}^r\sigma_{jk}(t,S(t))dW_k(t)
\end{equation}
for $j=1,\dots,d, t\geq 0$ together with the initial condition
\begin{equation}\label{e:RS}
R_j(0)\leq S_j(0), j=1,\dots,d~~~\PP-\text{a.s.},
\end{equation}
where $W=(W_1(t),\dots,W_r(t))_{t\geq 0}$ is an $r$-dimensional
 standard Brownian motion, and the coefficients $a_i,b_i,\sigma_{jk}$ are continuous mappings on $\R_+\times \R^d$. Suppose \eqref{e:R} and \eqref{e:S} have explosion times $\theta_R, \theta_S$.

Let (C0), (C1), and (C2) be the following conditions.
\begin{itemize}
\item[(C0)] The solution to \eqref{e:R} is pathwise unique and the drift coefficient $a(t,x)$ is quasi-monotonously
(see Definition \ref{def:mn})
 increasing with respect to $x$.
\item[(C1)] For every $t\geq 0$, $j=1,\dots,d$ and $x\in\R^d$ the following inequality holds
\[
a_j(t,x)\leq b_j(t,x).
\]
\item[(C2)] There exists a strictly increasing function $\rho:\R_+\to\R_+$ with $\rho(0)=0$ and
\[
\int_{0^+}^\infty\frac{1}{\rho^2(u)}\,du=\infty
\]
such that for each $j=1,\dots,d$
\[
\sum_{k=1}^r|\sigma_{jk}(t,x)-\sigma_{jk}(t,y)|\leq \rho(|x_j-y_j|) \ \hbox{ for all } \ t\geq 0, x,y\in\R^d.
\]
\end{itemize}

Sometimes it is assumed incorrectly that conditions (C1) and (C2) suffice to conclude that $\PP\{R(t)\leq Y(t), t\in [0, \theta_R\wedge \theta_S)\}=1$. Some illuminating counterexamples regarding this issue can be found in \cite[Section 3]{AM95}.
However, if in addition to conditions (C1) and (C2), one also has condition (C0), then \cite[Theorem 1.2]{GM94} indicates that  $\PP\{R(t)\leq Y(t), t\in [0, \theta_R\wedge \theta_S)\}=1$.
Note that, in the setting of our paper, the drift coefficient of \eqref{e1.1c} is quasi-monotonously increasing and we can pick $\rho(x)=x, x\in\R_+$. Therefore, conditions (C0), (C1), and C(2) hold,
which  allows us to use the comparison results.
In special cases one can prove comparison theorems even when quasi-monotonicity fails; see \cite[Theorem 6.1]{EHS15} and \cite[Corollary A.2]{BEM07}.
}
\end{rmk}

 To proceed, let us recall some technical concepts and results needed to prove the main theorem.
 Let ${\bf\Phi}=(\Phi_0,\Phi_1,\dots)$ be a discrete-time Markov chain on a general state space $(E,\mathcal{E})$, where $\mathcal{E}$ is a countably generated $\sigma$-algebra.
 Denote by $\mathcal{P}$ the Markov transition kernel for ${\bf\Phi}$.
If there is a non-trivial $\sigma$-finite positive measure $\varphi$ on $(E,\mathcal{E})$ such that for
any $A\in\mathcal{E}$ satisfying $\varphi(A)>0$ we have
$$\sum_{n=1}^\infty \mathcal{P}^n(x, A)>0,\, x\in E$$
 where  $ \mathcal{P}^n$ is the $n$-step transition kernel of ${\bf\Phi}$, then the Markov chain ${\bf\Phi}$ is called $\varphi$-\textit{irreducible}.
It can be shown (see \cite{EN}) that if ${\bf\Phi}$ is $\varphi$-irreducible, then there exists a positive integer $d$ and disjoint subsets $E_0,\dots,
E_{d-1}$ such that for all $i=0,\dots, d-1$ and all $x\in E_i$, we have
$$\mathcal{P}(x,E_j)=1 \text{ where } j=i+1 \text{ (mod } d)$$
and
$$ \varphi \left(E\setminus \bigcup_{i=0}^{d-1}E_i\right)=0.$$
The smallest positive integer $d$ satisfying the above is called the period of ${\bf\Phi}.$
An \textit{aperiodic} Markov chain is a chain with period $d=1$.

A set $C\in\mathcal{E}$ is called \textit{petite}, if there exists a non-negative sequence $(a_n)_{n\in\N}$ with $\sum_{n=1}^\infty a_n=1$
and a nontrivial positive measure $\nu$ on $(E,\mathcal{E})$
such that
$$\sum_{n=1}^\infty a_n \mathcal{P}^n(x, A)\geq\nu(A),\,\, x\in C, A\in\mathcal{E}.$$
The following theorem is extracted from \cite[Theorem 3.6]{JR}.

\begin{thm}\label{polyrate}
Suppose that ${\bf\Phi}$ is irreducible and aperiodic and fix $0<\gamma<1$. Assume that there exists a petite set $C\subset E$, positive constants $\kappa_1, \kappa_2$ and  a function $V:E\to[1,\infty)$ such that $$\mathcal{P}V\leq V-\kappa_1V^\gamma+\kappa_2\1_{C}.$$
Then there exists a probability measure $\pi$ on $(E,\mathcal{E})$ such that
$$(n+1)^{\frac{\gamma}{1-\gamma}}\|\mathcal{P}(x,\cdot)-\pi(\cdot)\|_{TV}\to0\text{ as } n\to\infty \text{ for all } x\in E.$$
\end{thm}

The next series of lemmas and propositions are used to show that we can construct a function $V$ satisfying the assumptions of Theorem \ref{polyrate}.

\begin{lem}\label{petite}
For any $T>0$, there exists an open set $N_0\subset \R^{n,\circ}_+$
%open
such that the Markov chain $\{(\BY(kT), S(kT)), k\in \N\}$ on $\Delta\times(0,\infty)$ is $\varphi$-irreducible and aperiodic, where $\varphi(\cdot)=m(\cdot\cap N_0)$
and
%for
$m(\cdot)$ is Lebesgue measure.
Moreover, every compact set $K\subset\Delta\times(0,\infty)$ is petite.
Similarly, $\Delta$ is a petite set of the Markov chain $\{\tilde\BY(kT), k\in\N\}$.
\end{lem}

\begin{proof}
To prove this lemma, it is more convenient to work with the process $\BX(t)$ that lives on $\R^n_+\setminus\{\mathbf{0}\}$.
Since $(\BX(t))_{t\geq 0}$
is a nondegenerate diffusion with smooth coefficients in $\R^{n,\circ}_+$, by \cite[Corollary 7.2]{LB},
the transition semigroup $P_\BX(t, \bx, \cdot)$ of $(\BX(t))_{t\geq 0}$
has a smooth, positive density $(0,\infty)\times\R^{2n,\circ}_+\ni (t,\bx,\bx')) \mapsto p_\BX(t, \bx, \bx')\in [0,\infty)$.
Fix a point $\bx_0\in \R^{n,\circ}_+$.
Since $\int_{\R^{n,\circ}_+}p(t, \bx_0, \bx)d\bx=1$
 there exists $\bx_1\in\R^{n,\circ}_+$ such that
$p_\BX\left(\frac{T}2, \bx_0, \bx_1\right)>0$. There exist bounded open sets $N_0\ni\bx_0$, $N_1\ni\bx_1$ satisfying
\begin{equation}\label{pe.1}
\hat p:=\inf\left\{p_\BX\left(\frac{T}2, \bx, \bx'\right)>0: \bx\in N_0, \bx'\in N_1\right\}>0.
\end{equation}
Slightly modifying the proof of \cite[Proposition 3.1]{ERSS13} (the part proving the irreducibility of the solution process),
we have that
$\tilde p_\bx:=P_\BX\left(\frac{T}2, \bx, N_0\right)>0$ for all $\bx\in \R^n_+\setminus\{\mathbf{0}\}.$
Since $(\BX(t))_{t\geq 0}$ has the Feller property,
there is a neighborhood $N_\bx\ni\bx$ such that
\begin{equation}\label{pe.2}
P_\BX\left(\frac{T}2, \bx', N_0\right)>\frac{\tilde p_\bx}2, \, \bx'\in N_\bx.
\end{equation}
For any compact set $K\in \R^n_+\setminus\{\mathbf{0}\}$, there are finite $\bx_2,\dots, \bx_k$ such that
$K\subset \bigcup_{i=2}^kN_{\bx_i}$.
As a result,
\begin{equation}\label{pe.3}
P_\BX\left(\frac{T}2, \bx', N_0\right)>\tilde p_K:=\min\left\{\frac{\tilde p_{\bx_i}}2, i=2,\dots, k\right\}.
\end{equation}
In view of \eqref{pe.1}, \eqref{pe.2}, and \eqref{pe.3}, an application of the Chapman-Kolmogorov equations
yields
that for any $\bx\in K$ and
any measurable set $A\subset\R^{n,\circ}_+$,
$$P_\BX\big(T, \bx, A\big)\geq \int_{N_0}P_\BX\left(\frac{T}2, \bx,  d\bx'\right)P_\BX\left(\frac{T}2, \bx', A\right)\geq \tilde p_K\hat p
m(A\cap N_1),$$
where $m(\cdot)$ is Lebesgue measure on $\R^{n,\circ}_+$.
Since the measure $\nu(\cdot)=m(\cdot\cap N_1)$ is non-trivial, we can easily obtain that
$K$ is a petite set of the Markov chain $\{(\BX(kT)), k\in\N\}.$
Moreover, $K$ can be chosen arbitrarily.
Hence, for any $\bx\in\R^{n}_+\setminus\{\mathbf{0}\}$ there is $\bar p_{\bx}>0$ such that
\begin{equation}\label{pe.4}
P_\BX\big(T, \bx, \cdot\big)\geq \bar p_{\bx} m(\cdot\cap N_1),
\end{equation}
which means that $\{(\BX(kT)), k\in\N\}$ is irreducible.

Suppose that $\{(\BX(kT)), k\in\N\}$ is not aperiodic.
Then there are disjoint subsets of $\R^{n}_+\setminus\{\mathbf{0}\}$, denoted by $A_0, \dots, A_{d-1}$ with $d>1$
such that for any $\bx\in A_i$,
$$P_\BX(T, \bx, A_j)=1 \text{ where } j=i+1 \text{ (mod } d).$$
Since $P(T, \bx, \cdot)$ has a density,
$m(A_i)>0$ for $i=0,\dots, d-1$.
In view of \eqref{pe.4}, we must have $m(N_0\cap A_i)=0$ for any $i=0,\dots,d-1$.
This contradicts the fact that
$$ m\left( N_0\bigcap \left(E\setminus \bigcup_{i=0}^{d-1}A_i\right)\right)=0.$$
This contradiction implies that $\{\BX(kT), k\in\N\}$ is aperiodic.
In the same manner, we can prove that $\tilde\BY(t)$ is irreducible, aperiodic and
its state space, $\Delta$, is petite.
\end{proof}

\begin{lem}\label{lem2.1}
There exists a positive constant $K_1$ such that
\begin{equation}\label{lm2.1a}
   \E S^{\by, s}(t) \leq e^{-\gamma_bt}s
  +K_1, \,\, (\by,s)\in \Delta\times(0,\infty), t\geq0.
\end{equation}
Moreover, for any $H>0, T>0$, and $\eps>0$, there is a $\tilde k=\tilde k(H, T, \eps)>0$ such that
\begin{equation}\label{lm2.1b}
\PP\{S^{\by, s}(t)<\tilde k,\, t\in[0,T]\}>1-\eps, (\by,s)\in\Delta\times(0,H].
\end{equation}
\end{lem}

\begin{proof}
In view of \eqref{e:b}, if $s\geq M_b$ then
 $-[\bb(s\by)]^\top \by+\ba^\top\by +\gamma_b\leq 0$.
Let
$$\tilde K_1=\sup\limits_{\by\in\Delta, s\leq M_b} \left\{s(-[\bb(s\by)]^\top \by+\ba^\top\by +\gamma_b)\right\}<\infty.$$
   For $k\in\N$, define the bounded stopping time
\begin{equation}\label{e.eta}
  \eta^{\by, s}_k=\inf\{t\geq 0: S^{\by, s}(t)\geq k\}.
\end{equation}
Dynkin's formula for the function $f(t,s):= e^{\gamma_b t} s$ and the bounded stopping time $t\wedge \eta^{\by, s}_k$ yield
\begin{equation}\label{e2.20}
  \begin{aligned}
  \E[e^{\gamma_bt\wedge\eta^{\by, s}_k}&S^{\by, s}(t\wedge\eta^{\by, s}_k)]\\
  &= s+\E\int_0^{t\wedge\eta^{\by, s}_k}e^{\gamma_bu}S^{\by, s}(u)\left(\gamma_b+[\ba -\bb(S^{\by, s}(u)\BY^{\by,s}(u)]^\top\BY^{\by,s}(u)\right)du
  \\
  &\leq s+\E\int_0^{t\wedge\eta^{\by, s}_k}\tilde K_1e^{\gamma_bu}du\leq s+\dfrac{\tilde K_1}{\gamma_b}(e^{\gamma_bt}-1).
  \end{aligned}
\end{equation}
The claim \eqref{lm2.1b} follows directly from \eqref{e2.20}.
Moreover, by letting $k\to \infty$ in \eqref{e2.20}, we obtain from Fatou's lemma that
\begin{equation}
  \E e^{\gamma_bt}S^{\by, s}(t) \leq s
  +K_1e^{\gamma_bt} \text{ for } K_1=\frac{\tilde K_1}{\gamma_b},
\end{equation}
which implies \eqref{lm2.1a}. \end{proof}

\begin{prop}\label{prop3.2}
For any $\eps>0$ and $T>0$, there is a $\delta=\delta(\eps, T)>0$ such that
$$\PP\left\{\|(\BY^{\by,s}(t), S^{\by, s}(t))-(\tilde \BY^{\by}(t),0) \|\leq\eps, \, 0\leq t\leq T\right\}>1-\eps$$
 given that
$(\by,s)\in\Delta\times(0,\delta).$
\end{prop}

\begin{proof}
In view of \eqref{lm2.1b},
 for any $\eps>0$, $T>0$, there is $\tilde k=\tilde k(\eps,  T)>0$ such that
\begin{equation}\label{e4.14}
\PP\{\eta_{\tilde k}^{\by,s}<T\}\geq 1-\dfrac{\eps}3, \, (\by,s)\in\Delta\times(0,\eps)
\end{equation}
where $\eta^{\by, s}_k$ is defined by \eqref{e.eta}.
Since the coefficients of Equation \eqref{e4.1} are locally Lipschitz,
using the arguments from \cite[Lemma 9.4]{MAO} and noting
$ S^{\by,0}(t)\equiv0$,  we obtain for any $(\by, s)\in\Delta\times(0,\eps)$ that
\begin{equation}\label{e4.15}
\E\left(\sup_{0\leq t\leq T\wedge\eta_{\tilde k}^{\by,s}\wedge\eta_{\tilde k}^{\by,0}}\left\|\left(\BY^{\by,s}(t), S^{\by,s}(t)\right)-\left(\BY^{\by,0}(t), 0\right)\right\|^2\right)\leq Cs^2,
\end{equation}
where $C$ is a constant that depends on $H, T,\tilde k$.
Applying Chebyshev's inequality to \eqref{e4.15},
there is a $\delta\in(0,\eps)$ such that for all $(\by, s)\in \Delta\times(0,\delta)$
\begin{equation}\label{e4.16}
\PP\left\{\sup_{0\leq t\leq T\wedge\eta_{\tilde k}^{\by,s}\wedge\eta_{\tilde k}^{\by,0}}\left\|\left(\BY^{\by,s}(t), S^{\by,s}(t)\right)-\left(\BY^{\by,0}(t), 0\right)\right\|<\eps\right\}>1-\dfrac\eps3.
\end{equation}
Combining \eqref{e4.15} and \eqref{e4.16} yields
\begin{equation}\label{e4.17}
\PP\left\{\sup_{0\leq t\leq T}\left\|\left(\BY^{\by,s}(t), S^{\by,s}(t)\right)-\left(\BY^{\by,0}(t), 0\right)\right\|<\eps\right\}>1-\eps.
\end{equation}
for any $(\by, s)\in \Delta\times(0,\delta)$.
The desired result is obtained by noting that $\BY^{\by,0}(t)=\tilde\BY^{\by}(t), t\geq 0$.
\end{proof}

\begin{lem}\label{lem2.2}
  There are positive constants $K_2$ and  $K_3$ such that for any $(\by, s)\in\Delta\times(0,\infty), T\geq0$, one has
  \begin{equation}\label{e2.21}
    \E \left([\ln S^{\by,s}(T)]^2\right)\leq ((\ln s)^2+1)K_2\exp\{K_3T\},
  \end{equation}
\end{lem}

\begin{proof}
In view of It\^o's formula,
\begin{equation}\label{e2.22}
\begin{aligned}
d\ln^2 S(t)=&\left({\BY(t)}^\top\Sigma{\BY(t)}+2\ln S(t)\left(\ba^\top{\BY(t)}-{[\bb(S(t)\BY(t))]}^\top\BY(t)-\frac12{\BY(t)}^\top\Sigma{\BY(t)}\right)\right)dt\\
&\qquad+2\ln S(t){\BY(t)}^\top d\BE(t).
\end{aligned}
\end{equation}
Now, we estimate $g(\by, s)=\by^\top\Sigma\by+2\ln s\left(\ba^\top{\by}-\bb(s\by)-\frac12{\by}^\top\Sigma{\by}\right)$ for $(\by,s)\in\Delta\times(0,\infty)$.
Let $M_b$ be as in \eqref{e:b}.
If $s>M_b$ then $\ln s>0$ and $\left(\ba^\top{\by}-\bb(s\by)-\frac12{\by}^\top\Sigma{\by}\right)<0$.
Letting $$M_1:=\sup\limits_{\{(\by, s)\in\Delta\times(0,M_b]\}}\left\{\left|\left(\ba^\top{\by}-\bb(s\by)-\frac12{\by}^\top\Sigma{\by}\right)\right|\right\}<\infty$$
and
$$\|\Sigma\|:=\sup\{\by^\top\Sigma\by:\by\in\Delta\},$$
$$g(\by, s)\leq \|\Sigma\|+M_1|\ln s|\leq M_1\ln^2 s+2M_1+\|\Sigma\|\,~\text{for all}~ (\by, s)\in\Delta\times(0,\infty).$$
With this estimate, we can apply Dynkin's formula to \eqref{e2.22}
and use standard arguments (e.g., \cite[Theorem 2.4.1]{MAO}) to obtain
$$    \E \left([\ln S^{\by,s}(T)]^2\1_{A}\right)\leq K_2(\ln s)^2\exp\{K_3T\}\,~\text{for all}~ (\by, s)\in\Delta\times(0,\infty)$$
for some positive constants $K_2$ and  $K_3$.
\end{proof}

\begin{lem}\label{lem2.4}
  There is a positive constant $K_4$ such that for any $(\by, s)\in\Delta\times(0,1)$, and $A\in\F$,
  \begin{equation}\label{e2.1}
    \E \left([\ln S^{\by,s}(T\wedge \zeta^{\by, s})]_-^2\right)\leq (\ln s)^2+K_4\sqrt{\PP(A)}(T+1)[\ln s]_-+K_4T^2,
  \end{equation}
  where
  $[\ln x]_-:=\max\{0, -\ln x\},$
and
\begin{equation}\label{ezeta}
\zeta^{\by, s}:=\inf\{t\geq0:  S^{\by,s}(t)=1\}.
\end{equation}
\end{lem}

\begin{proof}
Let $$M_2=\sup\limits_{\{(\by,s)\in\Delta\times(0,1)\}}\left\{(-\ba^\top{\bf y}+\frac12{\bf y}^\top\Sigma{\bf y}+\bb(s\by)^\top\by\right\}<\infty.$$
Using Dynkin's formula,
\begin{equation}\label{e2.3}
  \begin{aligned}
&-\ln S^{\by, s}(T\wedge \zeta^{\by, s})\\
& \quad = -\ln s -M^{\by, s}(T\wedge \zeta^{\by, s})\\
&\qquad +\int_0^{T\wedge \zeta^{\by, s}}\left(-\ba^\top{\BY^{\by, s}(t)}+\bb(S^{\by, s}(t)\BY^{\by, s}(t))^T\BY^{\by, s}(t)+\frac12{\BY^{\by, s}(t)}^\top\Sigma{\BY^{\by, s}(t)}\right)dt
  \\
 &\quad \leq  [\ln s]_-+M_2T+|M^{\by, s}(T\wedge \zeta^{\by, s})|,
  \end{aligned}
\end{equation}
where
\begin{equation}\label{e2.3a}
M^{\by, s}(t)=\int_0^{t} {\BY(t)}^\top d\BE(t)=\int_0^{t} {\BY(t)}^\top \Gamma d\BB(t).
\end{equation}
It follows from \eqref{e2.3} that
\begin{equation}\label{e2.4}
  \begin{aligned}
\text{}[\ln S^{\by, s}(T\wedge \zeta^{\by, s})]_-^2\1_A\leq& \left([\ln s]_-+M_2T+|M^{\by, s}(T\wedge \zeta^{\by, s})|\right)^2\1_A \\
    \leq& \left([\ln s]_-^2+2(M_2T+|M^{\by, s}(T\wedge \zeta^{\by, s})|)[\ln s]_-\right)\1_A\\
    &+\left(2(M_2T)^2+2|M^{\by, s}(T\wedge \zeta^{\by, s})|^2\right)\1_A
  \end{aligned}
\end{equation}
An application of It\^o's isometry yields
\begin{equation}\label{e:Iso}
\begin{aligned}
\E[|M_{z,y}(T\wedge \zeta^{\by, s})|^2\1_A]\leq \E|M_{z,y}(T\wedge \zeta^{\by, s})|^2\leq\|\Sigma\|T.
\end{aligned}
\end{equation}
By a straightforward use of H\"{o}lder's inequality and \eqref{e:Iso},
\begin{equation}\label{e:Hold}
\begin{aligned}
\E[|M_{z,y}(T\wedge \zeta^{\by, s})|\1_A]\leq&\Big(\PP(A)\E |M_{z,y}(T\wedge \zeta^{\by, s})|^2\Big)^{1/2}\\
\leq&\sqrt{\PP(A)}\sqrt{\|\Sigma\|T}
\leq\sqrt{\PP(A)\|\Sigma\|}(T+1).
\end{aligned}
\end{equation}
Taking expectation on both sides of \eqref{e2.4} and using the estimates from \eqref{e:Iso} and \eqref{e:Hold}, we have
  \begin{equation*}
   \E \left[[\ln S^{\by, s}(T\wedge \zeta^{\by, s})]_-^2\1_A\right]\leq[\ln s]_-^2\PP(A)+K_4(T+1)\sqrt{\PP(A)}[\ln s]_-+K_4T^2,
     \end{equation*}
for some positive constant $K_4.$
\end{proof}
Let $M_3$ be a positive constant such that
\begin{equation}\label{e2.24}
\left|\ba^\top{\bf y}-\frac12{\bf y}^\top\Sigma{\bf y}-\ba^\top{\bf y'}-\frac12{\bf y'}^\top\Sigma{\bf y'}\right|\leq M_3\|\by'-\by|,\, \by, \by'\in\Delta.
\end{equation}
From now on, we assume that $\eps\in(0,1)$ is chosen small enough to satisfy the following
\begin{equation}\label{e2.5}
\begin{split}
\left(M_3+2\right)\eps+\sup\limits_{\{0\leq s\leq\eps, \by\in\Delta\}}\{\bb(s\by)^\top \by\}&<\dfrac{ r}4\\
-\frac{3 r}2(1-3\eps)+2K_4\sqrt{\eps}&<- r
\end{split}
\end{equation}

\begin{lem}\label{lem2.5}
  For $\eps$ satisfying \eqref{e2.5}, there is $\delta(\eps)=\delta\in(0,1)$ and $T^*(\eps)=T^*>1$ such that
  \begin{equation}
    \PP\left\{\ln s+\frac{3 r T^*}{4}\leq \ln S^{\by,s}(T^*)<0\right\}\geq 1-3\eps
  \end{equation}
  for all $(\by, s)\in\Delta\times(0,\delta).$
\end{lem}

\begin{proof}
Since $\Delta$ is a petite set of $\{\tilde\BY(t):t\geq0\}$, in view of \cite[Theorem 6.1]{MT}, there are $\gamma_1$ and $\gamma_2>0$ such that
\begin{equation}\label{e2.23}
\|\tilde P(t, \by,\cdot)-\nu^*\|_{TV}\leq \gamma_1\exp(-\gamma_2t),\, \by\in\Delta, t\in[0,\infty).
\end{equation}
where $\tilde P(t,\by,\cdot)$ is the transition probability of $\{\tilde\BY(t):t\geq0\}$.
Let $M_4=\max\limits_{\by\in\Delta}\{|\ba^\top{\bf y}-\frac12{\bf y}^\top\Sigma{\bf y}|\}<\infty.$
In view of \eqref{lambda} and \eqref{e2.23}, we have
\begin{equation}\label{e2.23a}
\begin{aligned}
\dfrac1T\E&\left|\int_0^T\left(\ba^\top\tilde\BY^{\by}(t)-\frac12{\tilde\BY^{\by}(t)}^\top\Sigma{\tilde\BY^{\by}(t)}\right)dt- r T\right|\\
&\leq\dfrac1T\int_0^T\int_\Delta\left|\left(\ba^\top{\bf y'}-\frac12{\bf y'}^\top\Sigma{\bf y'}\right)\left(\tilde P(t,\by,d\by')-\nu^*(d\by')\right)\right|\\
&\leq \dfrac{M_4}T\int_0^T\|\tilde P(t, \by,\cdot)-\nu^*\|_{TV}dt
\leq \dfrac{M_4\gamma_1}T.
\end{aligned}
\end{equation}
On one hand, letting $M^{\by, s}(T)$ be defined as \eqref{e2.3a}, we have from It\^o's isometry that
\begin{equation}\label{e2.23b}
\E\left[\dfrac{M^{\by, s}(T)}T\right]^2=\dfrac{1}{T^2}\E\int_0^T{\BY^{\by,s}(t)}^\top\Sigma \BY^{\by,s}(t)dt\leq\dfrac{\|\Sigma\|}T.
\end{equation}
With standard estimation techniques,
it follows from \eqref{e2.23a} and \eqref{e2.23b} that for any $\eps>0$, there is a $T^*=T^*(\eps)$ such that
\begin{equation}\label{e2.25}
\PP\left\{\left|\dfrac1{T^*}\int_0^{T^*}\left(\ba^\top\tilde\BY^{\by}(t)-\frac12{\tilde\BY^{\by}(t)}^\top
\Sigma{\tilde\BY^{\by}(t)}\right)dt- r\right|<\eps\right\}>1-\eps,\, \by\in\Delta,
\end{equation}
and
\begin{equation}\label{e2.25a}
\PP\left\{\left|\dfrac{M^{\by, s}(T^*)}{T^*}\right|<\eps\right\}>1-\eps,\, (\by,s)\in\Delta\times (0,\infty).
\end{equation}

By virtue of Proposition \ref{prop3.2}, \eqref{e2.24}, and \eqref{e2.25},
there is $\delta=\delta(\eps, T^*)\in(0,\eps)$ such that
$$\PP(\Omega_1^{\by, s})>1-2\eps,\, (\by,s)\in\Delta\times(0,\delta)$$
where
$$\begin{array}{ll}\Omega_1^{\by, s}&\!\!\!\disp :=\left\{\int_0^{T^*}\left(\ba^\top\BY^{\by,s}(t)-\frac12{\BY^{\by,s}(t)}^\top\Sigma{\BY^{\by,s}(t)}\right)dt>T^*( r-(M_3+1)\eps)\right\}\\
&\disp \cap\left\{S^{\by, s}(t)<\eps,\, t\in[0,T^*]\right\}.\end{array}$$
Using
$\by^\top\bb(s\by)<\frac r4\,~\text{for all}~ (\by, s)\in\Delta\times(0,\eps)$ from \eqref{e2.5} we have that on the set \newline $\Omega_2^{\by, s}:=\Omega_1^{\by, s}\bigcap\left\{\left|\dfrac{M^{\by, s}(T)}T\right|<\eps\right\}$ the following holds
\begin{equation}\label{e2.26}
  \begin{aligned}
0>\ln\eps\geq\ln S^{\by, s}(T^*)=&\ln s+M^{\by, s}(T^*)-\int_0^{T^*} {\BY^{\by, s}(t)}^\top\bb(S^{\by, s}(t)\BY^{\by, s}(t))dt\\
&+\int_0^{T^*}\left(\ba^\top{\BY^{\by, s}(t)}-\frac12{\BY^{\by, s}(t)}^\top\Sigma{\BY^{\by, s}(t)}\right)dt
  \\\geq& \ln s+\left( r-(M_3+2)\eps-\sup\limits_{\{0\leq s\leq\eps, \by\in\Delta\}}\{\bb(s\by)^\top \by\}\right)T^*\\
  \geq& \ln s+\dfrac{3 r}4T^*.
  \end{aligned}
\end{equation}
Noting
$$\PP\left(\Omega_2^{\by, s}\right)\geq 1-3\eps\,~\text{for all}~ (\by, s)\in\Delta\times(0,\delta),$$
the proof is complete.
\end{proof}

\begin{prop}\label{prop2.3} Assume $ r>0$.
  Let $\delta$ and $T^*$ be  as in Lemma \ref{lem2.5}. There exists a positive constant $K^*=K^*(\delta, T^*)$ such that
  \begin{equation}
    \E[\ln S^{\by,s}(T^*)]_-^2\leq [\ln s]_-^2- r T^*[\ln s]_-+K^*
  \end{equation}
  for any $(\by, s)\in\Delta\times(0,\infty).$
\end{prop}

\begin{proof}
We look at three cases of the initial data $(\by,s)$.

{\bf Case I: $s\in(0,\delta)$.}
 We have from Lemma \ref{lem2.5} that $\PP(\Omega_2^{\by, s})\geq1-3\eps$ where
 $\Omega_2^{\by, s}$ is defined as in the proof of Lemma \ref{lem2.5}.
  On $\Omega_2^{\by, s}$, we have
  $$-\ln s-\frac{3 r T^*}4\geq-\ln S^{\by,s}(T^*)> 0.$$
  Hence,
  $$0\leq [\ln S^{\by,s}(T^*)]_-\leq [\ln s]_--\frac{3 r T^*}4.$$
 Squaring both sides yields
  $$[\ln S^{\by,s}(T^*)]_-^2\leq [\ln s]_-^2-\frac{3 r T^*}2[\ln s]_-+\frac{9 r^2 {T^*}^2}{16},$$
  which implies that
  \begin{equation}\label{e2.14}
    \E\left[\1_{\Omega_2^{\by, s}}[\ln S^{\by,s}(T^*)]^2_-\right]\leq\PP(\Omega_2^{\by, s})[\ln s]_-^2-\frac{3 r T^*}{2}\PP(\Omega_2^{\by, s})[\ln
    s]_-+\frac{9 r^2 {T^*}^2}{16}\PP(\Omega_2^{\by, s}).
  \end{equation}
 On $\Omega_3^{\by, s}:=\{\zeta^{\by, s}<T^*\}$ with $\zeta^{\by, s}$  defined in \eqref{ezeta},
since $\ln S^{\by,s}(\zeta^{\by, s})=0$, we have from Lemma \ref{lem2.2} and the strong Markov property of $(\BY(t), S(t))$ that
 \begin{equation}\label{e2.15}
    \E\left[\1_{\Omega_3^{\by, s}}[\ln S^{\by,s}(T^*)]^2_-\right]\leq\PP(\Omega_3^{\by, s})K_2\exp(K_3T^*).
 \end{equation}
 On the set $\Omega_4^{\by, s}:=\Omega\setminus(\Omega_2^{\by, s}\cup\Omega_3^{\by, s})$, applying Lemma \ref{lem2.4} and noting that $\zeta^{\by, s}>T^*$ in $\Omega_4^{\by, s}$ and $T^*>1$, we obtain
 \begin{equation}\label{e2.27}
     \E \left[\1_{\Omega_4^{\by, s}}[\ln S^{\by, s}(T^*)]_-^2\right]\leq[\ln s]_-^2\PP(\Omega_4^{\by, s})+2K_4T^*\sqrt{\PP(\Omega_4^{\by, s})}[\ln s]_-+K_4{T^*}^2.
\end{equation}
 Adding \eqref{e2.14}, \eqref{e2.15}, and \eqref{e2.27} side by side, we get
        \begin{equation}\label{e2.16}
        \begin{aligned}
    \E [\ln S^{\by, s}(T^*)]_-^2\leq&[\ln s]_-^2+\Big(-\frac{3 r}2(1-3\eps)+2K_4\sqrt{\eps}\Big)T^*[\ln
    s]_-+K^*_5(T^*)\\
    \leq&[\ln s]_-^2- r T^*[\ln
    s]_-+K^*_5(T^*),
    \end{aligned}
 \end{equation}
 where $K^*_5(T^*)$ is a positive constant independent of $(\by, s)\in\Delta\times(0,\delta)$.

 {\bf Case II:  $s\in [\delta, 1]$}.
 We have from Lemma \ref{lem2.2} that
 \begin{equation}\label{e2.28}
 \begin{aligned}
     \E [\ln S^{\by, s}(T^*)]_-^2\leq&\E [\ln S^{\by, s}(T^*)]^2\leq[\ln s]^2+K_2\exp(K_3T^*)\\
     \leq& ([\ln \delta]^2+1)K_2\exp(K_3T^*).
 \end{aligned}
\end{equation}

 {\bf Case III:  $s\in (1, \infty)$}.
     Note that if $\zeta^{\by, s}>T^*$, then $[\ln S^{\by, s}(T^*)]_-^2=0$. Thus, using Lemma \ref{lem2.2} and the strong Markov property of $(\BY(t), S(t))$ once more, we obtain
 \begin{equation}\label{e2.29}
      \E [\ln S^{\by, s}(T^*)]_-^2=\E\left(\1_{\{\zeta^{\by, s}<T^*\}}[\ln S^{\by, s}(T^*)]_-^2\right)\leq K_2\exp(K_3T^*).
      \end{equation}
Combing \eqref{e2.16}, \eqref{e2.28}, and \eqref{e2.29}, and setting
$K^*=\max\{K^*_5(T^*), ([\ln \delta]^2+1)K_2\exp(K_3T^*)\}$,
the proof is concluded.
\end{proof}

\begin{thm}\label{thm2.2}
Suppose that Assumptions \ref{a:dispersion} and \ref{a:nonsingular} hold and that $ r>0$. Let $P(t, (\by,s), \cdot)$ be the semigroup of the process $((\BY(t),S(t))_{t\geq 0}$. Then, there exists an
invariant probability measure $\pi^*$ of the process  $((\BY(t),S(t))_{t\geq 0}$ on $\Delta\times(0,\infty)$. Moreover, $\pi^*(\Delta^\circ\times(0,\infty))=1$, $\pi^*$ is absolutely continuous with respect to the Lebesgue measure on $\Delta\times(0,\infty)$ and
\begin{equation}\label{etv}
\lim\limits_{t\to\infty} t^{q^*}\|P(t, (\by,s), \cdot)-\pi^*(\cdot)\|_{TV}=0, \;(\by,s)\in\Delta^\circ\times(0,\infty),
\end{equation}
where $\|\cdot\|_{TV}$ is the total variation norm and $q^*$ is any positive number.
In addition, for any initial value $(\by,s)\in\Delta\times(0,\infty)$ and any $\pi^*$-integrable function $f$, we have
\begin{equation}\label{slln}
\PP\left\{\lim\limits_{T\to\infty}\dfrac1T\int_0^Tf\left(\BY^{\by,s}(t),S^{\by,s}(t)\right)dt=\int_{\Delta^\circ\times(0,\infty)}f(\by', s')\pi^*(d\by', s')\right\}=1.
\end{equation}
\end{thm}
\begin{proof}
By virtue of Lemma \ref{lem2.1}, there
is an $h_1:=1-\exp\left(-\gamma_bT^*\right)>0$ satisfying
\begin{equation}\label{e0.3}
\E S^{\by,s}+1\leq s+1-h_1s+K_1\leq s+1-h_1\sqrt{s+1}+K_1+h_1,\,(\by,s)\in\Delta\times(0,\infty).
\end{equation}
Let $V(s)=s+1+[\ln s]_-^2$.
In view of Proposition \ref{prop2.3} and \eqref{e0.3},
\begin{equation}\label{e0.4}
\begin{aligned}
\E V\Big(S^{\by,s}(T^*)\Big)&\leq s+1-h_2\big(\sqrt{s+1}+[\ln s]_-\big)+H_2\\
&\leq V(s)-\dfrac{h_2}2\sqrt{V(s)}+H_2 \text{ for all }(\by,s)\in\Delta\times(0,\infty),
\end{aligned}
\end{equation}
where $h_2=\min\{h_1, r T^*\}, H_2=H_1+h_1+K_1$.
Let $\kappa>1$ such that
\begin{equation}\label{e0.5}
\dfrac{\sqrt{V(s)}}2\geq H_2\,~\text{for all}~ s\notin[\kappa^{-1},\kappa].
\end{equation}
Combining \eqref{e0.4} and \eqref{e0.5}, we arrive at
\begin{equation}\label{e0.6}
\E V\Big(S^{\by,s}(T^*)\Big)\leq V(s)-\dfrac{h_2}4\sqrt{V(s)}+H_2\1_{\{(\by, s)\in \Delta\times[\kappa^{-1},\kappa]\}}\,~\text{for all}~ (\by,s)\in\Delta\times(0,\infty).
\end{equation}
Using the estimate \eqref{e0.6}, Lemma \ref{petite}, and Theorem \ref{polyrate},
the Markov chain $(\BY(kT^*), S(kT^*))_{k\geq 0}$ has a unique invariant probability measure $\pi^*$ and
\begin{equation}\label{e0.7}
k\|P(kT^*,(\by, s),\cdot)-\pi^*\|_{TV}\to 0 \text{ as } k\to\infty.
\end{equation}
As a direct consequence, for fixed $\by_0,s_0$, the family $\{P(kT^*,(\by_0, s_0),\cdot),k\in\N\}$ is tight, that is, for any
$\theta>0$, there is a compact set $K_\theta\subset\Delta\times(0,\infty)$ such that
\begin{equation}\label{e0.8}
P(kT^*,(\by_0, s_0),K_\theta)>1-\theta\,~\text{for all}~ k\in\N.
\end{equation}
Since $s^2+\ln^2 s\to \infty$ as $s\to0$ or $s\to\infty$,
in view of Lemmas \ref{lem2.1} and \ref{lem2.2} and a standard estimate,
there is  a $\kappa_\theta>1$ such that
$$\PP\left\{S^{\by,s}(t)\in[\kappa_\theta^{-1},\kappa_\theta]\right\}>1-\theta,~\text{for all}~ (\by,s)\in K_\theta, t\in[0,T^*],$$
or equivalently,
\begin{equation}\label{e0.9}
P\left(t,(\by,s), \Delta\times[\kappa_\theta^{-1},\kappa_\theta]\right)>1-\theta\,~\text{for all}~ (\by,s)\in K_\theta, t\in[0,T^*].
\end{equation}
Using the Chapman-Kolmogorov relation together with \eqref{e0.8} and \eqref{e0.9} yields
$$P\left(u,(\by_0,s_0), \Delta\times[\kappa_\theta^{-1},\kappa_\theta]\right)>1-2\theta\,~\text{for all}~ u\geq0,$$
which implies that the family of empirical measures
$\left\{\frac1T\int_0^TP(u,(\by_0,s_0),\cdot)du, T>0\right\}$ is tight in $\Delta\times(0,\infty)$.
Thus $(\BY(t), S(t))$ has an invariant probability
measure $\pi_*$ on $\Delta\times(0,\infty)$ (see e.g., \cite[Proposition 6.4]{EHS15}). As a result, the Markov chain $(\BY(kT^*), S(kT^*))_{k\in\N}$
has an
invariant probability measure
$\pi_*$.
In view of \eqref{e0.7}, $\pi_*$ must coincide with $\pi^*$.
Thus, $\pi^*$ is an invariant probability measure of the process $(\BY(t), S(t))_{t\geq 0}$ on $\Delta\times(0,\infty)$.

In the proofs, we used the function $[\ln s]_-^2$ for the sake of simplicity.
In fact, we can treat $[\ln s]_-^{1+q}$ for any small $q\in(0,1)$ in the same manner.
We can show that there are $h_q$, $H_q>0$, and a compact set $K_q\subset\Delta\times(0,\infty)$ satisfying
\begin{equation}\label{e0.10}
\E V_q(S^{\by, s}(T^*))\leq V_q(s)-h_q[V_q(s)]^{\frac1{1+q}}+H_q\1_{\{(\by, s)\in K_q\}},\, (\by,
s)\in\Delta\times(0,\infty),
\end{equation}
where $V_q(s):=s+1+[\ln s]_-^{1+q}.$
Then applying Theorem \ref{polyrate}, we
obtain
\begin{equation}\label{e0.11}
k^{1/q}\|P(kT^*,(\by, s),\cdot)-\pi^*\|\to 0 \text{ as } k\to\infty.
\end{equation}
Let $f:\Delta\times(0,\infty)\mapsto [-1,1]$ be a measurable function.
Since $\pi^*$ is an invariant measure, then for any $u\geq0$,
$$\int_{\Delta\times(0,\infty)}f(\by',s')\pi^*(d\by',ds')=\int_{\Delta\times(0,\infty)}\pi^*(d\by_1,ds_1)\int_{\Delta\times(0,\infty)}P(u, (\by_1,s_1), (\by',ds'))f(\by',s').$$
Using this equality and the Chapman-Kolmogorov equation, we have
$$
\begin{aligned}
|f&(\by',s')(P(t+u,(\by,s),d\by',ds')-\pi^*(d\by',ds')|\\
=&\bigg|\int_{\Delta\times(0,\infty)}\left(P(t, (\by,s), d\by_1, ds_1)-\pi^*(d\by_1,ds_1)\right)\\
&\qquad\qquad\times\int_{\Delta\times(0,\infty)}(f(\by',s')P(t,(\by_1,s_1),(d\by',ds'))\bigg|\\
\leq&\|P(t,(\by,s),\cdot)-\pi^*\|_{TV}\\
&\left(\text{ since } \Big|\int_{\Delta\times(0,\infty)}(f(\by',s')P(t,(\by_1,s_1),(d\by',ds')\Big|\leq1\,~\text{for all}~ \by_1,s_1\right),
\end{aligned}
$$
which means that $\|P(t,(\by,s),\cdot)-\pi^*\|_{TV}$ is decreasing in $t$.
As a result, we
deduce from \eqref{e0.10} that
$$t^{q^*}\|P(t,(\by,s),\cdot)-\pi^*\|_{TV}\to 0 \text{ as } t\to\infty,$$
where $q^*=1/q\in(1,\infty)$.

In view of Proposition \ref{e:solutions}, for any $t>0$, $\PP\{\BY^{\by,s}(t)\in\Delta^\circ\}=1$.
Thus, $$\pi^*(\Delta^\circ\times(0,\infty))=\int_{\Delta\times(0,\infty)}\PP\{\BY^{\by,s}(t)\in\Delta^\circ\}\pi^*(d\by,ds)=\pi^*(\Delta\times(0,\infty))=1.$$
By \cite[Theorem 20.17]{K02}, our process $(\BY(t),S(t))_{t\geq 0}$  is either Harris recurrent or uniformly transient on $\Delta^\circ\times(0,\infty)$. Using \cite[Theorem 20.21]{K02},
our process cannot be uniformly transient and also have an invariant probability measure. Therefore, our process is Harris recurrent. \cite[Theorem 20.17]{K02} further indicates
that any Harris recurrent Feller process on $\Delta^\circ\times(0,\infty)$ with strictly positive transition densities has a locally finite invariant measure that is equivalent to Lebesgue measure and is unique up to normalization. Since we already know that $(\BY(t),S(t))_{t\geq 0}$ has a unique invariant probability measure, this probability measure has an almost everywhere strictly positive density with respect to the Lebesgue measure.
\end{proof}

\section{The case $ r<0$}\label{sec:general-}

\begin{thm}
Suppose that $ r<0$.
Then for any $i=1,\dots,n$ and any $\mathbf{x} = (x_1,\dots,x_n)\in \R_+^{n}$,
\begin{equation}
\PP\left\{\lim_{t\to\infty}\frac{\ln {X}_i^{\mathbf{x}}(t)}{t}= r\right\}=1.
\end{equation}
In particular, for any $i=1,\dots,n$ and any $\mathbf{x} = (x_1,\dots,x_n)\in \R_+^{n}$
\[
\PP\left\{\lim_{t\to\infty}X^{\mathbf{x}}_i(t)=0\right\}=1.
\]

\end{thm}

\begin{proof}
Let $\theta>0$ and $\check a_i=a_i+\theta$, and
define the process ${\check\BX}^{\mathbf{x}}(t) = ({\check X}_1^{\mathbf{x}}(t),\dots,{\check X}_n^{\mathbf{x}}(t))$ as the solution to
\begin{equation}\label{e:X0}
d {{\check X}_i}(t)=\left({\check X}_i(t)(\check a_i)+\sum_{j=1}^n D_{ji} {{\check X}_j}(t)\right)dt+{\check X}_i(t)dE_i(t), \, i=1,\dots,n
\end{equation}
started at $\mathbf{x} = (x_1,\dots,x_n)\in \R_+^{n}$.
Letting $\check S(t)=\sum\check X_i(t)$ and $\check\BY(t)=\dfrac{\check\BX(t)}{S(t)}$, we have
\begin{equation}\label{e.checkys}
\begin{split}
d\check \BY(t)=&\left(\diag(\check \BY(t))-\check \BY(t)\check \BY^\top(t)\right)\Gamma^\top d\BB(t)\\
&+\BD^\top\check \BY(t)dt+\left(\diag(\check \BY(t))-\check \BY(t)\check \BY^\top(t)\right)(\check\ba-\Sigma \check \BY(t))dt\\
d\ln\check S(t)=&\left(\check\ba^\top\check\BY(t)-\frac12{\check\BY(t)}^\top\Sigma{\check\BY(t)}\right)dt+{\check \BY(t)}^\top \Gamma^\top d\BB(t)
\end{split}
\end{equation}
Let $(\check\BY^{\by}(t), \check S^{\by, s}(t))$ be the solution to \eqref{e.checkys} with initial condition $(\by, s)$. Note that $\check\BY^{\by}(t)$ does not depend on $s$.
First, fix $\by_0\in\Delta$.
We have that
\begin{equation}\label{clambda}
\begin{aligned}
\lim\limits_{t\to\infty}\dfrac1t&\left(\int_0^t\left(\check\ba^\top\check\BY^{\by_0}(u)-\frac12{\check\BY^{\by_0}(u)}^\top\Sigma{\check\BY^{\by_0}(u)}\right)du+\int_0^t{\check \BY^{\by_0}(u)}^\top \Gamma^\top d\BB(u)\right)\\
=&\check r:=\int_{\Delta}\left(\check\ba^\top{\bf y}-\frac12{\bf y}^\top\Sigma{\bf y}\right)\check\nu^*(d{\bf y}), \PP-\text{a.s.,}
\end{aligned}
\end{equation}
where $\check\nu^*$ is the unique invariant probability measure of $(\check\BY(t))_{t\geq0}$.
By the continuous dependence of $ r$ on the coefficients (established in the Proposition \ref{p:cts_dep}), there is $\theta>0$ such that
$\check r<\frac{ r}2<0$.
Let $\delta>0$ such that
$\sup\{-b_i(x):x<\delta, i=1,\dots,n\}<\theta$ (this is possible since the $b_i$'s are continuous and vanish at $0$).
Because $\check r<0$, it follows from \eqref{clambda} that
$$\sup\limits_{t\in[0,\infty)}\left(\int_0^t\left(\check\ba^\top\check\BY^{\by_0}(s)-\frac12{\check\BY^{\by_0}(s)}^\top\Sigma{\check\BY^{\by_0}(s)}\right)ds+\int_0^t{\check \BY^{\by_0}(s)}^\top \Gamma^\top d\BB(s)\right)<\infty\,\,\,\PP-\text{a.s.}.$$
As a result, for any  $\eps>0$, there is an
$H_\eps>0$ satisfying
$$
\PP\left\{\sup\limits_{t\in[0,\infty)}\left(\int_0^t\left(\check\ba^\top\check\BY^{\by_0}(u)-\frac12{\check\BY^{\by_0}(u)}^\top\Sigma{\check\BY^{\by_0}(u)}\right)du+\int_0^t{\check \BY^{\by_0}(u)}^\top \Gamma^\top d\BB(u)\right)<H_\eps\right\}>1-\eps,
$$
which combined with \eqref{e.checkys} implies that
\begin{equation}\label{e:delta}
\PP\left\{\sup\limits_{t\in[0,\infty)}\check S^{\by_0,s_0}(t)<\delta\right\}>1-\eps\,\text{ if }\,s_0<\delta\exp(-H_\eps).
\end{equation}
Then, a comparison argument shows (see Remark \ref{r:comparison}) that for $\bx_0=s_0\by_0\in \R_+^{n}$ and $i=1,\dots,n$
\begin{equation}\label{e:comparison}
\PP\left\{X^{\mathbf{x_0}}_i(t)\leq {\check X}^{\mathbf{x_0}}_i(t),\, t\in[0,\xi^{\bx_0})\right\}=1
\end{equation}
where $\xi^{\bx_0}=\inf\{t\geq0: \sum_{i=1}^n{\check X}^{\mathbf{x_0}}_i(t)\geq\delta\}$.
By virtue of \eqref{e:delta}, $\PP\{\xi^{\bx_0}=\infty\}>1-\eps$ if $s_0<\delta\exp(-H_\eps)$.
Using \eqref{clambda} and  \eqref{e:comparison} yields that \begin{equation}\label{e:evans0}
\PP\left\{\limsup_{t\to\infty}\frac{\ln S^{\by_0,s_0}}{t}\leq \check r<0\right\}>1-\eps\text{ if } s<\delta\exp(-H_\eps).
\end{equation}
Thus, the process $(\BY(t), S(t))_{t\geq0}$ is not a recurrent diffusion process in $\Delta\times(0,\infty)$.
Hence, it must be transient with probability 1, that is, for any compact $K\in(0,\infty)$ and any initial value $(\by,s)\in\Delta\times(0,\infty)$ we have
\begin{equation}\label{e:tran}
\PP\left\{\lim\limits_{t\to\infty}\1_{\{S^{\by,s}(t)\in K\}}=0\right\}=1.
\end{equation}
In view of Lemma \ref{lem2.1},
\begin{equation}\label{e:infty}
\PP\left\{\lim\limits_{t\to\infty}S^{\by,s}(t)=\infty\right\}=0.
\end{equation}
It follows from \eqref{e:tran} and \eqref{e:infty}
 that $\PP\left\{\lim\limits_{t\to\infty}S^{\by,s}(t)=0\right\}=1$ for any $(\by,s)\in\Delta\times(0,\infty)$.
Moreover, since $(\tilde\BY(t))_{\{t\geq0\}}$ has a unique invariant probability measure $\nu^*$,
on the boundary $\Delta\times\{0\}$, $(\BY(t), S(t))$ has a unique invariant probability measure
$\nu^*\times\delta_0^*$, where $\delta_0^*$ is the Dirac measure concentrated on $\{0\}$.
Fix  $(\by,s)\in \Delta\times(0,\infty)$, and define the normalized occupation measures,
$$\Pi_t(\cdot)=\dfrac1t\int_0^t\1_{\{(\BY^{\by,s}(u),S^{\by,s}(u))\in\cdot\}}du.$$
Since $\PP\left\{\lim\limits_{t\to\infty}S^{\by,s}(t)=0\right\}=1$,
 the family $\left\{\Pi_k(\cdot),k\in\N\right\}$
is tight in the space $\Delta\times[0,\infty)$ for almost every $\omega$.
In view of the proofs of \cite[Theorem 4.2]{EHS15} or \cite[Theorems 4, 5]{SBA11},
the set of $\text{weak}^*$ limit points of $\{\Pi_k,k\in\N\}$ is a nonempty set of invariant probability measures of the process $(\BY(t), S(t))$.
As pointed out above, the process $(\BY(t), S(t))$ has only one invariant probability measure, namely, $\nu^*\times\delta_0^*$.
Thus, for almost every $\omega\in\Omega$, $\{\Pi_k(\cdot),k\in\N\}$ converges weakly to $\nu^*\times\delta_0^*$
as $k\to\infty$.
As a result, for any bounded continuous function $g(\cdot,\cdot): \Delta\times[0,\infty)\mapsto\R$
we have
$\lim\limits_{k\to\infty}\frac1k\int_0^kg(\BY^{\by,s}(t),S^{\by,s}(t))dt=\int_\Delta g(\by',0)\nu^*(d\by')\,\,\PP\text{-a.s}.$
Since $g(\cdot,\cdot)$ is bounded, we easily obtain
\begin{equation}\label{e3.0-}
\lim\limits_{T\to\infty}\frac1T\int_0^Tg(\BY^{\by,s}(t),S^{\by,s}(t))dt=\int_\Delta g(\by',0)\nu^*(d\by')\,\,\PP\text{-a.s}.
\end{equation}
Consequently,
\begin{equation}\label{e3.1-}
\lim\limits_{T\to\infty}\frac1T\int_0^T\left(\ba^\top\BY^{\by,s}(t)-\frac12{\BY^{\by,s}(t)}^\top\Sigma{\BY^{\by,s}(t)}\right)dt= r\,\,\PP\text{-a.s}
\end{equation}
Since $\PP\left\{\lim\limits_{t\to\infty}S^{\by,s}(t)=0\right\}=1$ and $b_i(0)=0,i=1,\dots,n$, we have by Dominated Convergence that
\begin{equation}\label{e3.2-}
\lim\limits_{T\to\infty}\frac1T\int_0^{T} {\BY^{\by, s}(t)}^\top\bb(S^{\by, s}(t)\BY^{\by, s}(t))dt=0\,\,\PP\text{-a.s.}
\end{equation}
Applying the strong law of large numbers for martingales to the process $(M^{\by,s}(t))_{t\geq 0}$ defined by \eqref{e2.3a}, we deduce
\begin{equation}\label{e3.3-}
\lim\limits_{T\to\infty}\frac{M^{\by,s}(T)}T=0\,\,\PP\text{-a.s.}
\end{equation}

Note that
\begin{equation}\label{e3.4-}
  \begin{aligned}
\frac{\ln S^{\by, s}(T)}{T}=&\frac{\ln s}T+\frac{M^{\by, s}(T)}T-\frac1T\int_0^{T} {\BY^{\by, s}(t)}^\top\bb(S^{\by, s}(t)\BY^{\by, s}(t))dt\\
&+\dfrac1T\int_0^{T}\left(\ba^\top{\BY^{\by, s}(t)}-\frac12{\BY^{\by, s}(t)}^\top\Sigma{\BY^{\by, s}(t)}\right)dt
    \end{aligned}
\end{equation}
Applying \eqref{e3.1-}, \eqref{e3.2-}, and \eqref{e3.3-} to \eqref{e3.4-}, we
obtain
\begin{equation}\label{e3.5-}
\lim\limits_{T\to\infty}
\frac{\ln S^{\by, s}(T)}{T}= r,\,\,\PP\text{-a.s.}
\end{equation}
In light of \eqref{e3.5-}, to derive
$\PP\left\{\lim\limits_{T\to\infty}
\frac{\ln X^\bx_i(T)}{T}= r\right\}=1$, it suffices to show
$\PP\left\{\lim\limits_{T\to\infty}
\frac{\ln Y_i^{\by, s}(T)}{T}=0\right\}\!=1$ for each $i=1,\dots,n$.
In view of It\^o's lemma,
\begin{equation}\label{e3.6-}
\begin{split}
\dfrac{\ln Y_i^{\by, s}(T)}T=&\dfrac{\ln y_i}T+\dfrac1T\int_0^T\left(a_i-\sum_{j=1}^na_jY_j^{\by, s}(t)-D_{ii}-\frac{\sigma_{ii}}2+\sum_{j,k=1}^n\frac{\sigma_{kj}}2Y_k^{\by, s}(t)Y_j^{\by, s}(t))\right)dt\\
&+\dfrac1T\int_0^T\left(-b_i(S^{\by, s}(t)Y_i^{\by, s}(t))+\sum_{j=1}^nY_j^{\by, s}(t)b_j(S^{\by, s}(t)Y_j^{\by, s}(t))\right)dt\\
&+ \dfrac1T\int_0^T\left(\sum_{j=1,j\ne i}^{n}D_{ji}\dfrac{Y_j^{\by, s}(t)}{Y_i^{\by, s}(t)}\right)dt +\dfrac1T\int_0^T\left[dE_i(t)-\sum_{j=1}^n Y_j^{\by, s}(t)dE_j(t)\right],
\end{split}
\end{equation}
and
\begin{equation}\label{e3.7-}
\begin{split}
\dfrac{\ln\tilde Y_i^\by(T)}T=&\dfrac{\ln y_i}T+\dfrac1T\int_0^T\left(a_i-\sum_{j=1}^na_j\tilde Y_j^\by(t)-D_{ii}-\frac{\sigma_{ii}}2+\sum_{j,k=1}^n\frac{\sigma_{kj}}2\tilde Y_k^\by(t)\tilde Y_j^\by(t))\right)dt\\
& +\dfrac1T\int_0^T\left(\sum_{j=1,j\ne i}^{n}D_{ji}\dfrac{\tilde Y_j^\by(t)}{\tilde Y_i^\by(t)}\right)dt +\dfrac1T\int_0^T\left[dE_i(t)-\sum_{j=1}^n\tilde Y_j^\by(t)dE_j(t)\right].
\end{split}
\end{equation}
By the strong laws of large numbers for martingales,
\begin{equation}\label{e3.8-}
\lim\limits_{T\to\infty}\dfrac1T\int_0^T\left[dE_i(t)-\sum_{j=1}^n\tilde Y_j^\by(t)dE_j(t)\right]=0,\,\,\PP\text{-a.s.}
\end{equation}
Let $G_i=\sup\limits_{\by\in\Delta}\left\{\left|a_i-\sum_{j=1}^na_jy_j-D_{ii}-\frac{\sigma_{ii}}2+\sum_{j,k=1}^n\frac{\sigma_{kj}}2y_ky_j\right|\right\}<\infty.$
As a result of \eqref{e3.7-} and \eqref{e3.8-} and the fact that $\limsup\limits_{T\to\infty}\dfrac{\ln\tilde Y_i^\by(T)}T\leq0$ almost surely, we obtain
\begin{equation}\label{e3.9-}
\limsup\limits_{T\to\infty}\dfrac1T\int_0^T\left(\sum_{j=1,j\ne i}^{n}D_{ji}\dfrac{\tilde Y_j^\by(t)}{\tilde Y_i^\by(t)}\right)dt\leq G_i,\,\,\PP\text{-a.s.}
\end{equation}
For any $k>0$, it follows from \eqref{e3.8-} and the strong law of large numbers that
$$
\int_\Delta k\wedge\left(\sum_{j=1,j\ne i}^{n}D_{ji}\dfrac{y_j}{y_i}\right)\nu^*(d\by)
=\lim\limits_{T\to\infty}\dfrac1T\int_0^Tk\wedge\left(\sum_{j=1,j\ne i}^{n}D_{ji}\dfrac{\tilde Y_j^\by(t)}{\tilde Y_i^\by(t)}\right)dt \leq G_i
$$
Letting $k\to\infty$ we have
$$\rho_i:=\int_\Delta \sum_{j=1,j\ne i}^{n}D_{ji}\dfrac{y_j}{y_i}\nu^*(d\by)\leq G_i,$$
which implies
\begin{equation}\label{e3.10-}
\lim\limits_{T\to\infty}\dfrac1T\int_0^T\left(\sum_{j=1,j\ne i}^{n}D_{ji}\dfrac{\tilde Y_j^\by(t)}{\tilde Y_i^\by(t)}\right)dt=\rho_i.
\end{equation}
Using \eqref{e3.8-}, \eqref{e3.10-}, and applying the strong law of large numbers for the process $(\tilde\BY(t))_{t\geq0}$, we arrive at
\begin{equation}\label{e3.11-}
\begin{split}
\lim\limits_{T\to\infty}\dfrac{\ln\tilde Y_i^\by(T)}T
=\beta_i+\rho_i\leq0,\,\,\PP\text{}{-a.s.,}
\end{split}
\end{equation}
where $$\beta_i:=\int_\Delta\left(a_i-\sum_{j=1}^na_jy_j-D_{ii}-\frac{\sigma_{ii}}2+\sum_{j,k=1}^n\frac{\sigma_{kj}}2y_ky_j\right)\nu^*(d\by).$$
If $\beta_i+\rho_i<0$, then $\tilde Y_i^\by(T)\to 0$  almost surely as $T\to\infty$,
which contradicts
the fact
that $\tilde\BY^\by(T)$ converges weakly to $\nu^*$
that is concentrated on $\Delta^\circ$.
As a result, $\beta_i+\rho_i=0$.
For any $\theta>0$, there is $k_\theta>0$ such that
$$\int_\Delta k_\theta\wedge\left(\sum_{j=1,j\ne i}^{n}D_{ji}\dfrac{y_j}{y_i}\right)\nu^*(d\by)>\rho_i-\theta.$$
Using \eqref{e3.0-}, we have with probability 1 that
\begin{equation}\label{e3.12-}
\liminf\limits_{T\to\infty}\dfrac1T\int_0^T\left(\sum_{j=1,j\ne i}^{n}D_{ji}\dfrac{Y_j^{\by, s}(t)}{Y_i^{\by, s}(t)}\right)dt
\geq
\lim\limits_{T\to\infty}\dfrac1T\int_0^Tk_\theta\wedge\left(\sum_{j=1,j\ne i}^{n}D_{ji}\dfrac{Y_j^{\by, s}(t)}{Y_i^{\by, s}(t)}\right)dt\geq\rho_i-\theta.
\end{equation}
and
\begin{equation}\label{e3.13-}
\lim\limits_{T\to\infty}\dfrac1T\int_0^T\left(a_i-\sum_{j=1}^na_jY_j^{\by, s}(t)-D_{ii}-\frac{\sigma_{ii}}2+\sum_{j,k=1}^n\frac{\sigma_{kj}}2Y_k^{\by, s}(t)Y_j^{\by, s}(t))\right)dt=\beta_i.
\end{equation}
Applying \eqref{e3.12-}, \eqref{e3.13-}, and the fact $\PP\left\{\lim\limits_{T\to\infty}S^{\by,s}(T)=0\right\}=1$  to \eqref{e3.6-}, we obtain that
$$\liminf\limits_{T\to\infty}\dfrac{\ln Y_i^{\by, s}(T)}T\geq\beta_i+\rho_i-\theta=-\theta,\,\,\PP\text{-a.s.}$$
Since it  holds for any $\theta>0$, we have
$$\lim\limits_{T\to\infty}\dfrac{\ln Y_i^{\by, s}(T)}T=0,\,\,\PP\text{-a.s.}$$
The above equality combined with \eqref{e3.5-} and
$X_i(T)=Y_i(T)S(T)$ yield the desired result.
\end{proof}

\section{Degenerate
diffusion
in $\R^n$} \label{sec:degenerate}

If the correlation matrix $\Sigma$ is degenerate,
the diffusion $\tilde\BY(t)$ from \eqref{eq.by} still has an invariant probability measure $\nu^*$  since it is a Feller-Markov process in a compact set. Moreover, $\nu^*(\Delta^\circ)=1$ because the property that
$\PP\left\{\tilde\BY(t)\in\Delta^\circ,\, t>0\right\}=1$ is satisfied as long as Assumption \ref{a:dispersion} holds, that is, the dispersion matrix $(D_{ij})$ is  irreducible. It is
readily seen
that the following is true.

\begin{thm}\label{t:extinction_degenerate}
Assume that $\tilde\BY(t)$ has a \textbf{unique} invariant probability measure $\nu^*$. Define $ r$ by \eqref{lambda}. Suppose that $ r<0$.
Then for any $i=1,\dots,n$ and any $\mathbf{x} = (x_1,\dots,x_n)\in \R_+^{n}$
\begin{equation}
\PP\left\{\lim_{t\to\infty}\frac{\ln {X}_i^{\mathbf{x}}(t)}{t}=  r\right\}=1.
\end{equation}
In particular, for any $i=1,\dots,n$ and any $\mathbf{x} = (x_1,\dots,x_n)\in \R_+^{n}$
\[
\PP\left\{\lim_{t\to\infty}X^{\mathbf{x}}_i(t)=0\right\}=1.
\]
\end{thm}

\begin{rmk} {\rm
The Markov process $\left\{\tilde\BY(t), t\geq0\right\}$
has a unique invariant probability measure if it is irreducible.
Moreover, since $\PP\left\{\tilde\BY^\by(t)>0\,~\text{for all}~ t>0\right\}=1$ for any $\by\in\Delta$,
we
 need only check its irreducibility in $\Delta^\circ$.
To prove that the diffusion $\left\{\tilde\BY(t), t\geq0\right\}$
is irreducible in $\Delta^\circ$, we pursue the following approach:
\begin{itemize}
\item First, we show that the process $\left\{\tilde\BY(t), t\geq0\right\}$
 verifies H\"ormander's condition. As a result, the process $\left\{\tilde\BY(t), t\geq0\right\}$ has a smooth density function for any $t>0$; see e.g., \cite{LB}.
\item Next, we show that there is an open set $N\subset\Delta^\circ$ such that for any open set $N_0\subset N$, and $\by\in\Delta^\circ$, there is a
$t_0>0$ such that $\PP\left\{\tilde\BY^\by(t_0)\in N_0\right\}>0$. This claim is usually proved by analyzing the control systems corresponding to the diffusion and using the support theorem. We refer to \cite{WK, LB} for more details. This then shows that the process $\left\{\tilde\BY(t), t\geq0\right\}$ is irreducible in  $\Delta^\circ$.
\end{itemize}
}
\end{rmk}

Now we consider the case $ r>0$. We still assume that $\{\tilde \BY(t): t\geq0\}$  has a unique invariant probability measure.
In order to obtain Theorem \ref{t:survival} for our degenerate process,
we have to show that there is a sufficiently large $T>0$ such that the Markov chain $(\BY(kT),S(kT))_{k\in\N}$ is irreducible and aperiodic and every compact subset of $\Delta^\circ\times(0,\infty)$ is petite for this Markov chain.
Note that if every compact subset of $\Delta^\circ\times(0,\infty)$ is petite with respect to $(\BY(kT),S(kT))_{k\in\N}$, then
any compact subset of $\Delta\times(0,\infty)$ is petite with respect to $(\BY(kT),S(kT))_{k\in\N}$
by the arguments in the proof of Lemma \ref{petite}.

Sufficient conditions for the above properties can be obtained by
verifying the well-known H\"{o}rmander condition as well as
investigating the control systems associated with the diffusion \eqref{e4.1}.
Once we have
the Markov chain $(\BY(kT),S(kT))_{k\in\N}$
being irreducible and aperiodic,  and every compact subset of $\Delta^\circ\times(0,\infty)$ being petite for sufficiently large $T$,
we can follow the steps from Section \ref{sec:+}
to obtain the following result.

\begin{thm}\label{t:survival_degenerate}
Assume that $\tilde\BY(t)$ has a \textbf{unique} invariant probability measure $\nu^*$. Define $ r$ by \eqref{lambda}. Suppose that Assumption \ref{a:dispersion} holds and that $ r>0$.
Assume further that there is a sufficiently large $T>0$ such that the Markov chain $(\BY(kT),S(kT))_{k\in\N}$ is irreducible and aperiodic, and
that every compact set in $\Delta^\circ\times(0,\infty)$ is petite for this Markov chain.

The process $\BX(t) = (X_1(t),\dots,X_n(t))_{t\geq 0}$ has a unique invariant probability measure $\pi$ on $\R^{n,\circ}_+$ that is absolutely continuous with respect to the Lebesgue measure and for any $q^*>0$,
\begin{equation}
\lim\limits_{t\to\infty} t^{q^*}\|P_\BX(t, \mathbf{x}, \cdot)-\pi(\cdot)\|_{\text{TV}}=0, \;\mathbf{x}\in\R^{n,\circ}_+,
\end{equation}
where $\|\cdot,\cdot\|_{\text{TV}}$ is the total variation norm and $P_\BX(t,\mathbf{x},\cdot)$ is the transition probability of $(\BX(t))_{t\geq 0}$. Moreover, for any initial value $\mathbf{x}\in\R^{n}_+\setminus\{\mathbf{0}\}$ and any $\pi$-integrable function $f$, we have
\begin{equation}
\PP\left\{\lim\limits_{T\to\infty}\dfrac1T\int_0^Tf\left(\BX^{\mathbf{x}}(t)\right)dt=\int_{\R_+^{n,\circ}}f(\mathbf{u})\pi(d\mathbf{u})\right\}=1.
\end{equation}
\end{thm}

\subsection{Case study: $n=2$}
In what follows,
we show that
 if $ r>0$,  there is a sufficiently large $T>0$ such that the Markov chain $(\BY(kT), S(kT))_{k\in\N}$ is irreducible and aperiodic, and that every compact set in $\Delta^\circ\times(0,\infty)$ is petite for the Markov chain.

For simplicity of presentation,
we restrict ourselves
to the $n=2$ case,
and assume that $b_i(x)=b_ix,x\geq0,  i=1,2$ for some $b_1,b_2>0$.
As a result, \eqref{e4.0} becomes
\begin{equation}\label{e5.1}
\begin{cases}
dX_1(t)=\big(X_1(t)(a_1-b_1 X_1(t))-\alpha X_1(t)+\beta X_2(t)\big)dt+\sigma_1X_1(t)dB(t)  \\
dX_2(t)=\big(X_2(t)(a_2-b_2 X_2(t))+\alpha X_1(t)-\beta X_2(t)\big)dt+\sigma_2X_2(t)dB(t),
\end{cases}
\end{equation}
where $\sigma_1$, $\sigma_2$ are non-zero constants and $(B(t))_{t\geq 0}$ is a one dimensional Brownian motion.

Setting $S(t)=X_1(t)+X_2(t)$ and $Y_i(t)=X_i(t)/S(t)$, $i=1,2$, we have from
It\^o's Lemma,
\begin{equation}\label{e5.1a}
\begin{split}
dY_i(t)=&Y_i(t)\left(a_i-\sum_{j=1}^2a_jY_j-b_iS(t)Y_i(t)+S(t)\sum_{j=1}^2b_jY_j^2(t))\right)dt+(-1)^i\left(\alpha Y_1(t)-\beta Y_2(t)\right)dt\\
&+Y_i(t)\left(\sum_{j,k=1}^2\sigma_{k}\sigma_jY_k(t)Y_j(t))-\sum_{j=1}^2\sigma_i\sigma_{j}Y_j(t)\right)dt  +(-1)^i(\sigma_2-\sigma_1)Y_1(t)Y_2(t)dB(t)\\
dS(t)=&S(t)\left(\sum_{i=1}^2(a_iY_i(t)-Y_i(t)b_iS(t)Y_i(t)\right)dt+S(t)(\sigma_1Y_1(t)+\sigma_2Y_2(t))dB(t).
\end{split}
\end{equation}
We
use the process $(Y_1(t), Y_2(t), S(t))_{t\geq0}$ to construct a Lyapunov function for a suitable skeleton
$(Y_1(kT^*), Y_2(kT^*), S(kT^*))_{k\in\N}$ as we have done in Section \ref{sec:+}.
However, to simplify the computations when verifying the hypotheses of Theorems \eqref{t:extinction_degenerate} and \eqref{t:survival_degenerate},
instead of working with $(Y_1(t), Y_2(t), S(t))$, we treat the system $(Z(t), X_2(t))$ where $Z(t):=X_1(t)/X_2(t)$. An application of It\^o's Lemma yields
\begin{equation}\label{e5.2}
\begin{split}
dZ(t)&=\Big((b_2-b_1Z(t))Z(t)X_2(t)+\beta+\hat a_1 Z(t)-\alpha Z^2(t)\Big)dt+Z(t)[\sigma_1-\sigma_2]dB(t)  \\
dX_2(t)&=X_2(t)\Big((\hat a_2-b_2 X_2(t))+\alpha Z(t)\Big)dt+\sigma_2X_2(t)dB(t),
\end{split}
\end{equation}
where
$\hat a_1=a_1-a_2-\alpha+\beta+\sigma^2_2-\sigma_1\sigma_2$
and
$\hat a_2=a_2-\beta$.

To proceed,
we first convert \eqref{e5.2} to Stratonovich form to facilitate the  verification of H\"{o}rmander's condition.
System \eqref{e5.2} can be rewritten as
\begin{equation}\label{e5.5}
\begin{split}
dZ(t)=&\Big((b_2-b_1Z(t))Z(t)X_2(t)+\beta+\left(\hat a_1- \dfrac{(\sigma_1-\sigma_2)^2}2\right)Z(t)-\alpha Z^2(t)\Big)dt\\
&+Z(t)[\sigma_1-\sigma_2]\circ dB(t)  \\
dX_2(t)=&X_2(t)\left(\left(\hat a_2-\dfrac{\sigma_2^2}2-b_2 X_2(t)\right)+\alpha Z(t)\right)dt+\sigma_2X_2(t)\circ dB(t).
\end{split}
\end{equation}
Let
$$A_0(z,y)=\begin{pmatrix}
 (b_2-b_1z)zy+\beta+\left(\hat a_1- \dfrac{(\sigma_1-\sigma_2)^2}2\right)z-\alpha z^2\\
  y\left(\hat a_2-\dfrac{\sigma_2^2}2-b_2 y\right)+\alpha zy
\end{pmatrix},
$$
and
$$A_1(z,y)=\begin{pmatrix}
 (\sigma_1-\sigma_2)z\\
  \sigma_2y
\end{pmatrix}.$$
Recall that
the diffusion \eqref{e5.5} is said to satisfy \textit{H\"{o}rmander's condition} if the set
 of vector fields $A_1,$ $[A_1, A_0],$ $[A_1,[A_1,A_0]],$ $[A_0,[A_1, A_0]],$ $\dots$ spans $\R^2$ at every $(z,y)\in \R^{2,\circ}_+$, where
$[\cdot,\cdot]$ is the Lie bracket,
which
is defined as follows (see \cite{LB} for more details).
If $\Phi(z,y) = (\Phi_1(z,y),\Phi_2(z,y))^\top$ and $\Psi(z,y) = (\Psi_1(z,y),\Psi_2(z,y))^\top$ are vector fields on $\R^2$
(where $z^\top$ denotes the transpose of $z$), then the Lie bracket
$[\Phi, \Psi]$ is a vector field given by
\begin{align*}
[\Phi,\Psi]_j(z,y)&=\left(\Phi_1(z,y)\dfrac{\partial \Psi_j}{\partial z}(z,y)-\Psi_1(z,y)\dfrac{\partial \Phi_j}{\partial z}(z,y)\right)
\\
&\qquad
+\left(\Phi_2(z,y)\dfrac{\partial \Psi_j}{\partial y}(z,y)-\Psi_2(z,y)\dfrac{\partial \Phi_j}{\partial y}(z,y)\right),\;\; j=1,2.
\end{align*}
\begin{prop}\label{p:hormander}
Suppose that $\sigma_1\neq \sigma_2$ or $\beta+(b_2/b_1)( a_1-a_2)-\alpha (b_2/b_1)^2\ne 0$. Then H\"{o}rmander's condition holds for the diffusion $(Z(t),X_2(t))_{t\geq 0}$ given by \eqref{e5.5}. As a result, the transition probability $P(t, (z, y), \cdot)$ of $(Z(t),X_2(t))_{t\geq 0}$ has a smooth density $\R_+\times \R_+^4\ni(t, z, y, z', y')\mapsto p(t, z, y, z', y')\in\R_+$ with respect to Lebesgue measure.
\end{prop}

\begin{proof}
Set $\sigma:=\dfrac{\sigma_1-\sigma_2}{\sigma_2}$. By a direct calculation,
$$A_2(z,y):=\dfrac{1}{\sigma_2}[A_0, A_1](z,y)=\begin{pmatrix}
  \sigma(\beta+\alpha z^2)+(\sigma+1)b_1z^2y-zyb_2\\
   -\sigma\alpha zy+b_2y^2
\end{pmatrix},
$$
and for $k>2$, we have
$$A_k(z,y):=\dfrac{1}{\sigma_2}[A_1, A_{k-1}](z,y)=\begin{pmatrix}
  \sigma^{k-1}(\beta+(-1)^k\alpha z^2)+(-1)^k(\sigma+1)^2b_1z^2y+(-1)^{k+1}zyb_2\\
  (-1)^{k+1}\sigma^2\alpha zy+(-1)^kb_2y^2.
\end{pmatrix}.
$$
If $\sigma\ne0$ or equivalently $\sigma_1\ne\sigma_2$, a straightforward but tedious computation shows that the rank of the matrix with columns $A_1, A_2, A_3, A_4$ is always 2 for any $(z,y)\in\R^{2,\circ}_+$.
As a result, if $\sigma_1\ne\sigma_2$, H\"{o}rmander's condition is satisfied for the diffusion \eqref{e5.5}.
Therefore, the transition probability $P(t, (z, y), \cdot)$ of $(Z(t), X_2(t))$ has a smooth density function, denoted by $p(t, z, y, z', y')$; see \cite[Corollary 7.2]{LB}.

Now, we show that H\"{o}rmander's condition holds if $\sigma_1=\sigma_2$ and $\beta+(b_2/b_1)( a_1-a_2-\alpha+\beta)-\alpha (b_2/b_1)^2\ne 0$.
In this case,
$$A_2(z,y)=[A_0, A_1](z,y)=\begin{pmatrix}
-\alpha yz(b_2-b_1z)\\ \alpha b_2y^2
\end{pmatrix},
$$ and
$$C(z,y)=\begin{pmatrix}
C_1(z,y)\\
C_2(z,y)
\end{pmatrix}:=\left[A_0, \frac1{\alpha b_2}A_2\right](z,y),
$$
where
\begin{align*}
C_1(z,y)=&y(2b_1z/b_2-1)A_{0,1}(z,y)+yz(1-zb_1/b_2)\frac{\partial A_{0,1}(z,y)}{\partial z}
\\&+z(zb_1/b_2-1)A_{0,2}(z,y)+y^2z(zb_1/b_2-1).
\end{align*}
With $A_{0,i}(z,y)$ denoting the $i$-th component of $A_0(z,y).$
Observe that $A_1(x,y), A_2(z, y)$ span $\R^2$ for any $(z,y)\in \R^{2,\circ}_+$ satisfying $z\ne b_2/b_1.$
If $z= b_2/b_1$ we have $C_1( b_2/b_1,y)=yA_{0,1}(b_2/b_1,y)=y\big[\beta+(b_2/b_1)( a_1-a_2-\alpha+\beta)-\alpha (b_2/b_1)^2\big]\ne 0$
hence  $C(b_2/b_1,y)$ and $ A_2(b_2/b_1,y)$ span $\R^2$ for all $y>0$.
As a result, we obtain the desired result.

\end{proof}

To proceed, we consider the following control system, which is associated with \eqref{e5.5}.
\begin{equation}\label{e5.6}
\begin{cases}
dz_\phi(t)=&\!\!\!\! (b_2-b_1z_\phi(t))z_\phi(t) y_\phi(t)+\beta+\left(\hat a_1- \dfrac{(\sigma_1-\sigma_2)^2}2\right)z_\phi(t)-\alpha z^2_\phi(t)+(\sigma_1-\sigma_2)z_\phi\phi(t)\\
dy_\phi(t)=&\!\!\!\! y_\phi(t)\left(\hat a_2-\dfrac{\sigma_2^2}2-b_2 y_\phi(t)\right)+\alpha z_\phi(t) y_\phi(t)+\sigma_2 y_\phi(t)\phi(t)
\end{cases}
\end{equation}
Let $(z_\phi(t, z,y),$ $ y_\phi(t, z, y))$ be the solution to Equation \eqref{e5.6} with control $\phi$ and initial value $(z,y)$.
Denote by ${\mathcal O}_1^+(z, y)$ the reachable set from $(z, y)$, that is the set of $(z', y')\in\R^{2,\circ}_+$ such that there exists a $t\geq0$ and a control $\phi(\cdot)$ satisfying
$z_\phi(t, z, y)=z', y_\phi(t, z, y)=z'$.
We first recall some concepts introduced in  \cite{WK}.
Let $U$ be a subset of $\R^{2,\circ}_+$ satisfying
$u_2\in \bar{{\mathcal O}^+_1(u_1)}$ for any $u_1, u_2\in U$.
Then there is a unique maximal set $V\supset U$ such that this property still holds for $V$. Such $V$ is called a \textit{control set}.
A control set $C$ is said to be \textit{invariant} if $\bar{{\mathcal O}^+_1(w)}\subset\bar C$ for all $w\in C$.

Finding invariant control sets for \eqref{e5.6} is facilitated by
using a  change of variables argument.
Put $w_\phi(t)=z_\phi(t)y^{r+1}_\phi(t)$ with $r=\frac{-\sigma_1}{\sigma_2}$.
We have
\begin{equation}\label{e5.7}
\begin{cases}
dw_\phi(t)=&h(w_\phi(t),y_\phi(t))dt\\
dy_\phi(t)=&y_\phi(t)\left(\hat a_2-\dfrac{\sigma_2^2}2-b_2 y_\phi(t)\right)+\alpha w_\phi(t) y_\phi^{-r}(t)+\sigma_2 y_\phi(t)\phi(t)
,\end{cases}
\end{equation}
where
$$h(w, y)=w\left(a_1-\frac{\sigma_1^2}2+r(a_2-\frac{\sigma_2^2}2)+r\beta-\alpha-b_1w y^r-b_2ry+\beta y^{1-r}w^{-1}+\alpha rw y^{r-1}\right).$$
Denote by ${\mathcal O}^+_2(w, y)$ the set of $(w', y')\in\R^{2,\circ}_+$ such that there is a $t>0$ and a control $\phi(\cdot)$ such that
$w_\phi(t, w, y)=w', z_\phi(t, w, y)=w'$.
\begin{lem}
The control system \eqref{e5.7} has only one invariant control set $\tilde\C$ and
$\bar{{\mathcal O}^+_2(w, y)}\supset\tilde\C$ for any $(w,y)\in\R^{2,\circ}_+$,
The set $\tilde\C$ is defined by
$\tilde\C=\{(w, y)\in \R^{2,\circ}_+: w<c^*\}$,
where
$$c^*=\sup\left\{w: \sup\limits_{y>0}\{h(w',y)\}\geq 0\,~\text{for all}~ w'<w\right\}.$$
Consequently,
the control system \eqref{e5.6} has only one invariant control set $\C$ and
$\bar{{\mathcal O}^+_1(z, y)}\supset\C$ for any $(w,y)\in\R^{2,\circ}_+$,
where $\C:=\{(z,y)\in\R^{2,\circ}_+: zy^{r+1}\leq c^*\}$.
Moreover, by \cite[Lemma 4.1]{WK}, $(Z(t), X_2(t))$ has at most one invariant probability measure whose support is $\C$.
\end{lem}

\begin{proof}
First, we need to show that $c^*$ is well-defined (although it can be $+\infty$).
Since
$\lim\limits_{w\to0}h(w,y)=\infty$, which implies
that $\left\{w: \sup\limits_{y>0}\{h(w',y)\}\geq 0\,~\text{for all}~ w'\leq w\right\}$ is a nonempty set.
Hence $c^*$ is well-defined.
The claim that $\bar{{\mathcal O}^+_2(w, y)}\supset\tilde\C$ for any $(w,y)\in\R^{2,\circ}_+$
can be proved by standard arguments. Let us explain the main ideas here.
On the phase space $(w,y)\in\R^{2,\circ}_+$,
since the control $\phi(t)$ only appears in the equation of $y_\phi$,
we can easily control vertically, that is,
for any initial points $y_0$ and $w_0$,  there is a control so that
$y_\phi$ can reach any given point $y_1$ while $w_\phi$ stays in a given neighborhood of $w_0$.
If $h(w_0, y_0)<0$, we can choose a feedback control such that
$(w_\phi(t),u_\phi(t))$ reaches a point to the `left' $(w_1, y_0)$ with $w_1<w_0$ as long as $h(w, y_0)<0$ for $w\in[w_1, w_0]$.
Likewise, for
$h(w_0, y_0)>0$, we can choose a feedback control such that
$(w_\phi(t),u_\phi(t))$ can reach a point to the `right' $(w_1, y_0)$ with $w_1>w_0$ as long as $h(w, y_0)>0$ for $w\in[w_0, w_1]$.
We also have that
$\inf\limits_{y>0}\{h(w,y)\}=-\infty$ for any $w>0$.
Using these facts, we can follow the steps from \cite[Section 3]{DDY}
to obtain the desired results.
\end{proof}

\begin{lem}\label{lem5.2}
There is a point $(z^*, y^*)\in\C$ such that for any open set $	N^*\ni (z^*,y^*)$ and $T>0$, there is an open neighborhood $W^*\ni(z^*,y^*)$ and a control $\phi^*$ such that
$$(z_{\phi^*}(t, z, y),y_{\phi^*}(t, z, y))\in N^*\,~\text{for all}~ (z,y)\in W^*, t\in[0,T].$$
\end{lem}

\begin{proof}
To obtain the result, we  work on \eqref{e5.7}, which is equivalent to \eqref{e5.6}.
By the definition of $\tilde\C$ and the fact that
$\lim\limits_{y\to\infty}h(w,y)=-\infty$ if $r>0$ and
$\lim\limits_{y\to0}h(w,y)=-\infty$ if $r<0$,
there is a point $(w^*, y^*)\in \tilde\C$ such that
$h(w^*,y^*)=0$.
We can design a feedback control ${\phi^*}$ such that
\begin{equation}\label{e5.8}
\begin{cases}
dw_{\phi^*}(t)=&h(w_{\phi^*}(t),y^*)dt\\
dy_{\phi^*}(t)=&0.
\end{cases}
\end{equation}
If $w_{\phi^*}(t)=w^*$ then $w_{\phi^*}(t)=w^*\,~\text{for all}~ t>0$.
 By the continuous dependence on initial data of solutions to differential equations,
for any given neighborhood $\tilde N^*$ of $(w^*, y^*)$, we can find a neighborhood $\tilde W^*$
of $(w^*, y^*)$ such that
$(w_{\phi^*}(t, w, y), y_{\phi^*}(t, w, y))\in\tilde N^*$ for any $t\in[0,T]$ and $(w,y)\in\tilde W^*$,
which proves the lemma.
  \end{proof}

\begin{prop}\label{prop5.1}
Suppose $\sigma_1\neq \sigma_2$ or $\beta+(b_2/b_1)( a_1-a_2)-\alpha (b_2/b_1)^2\ne 0$.
For any $T>0$, every compact set $K\subset\R^{2,\circ}_+$ is petite set with respect to
the Markov chain $(Z(kT), X_2(kT))_{k\in\N}$.
\end{prop}

\begin{proof}
Let $(z^*, y^*)$ be as in Lemma \ref{lem5.2}.
Pick $(z^\diamond,y^\diamond)\in\R^{2,\circ}_+$ such that
$p(T, z^*, y^*, z^\diamond,y^\diamond)>0$.
By the smoothness of  $p(T,\cdot,\cdot,\cdot,\cdot)$, there exists  a  neighborhood $N^*$ and  an  open set $N^\diamond\ni(z^\diamond,y^\diamond)$ such that
\begin{equation}\label{bs1}
p(1, z,y,z',y')\geq p^\diamond>0  \;~\text{for all}~\, (z,y)\in N^*, (z',y')\in N^\diamond.
\end{equation}
Let $W^*$ be a neighborhood of $(z^*, y^*)$ satisfying
\begin{equation}\label{e5.9}
(z_{\phi^*}(t, z, y),y_{\phi^*}(t, z, y))\in N^*\,~\text{for all}~ (z,y)\in W^*, t\in[0,T].
\end{equation}
For each $(z, y)\in\R^{2,\circ}_+$, noting that $(z^*,y^*)\in \C\subset\bar{{\mathcal O}^+_1(z, y)}$, there is a control $\phi$ and $t_{z,y}>0$ such that
\begin{equation}\label{e5.10}
(z_\phi(t_{z,y}, z, y), y_\phi(t_{z,y}, z, y))\in W^*.
\end{equation}
Let $n_{z,y}\in\N$ such that $(n_{z,y}-1)T< t_{z,y}\leq n_{n,y}T$ and
 $\tilde\phi$ be defined as $\tilde\phi(t)=\phi(t)$ if $t<t_{z,y}$ and $\tilde\phi(t)=\phi^*(t)$ if $t>t_{z,y}$.
 Using the control $\tilde\phi$, we obtain from \eqref{e5.9} and \eqref{e5.10} that
\begin{equation}\label{e5.11}
\left(z_{\tilde\phi}(n_{z,y}T, z, y), y_{\tilde\phi}(n_{z,y}T, z, y)\right)\in N^*.
\end{equation}
In view of the support theorem (see \cite[Theorem 8.1, p. 518]{IW}),
$$P(n_{z,y}T,z,y, N^*):=2\rho_{z,y}>0.$$
Since $(Z_{z,y}(t), Y_{z,y}(t))$ is a Markov-Feller process,  there exists an open set $V_{z,y}\ni(z,y)$ such that
\newline $P(n_{z,y}T,z',y', N^*)\geq \rho_{u,v} \;~\text{for all}~ (z',y')\in V_{z,y}.$
Since $K$ is a compact set,   there is a finite number of $V_{z_i,y_i}, \; i=1, \ldots, k_0$ satisfying $K\subset\bigcup_{i=1}^{k_0}V_{z_i, y_i}.$
Let $\rho_K=\min\{\rho_{z_i,y_i}, \; i=1, \ldots, k_0\}.$
For each $(z,y)\in K$, there exists $n_{z_i,y_i}$ such that
\begin{equation}\label{bs2}
  P(n_{z_i,y_i}T, z,y, N^*)\geq\rho_K.
\end{equation}
From \eqref{bs1} and \eqref{bs2}, for all $(z,y)\in K$, there exists $n_{z_i,y_i}$ such that
\begin{equation}\label{bs3}
  p((n_{z_i,y_i}+1)T,z,y,z',y')\geq \rho_Kp^\diamond \;~\text{for all}~\, (z',y')\in N^\diamond.
\end{equation}
It follows from \eqref{bs3} that
\begin{equation}\label{bs4}
  \frac1{k_0}\sum_{i=1}^{k_0}P((n_{z_i,y_i}+1)T,z,y, A) \geq \frac1{k_0}\rho_Kp^\diamond m(N^\diamond\cap A)  \;~\text{for all}~ A\in\mathcal B(\R^{2,\circ}_+),
\end{equation}
where $m(\cdot)$ is the Lebesgue measure on $\R^{2,\circ}_+.$
Equation \eqref{bs4} implies that every compact set $K\subset \R^{2,\circ}_+$  is petite for the Markov chain $(Z(kT), X_2(kT))_{k\in\N}.$
\end{proof}
We have shown in the beginning of subsection 2.2. that $\tilde\BY(t)$ has a \textbf{unique} invariant probability measure $\nu^*$.
Having Proposition \ref{prop5.1}, we note that the assumptions, and therefore the conclusions, of Theorems \ref{t:extinction_degenerate} and \ref{t:survival_degenerate} hold for model \eqref{e5.1}. This argument proves Theorems \ref{t:extinction_degenerate_2} and \ref{t:survival_degenerate_2}.

\section{Robustness of the model}\label{s:robust}
The robustness is studied from several angles, including continuous dependence of
$ r$ on the coefficients of the stochastic differential equation,
robustness of persistence, and
robust attenuation against extinction.
They are presented in a couple subsections.

\subsection{Continuous dependence of $ r$ on the coefficients}
We show that $ r$ depends continuously on the coefficients of the stochastic differential equation \eqref{eq.by}. Consider the equation
\begin{equation}\label{eq.bhy}
\begin{split}
d\hat\BY(t)=&\left(\diag(\hat\BY(t))-\hat\BY(t)\hat\BY^\top(t)\right)\hat\Gamma^\top d\BB(t)\\
&+\hat\BD^\top\hat\BY(t)dt+\left(\diag(\hat\BY(t))-\hat\BY(t)\hat\BY^\top(t)\right)(\hat\ba-\hat\Sigma \hat\BY(t))dt
\end{split}
\end{equation}
on the simplex $\Delta$.
Suppose that $\hat\Sigma$ is positive definite. In this case,
$(\hat\BY(t))_{t\geq0}$ has a unique invariant probability measure $\hat\nu^*$. Define
\begin{equation}\label{hatlambda}
\hat r:=\int_{\Delta}\left(\hat\ba^\top{\bf y}-\frac12{\bf y}^\top\hat\Sigma{\bf y}\right)\hat\nu^*(d{\bf y})
\end{equation}
Fix the coefficients of \ref{eq.by}.

\begin{pron}\label{p:cts_dep}
For any $\eps>0$, there is a $\theta_2>0$ such that if
$$\max\left\{\|\ba-\hat \ba\|, \|D-\hat D\|, \|\Gamma-\hat\Gamma\|\right\}<\theta_2$$
then
$$| r-\hat r|<\eps.$$
\end{pron}

\begin{proof}
First, let $\theta_1>0$ such that if $\max\left\{\|\ba-\hat \ba\|, \|D-\hat D\|, \|\Gamma-\hat\Gamma\|\right\}<\theta_1$, then
\begin{equation}\label{cd.e1}
\left|\left(\hat\ba^\top{\bf y}-\frac12{\bf y}^\top\hat\Sigma{\bf y}\right)-\left(\ba^\top{\bf y}-\frac12{\bf y}^\top\Sigma{\bf y}\right)\right|<\dfrac{\eps}3\text{ for all }\by\in\Delta.
\end{equation}
Let $\gamma_1$, $\gamma_2, M_3, M_4$ be defined as in the proof of Lemma \ref{lem2.5}.
Pick $T=T(\eps)>0$ such that
\begin{equation}\label{cd.e0}
\|\tilde P(T, \by,\cdot)-\nu^*\|_{TV}\leq \gamma_1\exp(-\gamma_2T)<\frac{\eps}{3M_4}\,\text{ for all } \by\in\Delta.
\end{equation}
By standard arguments, there is a $\theta_2\in(0,\theta_1)$ such that
if  $\max\left\{\|\ba-\hat \ba\|, \|D-\hat D\|, \|\Gamma-\hat\Gamma\|\right\}<\delta_2$,
then
\begin{equation}\label{cd.e2}
\PP\left\{\|\tilde \BY^\by(T)-\hat \BY^\by(T)\|<\dfrac{\eps}{6M_3}\right\}>\dfrac{\eps}{6M_4}\text{ for all }\by\in\Delta
\end{equation}
Let $\by^*$ be a $\Delta$-valued and $\F_0$-measurable random variable whose distribution is $\hat\nu^*$.
Clearly,
\begin{equation}\label{cd.e3}
\int_{\Delta}\left(\ba^\top{\bf y}-\frac12{\bf y}^\top\Sigma{\bf y}\right)\hat\nu^*(d{\bf y})=\E\left( \ba^\top\hat\BY^{\by^*}(T)-\frac12(\hat\BY^{\by^*}(T))^\top\Sigma\hat\BY^{\by^*}(T)\right).
\end{equation}
In view of \eqref{cd.e0}, \begin{equation}\label{cd.e4}
\left|\E\left( \ba^\top\tilde\BY^{\by^*}(T)-\frac12(\tilde\BY^{\by^*}(T))^\top\Sigma\tilde\BY^{\by^*}(T)\right)- r\right|\leq M_4\sup_{\by\in\Delta}\{\|\tilde P(t,\by,\cdot)-\mu^*\|\}\leq\dfrac{\eps}3.
\end{equation}
It follows from \eqref{cd.e2} that
\begin{equation}\label{cd.e5}
\begin{aligned}
\E&\left|\ba^\top\hat\BY^{\by^*}(T)-\frac12(\hat\BY^{\by^*}(T))^\top
\Sigma\hat\BY^{\by^*}(T)-\ba^\top\tilde\BY^{\by^*}(T)+\frac12(\tilde\BY^{\by^*}(T))^\top\Sigma\tilde\BY^{\by^*}(T)\right|\\
\,&\leq M_3\dfrac{\eps}{6M_3}\PP\left\{\|\tilde \BY^{\by^*}-\hat \BY^{\by^*}\|<\dfrac{\eps}{6M_3}\right\}+M_4\PP\left\{\|\tilde \BY^{\by^*}-\hat \BY^{\by^*}\|\geq\dfrac{\eps}{6M_3}\right\}\leq\dfrac{\eps}3.
\end{aligned}
\end{equation}
In view  of \eqref{hatlambda}, \eqref{cd.e1}, \eqref{cd.e3}, \eqref{cd.e4}, and \eqref{cd.e5},
if
$$\max\left\{\|\ba-\hat \ba\|, \|D-\hat D\|, \|\Gamma-\hat\Gamma\|\right\}<\theta_2$$ then $| r-\hat r|<\eps$, which
completes the proof.
\end{proof}
\begin{rmk}
{\rm The continuous dependence of $r$ on the coefficients
can also be proved by generalizing the arguments from the proof of \cite[Proposition 3]{ERSS13}.
Since \cite[Proposition 3]{ERSS13} focuses only on the continuity for a specific parameter rather than
all parameters, we provided an alternative proof for the sake of completeness.
}
\end{rmk}
\subsection{Robust persistence and extinction}

\begin{proof}[Sketch of proof of Theorem \ref{probust2}]
As usual, we work with
\begin{equation}\label{e.hatbys}
\begin{split}
d\hat\BY(t)=&\left(\diag(\hat\BY(t))-\hat\BY(t)\hat\BY^\top(t)\right)\hat\Gamma^\top(\hat S(t)\hat\BY(t)) d\BB(t)+\hat\BD(\hat S(t)\hat\BY(t))^\top\hat\BY(t)dt\\
&+\left(\diag(\hat\BY(t))-\hat\BY(t)\hat\BY^\top(t)\right)(\hat\ba-\hat\Sigma(\hat S(t)\hat\BY(t)) \hat\BY(t)-\hat\bb(\hat S(t)\hat\BY(t)))dt\\
d\hat S(t)=&\hat S(t)\left[\hat \ba-{\hat\bb(\hat S(t)\hat\BY(t))}\right]^\top\hat\BY(t)dt+\hat S(t){\hat\BY(t)}^\top \hat\Gamma^\top(\hat S(t)\hat\BY(t)) d\BB(t),
\end{split}
\end{equation}
where $\hat S(t):=\sum_{i}\hat X_i(t)$, $\hat\BY(t):=\frac{\hat\BX(t)}{\hat S(t)}$.
In order to have a complete proof for this proposition one can follow the steps from Section \ref{sec:+}.
First, since $\Sigma$ is positive definite
then so is $\hat\Sigma(\bx):=\hat\Gamma(\bx)^\top\hat\Gamma(\bx)$ if
$\sup_{\bx\in\R^{n,\circ}_+}\|\hat\Gamma(\bx)-\Gamma\|$ is sufficiently small.  As a result,
$(\hat\BX(t))_{t\geq0}$ is a nondegenerate diffusion in $\R^{n,\circ}_+$
and
Lemma \ref{petite} holds for $(\hat\BY(nT),\hat S(nT))_{n\in\N}$.
We also have the following results:
 there exist positive constants $\hat K_i: i=1,\dots,4$, which do not depend on $\theta$ as long as $\theta$ is sufficiently small, such that
\begin{equation}\label{eh2.1}
   \E \hat S^{\by, s}(t) \leq e^{-\gamma_b t/2}s
  +\hat K_1, \,\, (\by,s)\in \Delta\times(0,\infty), t\geq0.
\end{equation}
  \begin{equation}\label{eh2.2}
    \E \left([\ln \hat S^{\by,s}(T)]^2\right)\leq ((\ln s)^2+1)\hat K_2\exp\{\hat K_3T\}, \, (\by, s)\in\Delta\times(0,\infty), T\geq0,
  \end{equation}
  and
  \begin{equation}\label{eh2.3}
    \E \left([\ln \hat S^{\by,s}(T\wedge \hat\zeta^{\by, s})]_-^2\right)\leq (\ln s)^2+\hat K_4\sqrt{\PP(A)}(T+1)[\ln s]_-+\hat K_4T^2
  \end{equation}
for all $ (\by, s)\in\Delta\times(0,1), A\in\F$ where
\[
\hat \zeta^{\by, s}:=\inf\{t\geq0:  \hat S^{\by,s}(t)=1\}.
\]
  On the other hand, standard arguments show that
  for any $\eps>0$, $T>0$, there is a $\theta=\theta(\eps, T)>0$ such that
$$\PP\left\{\left\|(\BY^{\by,s}(t), S^{\by, s}(t))-(\hat \BY^{\by, s}(t),\hat S^{\by, s}(t)) \right\|\leq\eps, \, 0\leq t\leq T\right\}>1-\eps$$
 given that
$(\by,s)\in\Delta\times[0,1].$
Combining this fact with Proposition \ref{prop3.2}, one can find $\delta=\delta(\eps,T)>0$ and $\theta=\theta(\eps,T)>0$ such that
$$\PP\left\{\left\|(\tilde\BY^{\by,s}(t), 0)-(\hat \BY^{\by, s}(t),\hat S^{\by, s}(t)) \right\|\leq\eps, \, 0\leq t\leq T\right\}>1-\eps$$
given that
$(\by,s)\in\Delta\times(0,\delta)$
and \eqref{robust} holds.
With this fact, we can use Lemma \ref{lem2.5} with slight modification to show that,
for any $\eps>0$, there is a $T^*=T^*(\eps)$ and
$\delta=\delta(\eps,T^*),\theta=\theta(\eps,T^*)$ such that
  \begin{equation}\label{eh2.4}
    \PP\left\{\ln s+\frac{3 r T^*}{4}\leq \ln \hat S^{\by,s}(T^*)<0\right\}\geq 1-3\eps\text{ for all } (\by, s)\in\Delta\times(0,\delta)
  \end{equation}
  given that
  \eqref{robust} holds.
Having \eqref{eh2.1}, \eqref{eh2.2}, \eqref{eh2.3}, and \eqref{eh2.4}, we can use the
arguments from Proposition \ref{prop2.3} and Theorem \ref{thm2.2} to finish the proof.
\end{proof}

 \begin{rmk}
{\rm If $ r<0$, $\BX(t)$ converges to $\bf0$ with probability 1.
By virtue of Proposition \ref{p:cts_dep},
 if $\hat D,\hat\Gamma$ are constant matrices and
 $\max\left\{\|\ba-\hat \ba\|, \|D-\hat D\|, \|\Gamma-\hat\Gamma\|\right\}$ is sufficiently small
 then
 $\hat\BX(t)$ converges to $\bf0$ with an exponential rate almost surely.
 We conjecture that this result
 holds for any $\theta$-perturbation of $\BX(t)$ defined by \eqref{robust}.
 However, when $\hat D:=\hat D(\bx),\hat\Gamma:=\hat\Gamma(\bx)$,
 comparison arguments may be
 not
 applicable.
 Moreover, it is also difficult to analyze the asymptotic behavior of the equation without competition terms, namely
 \begin{equation}
d\mathcal {\hat X}(t)=\left(\diag(\mathcal {\hat X}(t))\hat\ba+ {\hat D(\mathcal{\hat X}(t))}^\top\mathcal{\hat X}(t)\right)dt+\diag(\mathcal {\hat X}(t)) {\hat\Gamma(\mathcal{\hat X}(t))}^\top d\BB(t).
\end{equation}

 }\end{rmk}

\end{document}